\documentclass{amsart}
\usepackage{amssymb,amsmath,amsthm}
\usepackage[all]{xy}
\usepackage{comment}
\usepackage{tabls}
\usepackage{lscape}
\usepackage[dvips]{graphicx}

\newcommand{\C}{{\mathbb {C}}}
\newcommand{\R}{{\mathbb {R}}}
\newcommand{\Q}{{\mathbb {Q}}}
\newcommand{\Z}{{\mathbb {Z}}}
\newcommand{\F}{{\mathcal {F}}}
\newcommand{\G}{{\mathcal {G}}}
\newcommand{\B}{{\mathcal {B}}}
\newcommand{\suchthat}{\mathrel{|}}
\newcommand{\arrow} [1] [\gamma]{\overset{#1}{\to}}

\makeatletter
 
 \@addtoreset{equation}{section}
\makeatother

\newtheorem{thm}{Theorem}[section]
\newtheorem{lem}[thm]{Lemma}
\newtheorem{prop}[thm]{Proposition}
\newtheorem{cor}[thm]{Corollary}

\theoremstyle{definition}

\newtheorem{ex}{Example}

\theoremstyle{remark}
\newtheorem{rem}{Remark}

\theoremstyle{nonumber}
\newtheorem{claim}{Claim}

\title{The $T$-equivariant Integral Cohomology Ring of $E_6/T$}
\date{\today}
\author{Takashi Sato}
\address{Department of Mathematics,
Graduate School of Science,
Kyoto University, Sakyo-ku Kyoto-shi, Kyoto, 606-8502, Japan.}
\email{t-sato@math.kyoto-u.ac.jp}


\begin{document}

\maketitle

\begin{abstract}
We prove the equivariant Leray-Hirsch theorem combinatorially
for sufficiently good torus equivariant fiber bundles
consisting of homogeneous spaces of Lie groups.
We apply this theorem to determining the equivariant integral
cohomology ring of the flag manifold of type $E_6$
and express it explicitly as a quotient ring of a polynomial ring.
\end{abstract}

\section{\textsc{Introduction}}
Let $G$ be a compact connected Lie group and
$T$ its maximal torus.
The homogeneous space $G/T$ is called a flag manifold
and it plays an important role in topology,
algebraic geometry, representation theory, and combinatorics.
The maximal torus $T$ acts on $G/T$ by the left multiplication
and the equivariant cohomology ring $H^*_T(G/T) = H^*(ET \times_T G/T)$
is very important.

The GKM theory,
which was established by Goresky, Kottwitz, and MacPherson \cite{GKM},
gives us a powerful method to compute equivariant cohomology rings.
In particular, the equivariant cohomology of $G/T$ can be computed as follows.
The fixed point set $(G/T)^T$ is the Weyl group $W(G)$
and
the inclusion $i \colon (G/T)^T \to G/T$ induces
a ring homomorphism
\[
i^* \colon H^*_T(G/T) \to H^*_T((G/T)^T) = \prod_{W(G)} H^*(BT) = {\rm Map}(W(G),H^*(BT)).
\]
The localization theorem claims that
$i^*$ tensored with $\C$ is injective.
Hence we need to compute its image.
Guillemin and Zara \cite{GZ} restated the GKM theory
and claimed that its image is completely determined by
a special graph obtained from the subspace of $G/T$
consisting all orbits of dimension at most one.
This special graph is called the GKM graph of $G/T$.
We can compute the equivariant cohomology by these combinatorial data.
Harada, Henriques, and Holm \cite{HHH} showed that $i^*$ is injective
with integer coefficient under some assumption
and that the equivariant integral cohomology ring is also determined by
the GKM graph.

Let $G_\C$ be the complexification of $G$
and $B$ a Borel subgroup containing $T$.
It is well-known that $G/T$ and $G_\C/B$ is $T$-equivariant isomorphic.
When $P$ is a parabolic subgroup containing $B$,
the GKM graph of $P/B$ is a subgraph of the GKM graph of $G/T \cong G_\C/B$.
The GKM graph of $G_\C/B$ can be understood roughly
as the union of slightly modified copies of the subgraph.
This viewpoint ignores the structure of the GKM graph of $G_\C/P$
introduced by the natural projection $G_\C/B \to G_\C/P$.
Hence, if we know the equivariant cohomology of $P/B$ and $G_\C/P$,
we can determine $H^*_T(G_\C/B)$ combinatorially.
This statement is formulated as Theorem \ref{generators} and \ref{relations},
and by combining them we obtain the equivariant Leray-Hirsch theorem.

The ordinary integral cohomology ring of $E_6/T$
is determined by Toda and Watanabe \cite{TW}.
The fiber bundle 
\[
G/T \to ET \times_T G/T \to BT
\]
shows that the ordinary cohomology is easily obtained from the equivariant one.
Indeed, the results of this paper covers their results for $E_6/T$.

\section{\textsc{Abstract GKM graphs}}

In this section, we discuss graphs whose edges are equipped with ideals of some ring
and the ``cohomology'' rings of such graphs.

\subsection{Definitions}
Guillemin and Zara restated the GKM theory
and claim that the cohomology of a manifold with some good torus action
is determined by a graph obtained from the fixed points and invariant $2$-spheres of such manifold,
which is called a GKM graph (cf. \cite[Section 1]{GZ}).
In this subsection, we generalize their definition of GKM graphs in purely combinatorial terms.

Let $G=(V,E)$ be a finite graph, possibly with multiple edges and loops,
where $V$ and $E$ are the vertex and the edge sets, respectively,
and let $\alpha \colon E \to \mathcal{I}_R$ be a set function to the set $\mathcal{I}_R$ of all ideals of a given ring $R$.
We call the pair $(G, \alpha)$ a \textit{GKM graph} and the function $\alpha$ an \textit{axial function}.
For an edge $e$, we call $\alpha(e)$ the \textit{label} of $e$.

Next we define homomorphisms between GKM graphs.
Let $(G,\alpha)$ and $(G',\alpha')$ be GKM graphs
such that the targets of $\alpha$ and $\alpha'$ are $\mathcal{I}_R$ and $\mathcal{I}_{R'}$, respectively.
A homomorphism between GKM graphs consists of
a graph homomorphism $\phi \colon G \to G'$
and a ring homomorphism $\psi \colon R' \to R$
such that
\[
\psi(\alpha'(\phi(e))) \subset \alpha(e)
\]
for any edge $e$ of $G$.
If $G=G'$, then any ring homomorphism $\psi \colon R' \to R$
satisfying the above condition with $\phi = \mathrm{id}$ induces
a homomorphism $(G,\alpha) \to (G',\alpha')$,
which we denote by the same symbol $\psi$.

Finally we define the ``cohomology'' of a GKM graph.
Let $(G,\alpha)$ be a GKM graph with $G=(V,E)$
and $\alpha \colon E \to \mathcal{I}_R$.
A set function $f \colon V \to R$ is called a \textit{GKM function} on $(G,\alpha)$
if, for any vertices $u$ and $v \in V$ and any edge $e \in E$ between $u$ and $v$,
$f$ satisfies the condition that $f(u) - f(v)$ belongs to the ideal $\alpha(e)$.
This condition is called the \textit{GKM condition}.
It is clear that the set of all GKM functions on $(G,\alpha)$ forms a ring,
and we call it the \textit{cohomology} of $(G,\alpha)$ and denote it by $H^*(G,\alpha)$.
Notice that any homomorphism $(\phi,\psi) \colon (G,\alpha) \to (G',\alpha')$ between GKM graphs
induces a ring homomorphism $H^*(\phi,\psi) \colon H^*(G',\alpha') \to H^*(G,\alpha)$
defined by $H^*(\phi,\psi)(f) = \psi \circ f \circ \phi$ for $f \in H^*(G',\alpha')$,
and $H^*$ is a contravariant functor from the category of GKM graphs to the category of rings.

\subsection{Change of rings}
Let $(G,\alpha_1)$ and $(G,\alpha_2)$ be GKM graphs with the same underlying graph $G=(V,E)$
and $\alpha_i \colon E \to \mathcal{I}_{R_i}$ for $i=1,2$.
We consider the homomorphism between the cohomology rings of GKM graphs
$\psi^* = H^*(\psi) \colon H^*(G,\alpha_1) \to H^*(G,\alpha_2)$
induced from a ring homomorphism $\psi \colon R_1 \to R_2$.
Through this homomorphism,
we will translate generators of one of these cohomology rings and their relations
to those of the other cohomology ring and their relations.

More generally, we consider the following setting.
For $i=1,2$,
let $R_i$ be a commutative ring with unity,
$\psi \colon R_1 \to R_2$ a homomorphism, and $M_i$ an $R_i$-algebra.
We suppose two conditions hold:
\begin{enumerate}
\item
$M_1$ is a free $R_1$-module.
\item
$M_2$ is isomorphic with $R_2 \otimes_{R_1} M_1$.
\end{enumerate}
Then $\psi$ induces an $R_1$-module homomorphism $M_1 \to M_2$
and we denote it by $\psi^*$.

First let us consider the case where we know generators of $M_1$ and their relations.
Let $\{ g_\lambda \}_{\lambda \in \Lambda}$ be $R_1$-algebra generators of $M_1$.
Then they are also regarded as $R_2$-algebra generators of $M_2$.
Consider the projection from a polynomial ring
\begin{align*}
&\Xi_1 \colon R_1[x_\lambda \suchthat \lambda \in \Lambda] \to M_1, \quad x_\lambda \mapsto g_\lambda,\\
&\Xi_2 \colon R_2[x_\lambda \suchthat \lambda \in \Lambda] \to M_2, \quad x_\lambda \mapsto \psi^*(g_\lambda).
\end{align*}
The following proposition says that
relations of $g_\lambda$'s in $M_1$
can be regarded as relations of $\psi^*(g_\lambda)$'s in $M_2$.

\begin{prop}
\label{images of relations}
Let $\{ r_\theta \}_{\theta \in \Theta}$ be generators of $\mathrm{Ker} \: \Xi_1$.
Then $\{ \psi^*(r_\theta) \}_{\theta \in \Theta}$
generates $\mathrm{Ker} \: \Xi_2$.
\end{prop}

\begin{proof}
Applying the right exact functor $R_2 \otimes_{R_1} \cdot$ to the exact sequence
\[
\bigoplus_{\theta \in \Theta} R_1[x_\lambda \suchthat \lambda \in \Lambda]r_\theta \to R_1[x_\lambda \suchthat \lambda \in \Lambda] \to M_1 \to 0,
\]
we obtain an exact sequence
\[
\bigoplus_{\theta \in \Theta} R_2[x_\lambda \suchthat \lambda \in \Lambda]\psi^*(r_\theta) \to R_2[x_\lambda \suchthat \lambda \in \Lambda] \to R_2 \otimes_{R_1} M_1 \to 0,
\]
where $R_2 \otimes_{R_1} M_1 \cong M_2$ by the second condition.
Therefore $\{ \psi^*(r_\theta) \}_{\theta \in \Theta}$ generates $\mathrm{Ker} \: \Xi_2$.
\end{proof}

Next let us consider the case where we know generators of $M_2$ and their relations.
In this case, for $i=1,2$,
we suppose that $R_i$ and $M_i$ are non-negatively graded.
Let $R_i^n$ be the submodule of $R_i$ consisting of all elements of degree $n$
and $R_i^+$ the submodule consisting of all elements of positive degree.
Moreover we suppose that $\psi$ induces an isomorphism $R_1/R_1^+ \cong R_2/R_2^+$ as $R_1$-modules.
For a set function $\phi \colon X \to Y$ and an element $y \in Y$,
$x \in X$ is called a \textit{lift} of $y$ through $\phi$ when $\phi(x)=y$.
\begin{prop}
\label{lifts of generators}
Let $\{ g_\lambda \}_{\lambda \in \Lambda}$ be $R_2$-algebra generators of $M_2$.
Then any lifts $\{ \tilde{g}_\lambda \}_{\lambda \in \Lambda}$ of $\{ g_\lambda \}_{\lambda \in \Lambda}$
through $\psi^*$ are $R_1$-algebra generators of $M_1$.
\end{prop}

\begin{proof}
Let $N$ denote the cokernel of the $R_1$-algebra homomorphism
$R_1[x_\lambda \suchthat \lambda \in \Lambda] \to M_1$ defined by $x_\lambda \mapsto \tilde{g}_\lambda$.
We will show $N=0$.
Applying the right exact functor $R_2 \otimes_{R_1} \cdot$ to the exact sequence
\[
R_1[x_\lambda \suchthat \lambda \in \Lambda] \to M_1 \to N \to 0,
\]
we obtain an exact sequence
\[
R_2[x_\lambda \suchthat \lambda \in \Lambda] \to R_2 \otimes_{R_1} M_1 \to R_2 \otimes_{R_1} N \to 0.
\]
By the hypothesis, $R_2[x_\lambda \suchthat \lambda \in \Lambda] \to R_2 \otimes_{R_1} M_1 \cong M_2$ is surjective,
and then $R_2 \otimes_{R_1} N = 0$.
Since $R_1/R_1^+ \cong R_2/R_2^+$,
we have $N \otimes_{R_1} R_1/R_1^+ \cong N \otimes_{R_1} R_2 \otimes_{R_2} R_2/R_2^+ = 0$.
By the graded Nakayama's lemma (cf. \cite[Lemma 13.4]{P}),
$N /R_1^+N = 0$ implies $N = 0$.
\end{proof}

Consider the projections
\begin{align*}
&\Xi_1 \colon R_1[x_\lambda \suchthat \lambda \in \Lambda] \to M_1, \quad x_\lambda \mapsto \tilde{g}_\lambda,\\
&\Xi_2 \colon R_2[x_\lambda \suchthat \lambda \in \Lambda] \to M_2, \quad x_\lambda \mapsto g_\lambda
\end{align*}
and the homomorphism $R_1[x_\lambda \suchthat \lambda \in \Lambda] \to R_2[x_\lambda \suchthat \lambda \in \Lambda]$
induced by $\psi$, which we also denote by $\psi^*$.

\begin{prop}
\label{lifts of relations}
Let $\{ r_\theta \}_{\theta \in \Theta}$ be generators of $\mathrm{Ker} \: \Xi_2$.
Then any lifts $\{ \tilde{r}_\theta \}_{\theta \in \Theta}$
of $\{ r_\theta \}_{\theta \in \Theta}$ through $\psi^*$
contained in $\mathrm{Ker} \: \Xi_1$ generates $\mathrm{Ker} \: \Xi_1$.
\end{prop}

\begin{proof}
Let $N$ denote the cokernel of the natural $R_1$-module homomorphism
$\bigoplus_{\theta \in \Theta} R_1[x_\lambda \suchthat \lambda \in \Lambda]\tilde{r}_\theta \to \mathrm{Ker} \: \Xi_1$.
We will show $N=0$.
Since $M_1$ is $R_1$-free,
applying the right exact functor $R_2 \otimes_{R_1} \cdot$ to the exact sequences
\[
\bigoplus_{\theta \in \Theta} R_1[x_\lambda \suchthat \lambda \in \Lambda]\tilde{r}_\theta \to \mathrm{Ker} \: \Xi_1 \to N \to 0
\]
and
\[
0 \to \mathrm{Ker} \: \Xi_1 \to R_1[x_\lambda \suchthat \lambda \in \Lambda] \to M_1 \to 0,
\]
we obtain two exact sequences
\[
\bigoplus_{\theta \in \Theta} R_2[x_\lambda \suchthat \lambda \in \Lambda]r_\theta \to R_2 \otimes_{R_1} \mathrm{Ker} \: \Xi_1 \to R_2 \otimes_{R_1} N \to 0
\]
and 
\[
0 \to R_2 \otimes_{R_1} \mathrm{Ker} \: \Xi_1 \to R_2[x_\lambda \suchthat \lambda \in \Lambda] \to M_2 \to 0,
\]
respectively.
The second exact sequence says that $R_2 \otimes_{R_1} \mathrm{Ker} \: \Xi_1 \cong \mathrm{Ker} \: \Xi_2$,
and then the first one says that $R_2 \otimes_{R_1} N = 0$.
Since $R_1/R_1^+ \cong R_2/R_2^+$,
we have $N \otimes_{R_1} R_1/R_1^+ \cong N \otimes_{R_1} R_2 \otimes_{R_2} R_2/R_2^+ = 0$.
By the graded Nakayama's lemma, $N /R_1^+N = 0$ implies $N = 0$.
\end{proof}

\section{GKM graphs for Lie groups}
In this section, 
we define the GKM graphs of (general) flag manifolds
and reveal the module structure of their cohomology.
Moreover we give a powerful method to analyze their cohomology rings.
\subsection{Definition}
Let $G$ be a compact connected Lie group
and $T$ a maximal torus of $G$.
We will use terminology in Schubert calculus,
so it is convenient to consider the complexification $G_\C$ of $G$.
Let $\Phi$ and $W$ denote the root system and the Weyl group of $G_\C$, respectively.
We denote the reflection corresponding to a root $\alpha \in \Phi$ by $\sigma_\alpha$.
In the standard manner, we regard $\Phi$ as the subgroup of $H^2(BT)$.
The Weyl group $W$ acts naturally on $H^2(BT)$
which restricts to the action on $\Phi$.
Let $B$ be a Borel subgroup of $G_\C$ containing $T$.
Then we have a $T$-equivariant isomorphism
\[
G_\C/B \cong G/T
\]
where $T$ acts on both sides by the left multiplication.

Note that the complexification of a connected compact Lie group is reductive
(cf. \cite[Theorem 5.8]{Ka}).
Let $P$ be a parabolic subgroup of $G_\C$ containing the Borel subgroup $B$
and $W_P$ the Weyl group of $P$
(cf. \cite[Section 28 and 29]{Hu}).
Of course, $W_B=\{e\}$.
We define the GKM graph $\G(G_\C/P)$ as follows.
The vertex set of $\G(G_\C/P)$ is $W/W_P$.
Moreover for any vertices $vW_P \neq wW_P$,
if there is a positive root $\alpha \in \Phi$ such that $\sigma_\alpha vW_P = wW_P$,
there is a corresponding edge and its label is the ideal $(\alpha) \subset H^*(BT)$,
where $(x_1, \ldots , x_n)$ means the ideal generated by $x_1, \ldots , x_n$.
We will not distinguish the ideal $(\alpha)$ and its generator $\alpha$,
and so we call $\alpha$ the label of an edge $e$.
To accord with our definition of GKM graph homomorphisms,
we will treat $\G(G_\C/P)$ as if it has a loop labelled by the ideal $\{ 0 \} \subset H^*(BT)$ at each vertex.
Note that the existence of a loop (with any label) does not affect the cohomology of a GKM graph.
The following theorem claims that
the GKM graph of $\G(G_\C/P)$ completely determines the equivariant integral cohomology of $G_\C/P$.

\begin{thm}[{\cite[Theorem 3.1 and Lemma 5.2]{HHH}}]
\label{HHH}
Suppose that the set $\Phi^+$ of the positive roots of $G$ is pairwise relatively prime,
that is, any two of $\Phi^+$ are relatively prime in $H^*(BT)$.
Then there is an isomorphism
\[
H^*_T(G_\C/P) \cong H^*(\G(G_\C/P)).
\]
\end{thm}

\subsection{Divided difference operators}
Let us introduce the action of the Weyl group $W$ on $H^*(\G(G_\C/P))$:
for a GKM function $f$ on $\G(G_\C/P)$, $w \in W$, and a vertex $v$ of $\G(G_\C/P)$,
\[
(w \cdot f)(v) = w(f(w^{-1}v)).
\]
Assume that the set $\Phi^+$ of all the positive roots is pairwise relatively prime in $H^*(BT)$.
For $\alpha \in \Phi$,
we can define an operator $\delta_\alpha \colon H^*(\G(G_\C/P)) \to H^*(\G(G_\C/P))$,
called the \textit{left divided difference operator}, as follows;
for a GKM function $f$ on $\G(G_\C/P)$ and a vertex $v$ of $\G(G_\C/P)$,
\[
\delta_\alpha f(v) = \frac{1}{\alpha}(f(v)-\sigma_\alpha \cdot f(v)).
\]
We use the left divided difference operators 
and obtain $H^*(BT)$-module basis of $H^*(\G(G_\C/P))$ in the next subsection.

Let us verify that $\delta_\alpha f$ is a GKM function on $\G(G_\C/P)$.
First we will check that
$\delta_\alpha f(v) = \frac{1}{\alpha}(f(v)-\sigma_\alpha \cdot f (v))$
is well-defined and contained in $H^*(BT)$.
We have
\begin{equation}
\label{mult. of alpha}
  f(v) -\sigma_\alpha \cdot f (v) = (1-\sigma_\alpha)f(v) +\sigma_\alpha(f(v) -f(\sigma_\alpha v)).
\end{equation}
The second term of the right-hand side is a multiple of $\alpha$.
Since $H^*(BT)$ is generated by elements of degree $2$,
write $f(v) = \sum_i \prod_{j_i} x_{i,j_i}$, where $x_{i,j_i} \in H^2(BT)$.
Then we have
\[
(1-\sigma_\alpha)f(v) = \sum_i \prod_{j_i} x_{i,j_i} - \sum_i \prod_{j_i} \Bigl(x_{i,j_i}- 2\frac{(x_{i,j_i},\alpha)}{(\alpha,\alpha)}\alpha \Bigr).
\]
Since $H^2(BT)$ is the set of weights, as in \cite[Theorem 6.3. (2)]{MT},
$2\frac{(x_{i,j_i},\alpha)}{(\alpha,\alpha)}$ is an integer.
Hence the first term of the right-hand side of \eqref{mult. of alpha} is a multiple of $\alpha$.
Thus the fraction $\frac{1}{\alpha}(f(v)-\sigma_\alpha \cdot f (v))$ is well-defined.

Next we will check that $\delta_\alpha f$ satisfies the GKM condition.
For a positive root $\beta \in \Phi^+$, we have
\begin{align*}
 &\delta_\alpha f(v) - \delta_\alpha f(\sigma_\beta v)\\
=&\frac{1}{\alpha}\Bigl(f(v) - \sigma_\alpha f(\sigma_\alpha v) - f(\sigma_\beta v) + \sigma_\alpha f(\sigma_\alpha \sigma_\beta v)\Bigr)\\
=&\frac{1}{\alpha}\Bigl(f(v) - f(\sigma_\beta v) - \sigma_\alpha \bigl(f(\sigma_\alpha v) - f(\sigma_{\sigma_\alpha \beta} \sigma_\alpha v)\bigr)\Bigr).
\end{align*}
Since $f$ satisfies the GKM condition,
$f(v) - f(\sigma_\beta v)$ is a multiple of $\beta$
and $f(\sigma_\alpha v) - f(\sigma_{\sigma_\alpha \beta} \sigma_\alpha v)$ is a multiple of $\sigma_\alpha \beta$.
Hence $f(v) - \sigma_\alpha f(\sigma_\alpha v) - f(\sigma_\beta v) + \sigma_\alpha f(\sigma_\alpha \sigma_\beta v)$
is a multiple of $\beta$.
It is also a multiple of $\alpha$ by the above well-definedness of $\delta_\alpha f$.
Hence, if $\alpha \neq \beta$, it is a multiple of $\alpha \beta$ since $\alpha$ and $\beta$ are relatively prime.
If $\alpha = \beta$, we have
\begin{equation}
\label{alpha^2}
f(v) - \sigma_\alpha f(\sigma_\alpha v) - f(\sigma_\alpha v) + \sigma_\alpha f(v) = (1+\sigma_\alpha)(f(v)-f(\sigma_\alpha v)).
\end{equation}
Put $f(v) - f(\sigma_\alpha v) = \alpha x$, where $x \in H^*(BT)$.
Then the right hand side of \eqref{alpha^2} is $\alpha x -\alpha \sigma_\alpha x = \alpha (1-\sigma_\alpha)x$.
Therefore it is a multiple of $\alpha^2$
since we have verified that $(1-\sigma_\alpha)x$ is a multiple of $\alpha$ for any $x \in H^*(BT)$.

If $\Phi^+$ is not pairwise relatively prime in $H^*(BT)$,
we can not define the divided difference operator $\delta_\alpha \colon H^*(\G(G_\C/P)) \to H^*(\G(G_\C/P))$
as in the following example.
\begin{ex}
 Set $G = Sp(1) \times Sp(1)$ and fix a maximal torus $T$ of $G$.
 The root system of $G$ is given as $\{ \pm 2t_1, \pm 2t_2 \}$
 where $H^*(BT) = \Z[t_1,t_2]$.
 Then the GKM graph $\G(G_\C/B)$ has only 4 vertices.
 Let $f$ be a GKM function of $\G(G_\C/B)$ defined as follows;
\begin{align*}
 f(w) =
 \begin{cases}
  2 t_1 t_2 & w=e,\\
  0 & \text{otherwise.}
 \end{cases}
\end{align*}
Then we have $\delta_{2t_1}f(e) = t_2$, $\delta_{2t_1}f(\sigma_{2t_2})=0$,
and $\delta_{2t_1}f(e)-\delta_{2t_1}f(\sigma_{2t_2}) \not \in (2t_2)$.
Hence $\delta_{2t_1}f$ is no longer a GKM function.
\end{ex}

\subsection{Equivariant Schubert classes}
\label{Equivariant Schubert classes}
Now we can construct $H^*(BT)$-module basis of $H^*(\G(G_\C/P))$
by virtue of the left divided difference operators.
We need a suitable order on the vertex set $W/W_P$ to construct the basis.

We recall basic properties of Weyl groups.
Let $\Delta$ be the set of all the simple roots of $G_\C$.
The Weyl group $W$ has the \textit{Bruhat order} and \textit{length function} $l$ as follows.
The length function $l \colon W \to \Z_{\geq 0}$
assigns to an element $w \in W$ the least number of factors in the decomposition
\[
w = \sigma_{\alpha_1} \sigma_{\alpha_2} \cdots \sigma_{\alpha_n}, \quad \alpha_i \in \Delta.
\]
A decomposition
$w = \sigma_{\alpha_1} \sigma_{\alpha_2} \cdots \sigma_{\alpha_n}$ ($\alpha_i \in \Delta$)
is called a \textit{reduced decomposition}
when $n=l(w)$.
Let $w_1$, $w_2 \in W$ and $\gamma \in \Phi^+$.
Then $w_1 \arrow w_2$ indicates the fact that
$\sigma_\gamma w_1 = w_2$ and $l(w_2)=l(w_1)+1$.
We put $w < w'$ if there is a chain
\[
w=w_1 \to w_2 \to \cdots \to w_k =w'.
\]
This order on $W$ is called the Bruhat order.


Let $\Phi^-$ denote the set of all the negative roots of $G_\C$.
It is well-known that, for $w \in W$, the length $l(w)$ coincides with
the number of elements of the set $\Phi^+ \cap w(\Phi^-)$
and that for $\alpha \in \Delta$, the reflection $\sigma_\alpha$
sends $\alpha$ to the negative root $-\alpha$ and permutes other positive roots.

\begin{lem}[{\cite[Lemma 2.2]{BGG}}]
\label{BGG lem. 2.2}
Let $w = \sigma_{\alpha_1} \sigma_{\alpha_2} \cdots \sigma_{\alpha_l}$ be a reduced decomposition.
We put $\gamma_i = \sigma_{\alpha_1}\cdots \sigma_{\alpha_{i-1}} \alpha_i$.
Then the roots $\gamma_1, \ldots, \gamma_l$ are distinct and the set $\{\gamma_1, \ldots, \gamma_l\}$
coincides with $\Phi^+ \cap w(\Phi^-)$.
\end{lem}

\begin{proof}
Since $w = \sigma_{\alpha_1} \sigma_{\alpha_2} \cdots \sigma_{\alpha_l}$
is a reduced decomposition,
for any $1 \leq i \leq l$,
$w_i = \sigma_{\alpha_i} \cdots \sigma_{\alpha_l}$
sends exactly $l-i+1$ positive roots to negative roots.
Since $\sigma_{\alpha_i} \alpha_i = -\alpha_i$ is negative,
$\alpha_i \in w_{i+1}(\Phi^+)$
and $\sigma_{\alpha_1} \cdots \sigma_{\alpha_{i-1}}$
sends $-\alpha_i \in w_i(\Phi^+)$ to a negative root.
Hence $\gamma_i = \sigma_{\alpha_1} \cdots \sigma_{\alpha_{i-1}}\alpha_i = w (w_i^{-1} \alpha_i)$
is positive and $w_i^{-1} \alpha_i$ is negative.
Moreover, for $j<i$, $w_j (w_i^{-1} \alpha_i)$ is positive and,
for $j \geq i$, $w_j (w_i^{-1} \alpha_i)$ is negative.
Hence $w_i^{-1} \alpha_i$'s are distinct.
Therefore $\gamma_i$'s are distinct.
\end{proof}

\begin{cor}[{\cite[Corollary 2.3]{BGG}}]
\label{BGG cor. 2.3}
{\rm (i)} Let $w = \sigma_{\alpha_1} \sigma_{\alpha_2} \cdots \sigma_{\alpha_l}$ be a reduced decomposition
and let $\gamma \in \Phi^+$ be a root such that $w^{-1}\gamma \in \Phi^-$.
Then for some $i$
\[
\sigma_\gamma \sigma_{\alpha_1} \cdots \sigma_{\alpha_i} = \sigma_{\alpha_1}  \cdots \sigma_{\alpha_{i-1}}.
\]

{\rm (ii)} Let $w \in W$, $\gamma \in \Phi^+$.
Then $l(w) < l(\sigma_\gamma w)$, if and only if $w^{-1}\gamma \in \Phi^+$.
\end{cor}

\begin{proof}
{\rm (i)}
From Lemma \ref{BGG lem. 2.2} we deduce that
$\gamma = \sigma_{\alpha_1} \cdots \sigma_{\alpha_{i-1}}\alpha_i$ for some $i$,
and then the reflection $\sigma_\gamma$ is the conjugation of
$\sigma_{\alpha_i}$ by $\sigma_{\alpha_{i-1}} \cdots \sigma_{\alpha_1}$.

{\rm (ii)}
If $w^{-1}\gamma \in \Phi^-$,
then by {\rm (i)} $\sigma_\gamma w = \sigma_{\alpha_1} \cdots \sigma_{\alpha_{i-1}}\sigma_{\alpha_{i+1}} \cdots \sigma_{\alpha_l}$,
that is $l(\sigma_\gamma w) < l(w)$.
Interchanging $w$ and $\sigma_\gamma w$,
we see that if $w^{-1}\gamma \in \Phi^+$,
then $l(w) < l(\sigma_\gamma w)$.
\end{proof}

\begin{lem}[{\cite[Lemma 2.4]{BGG}}]
\label{lem. wedge}
Let $w_1$, $w_2 \in W$, $\alpha \in \Delta$, $\gamma \in \Phi^+$, and $\gamma \neq \alpha$.
Let $\gamma' = \sigma_\alpha \gamma \in \Phi^+$.
The diagrams
\[
  \begin{xy}
  	( 0, 0)="saw1"*{\sigma_\alpha w_1},
  	(24, 6)="w2"*{w_2},
  	(24,-6)="w1"*{w_1},
  	(35, 0)=""*{\text{and}},
  	(70, 0)="saw2"*{\sigma_\alpha w_2},
  	(46, 6)="2w2"*{w_2},
  	(46,-6)="2w1"*{w_1},
      \ar^{\gamma}"saw1" +/u1mm/+/r4.5mm/;	"w2"+/d1mm/+/l3mm/,
      \ar^{\alpha}"saw1" +/d1mm/+/r4.5mm/;	"w1"+/u1mm/+/l3mm/,
      \ar^{\alpha}"2w2"+/d1mm/+/r3mm/;	"saw2" +/u1mm/+/l4.5mm/,
      \ar^{\gamma'}"2w1"+/u1mm/+/r3mm/;	"saw2" +/d1mm/+/l4.5mm/
  \end{xy}
\]
are equivalent.
\end{lem}

\begin{proof}
Since $\alpha \in \Delta$ and $\gamma \neq \alpha$,
we have $\gamma' = \sigma_\alpha \gamma \in \Phi^+$.
Therefore it is sufficient to show that
$l(\sigma_\alpha w_2) > l(w_2)= l(w_1)$.
This follows from Corollary \ref{BGG cor. 2.3} {\rm (ii)},
because $\sigma_\alpha w_2 = \sigma_\alpha \sigma_\gamma \sigma_\alpha^{-1} w_1 = \sigma_{\gamma'} w_1$
and $(\sigma_\alpha w_2)^{-1}\gamma' = w_2^{-1} \sigma_\alpha\gamma' = w_2^{-1}\gamma \in \Phi^-$.
The inverse implication is proved similarly.
\end{proof}

\begin{prop}[{\cite[Propositin 2.8 a)]{BGG}}]
\label{BGG prop. 2.8}
Let $w \in W$ and $w = \sigma_{\alpha_1} \sigma_{\alpha_2} \cdots \sigma_{\alpha_l}$ be a reduced decomposition.
If $1 \leq i_1 < i_2 < \cdots < i_k \leq l$
and $w' = \sigma_{\alpha_{i_1}} \cdots \sigma_{\alpha_{i_k}}$,
then $w' \leq w$.
\end{prop}

\begin{proof}
We prove this by induction on $l$.
We treat two cases separately.

{\rm (i)} $i_1>1$.
By the inductive hypothesis $w' \leq \sigma_{\alpha_2} \cdots \sigma_{\alpha_l} < w$.

{\rm (ii)} $i_1=1$.
By the inductive hypothesis
$\sigma_{\alpha_1} w' \leq \sigma_{\alpha_2} \cdots \sigma_{\alpha_l} < \sigma_{\alpha_1}w$.
Applying Lemma \ref{lem. wedge} to the chain
$w' \overset{\alpha_1}{\leftarrow} \sigma_{\alpha_1} w' \to \cdots \to \sigma_{\alpha_1}w$ repeatedly,
we can see that $w' < w$
\end{proof}

For $w \in W$ and  $\gamma \in \Phi^+$,
combining Proposition \ref{BGG prop. 2.8} and Corollary \ref{BGG cor. 2.3} (i),
we obtain that $\gamma \in \Phi^+ \cap w(\Phi^-)$ implies $\sigma_\gamma w < w$
and that $\gamma \not \in \Phi^+ \cap w(\Phi^-)$ ($\Leftrightarrow \gamma \in \Phi^+ \cap \sigma_\gamma w(\Phi^-)$) implies $\sigma_\gamma w > w$.

Let $P$ be the parabolic subgroup of $G_\C$ corresponding to a subset $\Sigma \subset \Delta$
and $\langle \Sigma \rangle$ denote the set of positive roots
that can be written as linear combinations of roots in $\Sigma$.
The Bruhat order on $W$ induces an relation
 on $W/W_P$ as follows;
$vW_P \leq v'W_P$ if and only if there are coset representatives $v_0$ and $v_0'$
of $vW_P$ and $v'W_P$, respectively, such that $v_0 \leq v_0'$.
We will show that this relation is a partial order 
and call this order the Bruhat order on $W/W_P$.
Lemma \ref{lem. wedge} helps us to analyze this relation.

We define the length function $l^P$ on $W/W_P$ as
$l^P(wW_P) = \mathrm{min}\{ l(v) \suchthat v \in wW_P \}$.
If $v \in wW_P$ satisfies $l(v)=l^P(wW_P)$,
we call $v$ a \textit{minimal length representative} of $wW_P$,
and let $\bar{w}$ denote a minimal length representative of $wW_P$.
\begin{lem}
\label{minimal length}
For any $w \in W$, a minimal length representative $\bar{w}$ of $wW_P$ is unique,
and, for $w= \bar{w} \sigma \in wW_P$,
we have $l(w) = l(\bar{w}) + l(\sigma)$.
\end{lem}

\begin{proof}
For any $v \in wW_P$,
there is $\sigma \in W_P$ such that $\bar{w}\sigma = v$,
and let $\sigma = \sigma_{\alpha_1} \cdots \sigma_{\alpha_n}$ be a reduced decomposition
($\alpha_i \in \Sigma$).
We can draw the following picture of chains which looks like a mountain range.
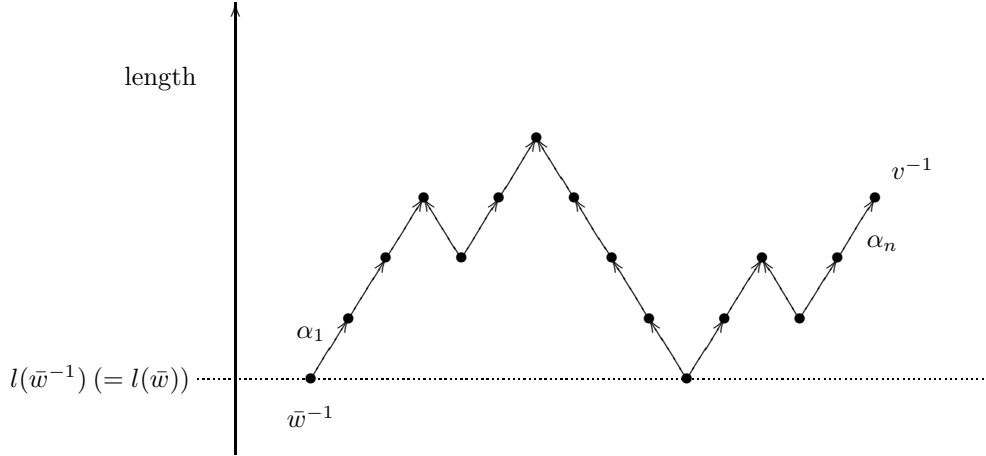
\[
\begin{xy}
	( -5, 0)="left",
	(100, 0)="right",
  	(  0,-10)="down",
  	(  0, 50)="up",
  	(-10, 40)=""*{\text{length}},
  	(-18,  0)=""*{l(\bar{w}^{-1})\: (=l(\bar{w}))},
	( 10,  0)="w0"*{\bullet},
      "w0"+(5,8)="w1"*{\bullet},
      "w1"+(5,8)="w2"*{\bullet},
      "w2"+(5,8)="w3"*{\bullet},
      "w3"+(5,-8)="w4"*{\bullet},
      "w4"+(5,8)="w5"*{\bullet},
      "w5"+(5,8)="w6"*{\bullet},
      "w6"+(5,-8)="w7"*{\bullet},
      "w7"+(5,-8)="w8"*{\bullet},
      "w8"+(5,-8)="w9"*{\bullet},
      "w9"+(5,-8)="w10"*{\bullet},
      "w10"+(5,8)="w11"*{\bullet},
      "w11"+(5,8)="w12"*{\bullet},
      "w12"+(5,-8)="w13"*{\bullet},
      "w13"+(5,8)="w14"*{\bullet},
      "w14"+(5,8)="w15"*{\bullet},
	( 10, -5)=""*{\bar{w}^{-1}},
	( 90, 28)=""*{v^{-1}},
	( 10, 6)=""*{\alpha_1},
	( 86, 18)=""*{\alpha_n},
      \ar@{.}"left";	"right",
      \ar"down";	"up",
      \ar"w0";"w1",
      \ar"w1";"w2",
      \ar"w2";"w3",
      \ar"w4";"w3",
      \ar"w4";"w5",
      \ar"w5";"w6",
      \ar"w7";"w6",
      \ar"w8";"w7",
      \ar"w9";"w8",
      \ar"w10";"w9",
      \ar"w10";"w11",
      \ar"w11";"w12",
      \ar"w13";"w12",
      \ar"w13";"w14",
      \ar"w14";"w15"
\end{xy}
\]
The $i$-th node denotes the inverse of $\bar{w}\sigma_{\alpha_1} \cdots \sigma_{\alpha_{i-1}} \in wW_P$
and the $i$-th arrow is equipped with $\alpha_i$.
By definition, for any $v \in W$, we have $l(v) = l(v^{-1})$.
Hence, for any $v \in wW_P$, we have $l(\bar{w}^{-1}) \leq l(v^{-1})$.
Applying Lemma \ref{lem. wedge} to the summits,
which are locally highest nodes except the right end node $v^{-1}$,
we can obtain a new lower mountain range
such that all right-up arrows are equipped with simple roots in $\Sigma$.
Moreover, if the new mountain range has a summit $w_2$
of the form $w_1 \arrow[\alpha] w_2 \overset{\alpha}{\leftarrow} w_1$,
then we remove this to get the node $w_1$.
Any node of the new mountain range is also of length at least $l(\bar{w}) = l(\bar{w}^{-1})$.
Finally, by iteration of this modification,
we obtain an upward chain $\bar{w}^{-1} \to \cdots \to v^{-1}$ such that all arrows are equipped with simple roots in $\Sigma$.
Therefore a minimal length representative of $wW_P$ is unique.
\end{proof}

Lemma \ref{minimal length} says
that, for $w \in W$, the minimal length representative $\bar{w}$ of $wW_P$
is minimum in $wW_P$ as a subposet of $W$.
Moreover the left multiplication $\bar{w} \colon W_P \to \bar{w}W_P$
is a poset isomorphism,
since, for $\gamma \in \langle \Sigma \rangle$, $v_1 \arrow v_2$ if and only if $\bar{w}v_1 \arrow[\bar{w}\gamma] \bar{w}v_2$.
The following lemma is proved by the same argument as in the proof of Lemma \ref{minimal length}
and shows that the relation on $W/W_P$, which we have defined above, is actually a partial order.

\begin{lem}[{\cite[Lemma 3.5]{D}}]
\label{D lem. 3.5}
Let $w_1$, $w_2 \in W$ and $\bar{w}_1$, $\bar{w}_2$ the minimal length representatives
of $w_1W_P$, $w_2W_P$, respectively.
Then $\bar{w}_1 \leq \bar{w}_2$ if and only if $\bar{w}_1 \leq w_2$.
\end{lem}

\begin{proof}
We only have to show that $\bar{w}_1 \leq \bar{w}_2$ if $\bar{w}_1 \leq w_2$.
By Lemma \ref{minimal length}, there is a chain
$\bar{w}_1 \to \cdots \to w_2 \overset{\alpha_1}{\leftarrow} \cdots \overset{\alpha_n}{\leftarrow} \bar{w_2}$,
where $\alpha_1, \ldots, \alpha_n \in \Sigma$.
Then, applying Lemma \ref{lem. wedge} to this chain repeatedly,
we see $\bar{w}_1 \leq \bar{w}_2$ since $\bar{w}_2$ is the minimum element of $w_2W_P$.
\end{proof}

Lemma \ref{D lem. 3.5} claims that
the poset $W/W_P$ is identified with the subposet of $W$
consisting of all minimal length representatives.
For $w \in W$ and $\gamma \in \Phi^+ \setminus w(-\langle \Sigma \rangle)$,
combining Proposition \ref{BGG prop. 2.8} and Lemma \ref{D lem. 3.5},
Corollary \ref{BGG cor. 2.3} (i) guarantees that
$\gamma \in \Phi^+ \cap w(\Phi^- \setminus -\langle \Sigma \rangle)$ implies $\sigma_\gamma w W_P < w W_P$
and that $\gamma \not \in \Phi^+ \cap w(\Phi^- \setminus -\langle \Sigma \rangle)$ implies $\sigma_\gamma w W_P > w W_P$.
Thus we obtain the following proposition.
\begin{prop}
\label{length}
The length $l^P(w)$ coincides with the number of elements of $\Phi^+ \cap w(\Phi^- \setminus -\langle \Sigma \rangle)$.
\end{prop}

Now we are ready to construct a basis of $H^*(\G(G_\C/P))$.
A non-zero GKM function $f$ from the vertex set $V$ of a given graph
to a given graded ring $R$ is called \textit{homogeneous} of degree $n$
if, for any $v \in V$, $f$ satisfies that $f(v)=0$ or $f(v)$ is homogeneous of degree $n$.
The (equivariant) \textit{Schubert classes} $\{ S_w^P \}_{w \in W/W_P}$ of $G_\C/P$ are elements of $H^*(\G(G_\C/P))$,
which satisfy the following conditions;
\begin{itemize}
\item
$S_w^P$ is homogeneous of degree $2l^P(w)$,
\item
$S_w^P(w) = \prod_{\alpha \in \Phi^+ \cap w(\Phi^- \setminus -\langle \Sigma \rangle)} \alpha$,
\item
if $w \not< v$, then $S_w^P(v)=0$.
\end{itemize}
When $P=B$, we write simply $S_w$ for $S_w^P$.
These are equivariant analogues of the ordinary Schubert classes,
which are the Poincar\'e dual of the closures of the Schubert cells (cf. \cite{BGG}).
We will prove that the Schubert classes uniquely exist (Proposition \ref{Sc. unique} and \ref{Sc. exists})
and form a $H^*(BT)$-basis of $H^*(\G(G_\C/P))$ (Proposition \ref{Sc. basis}).

\begin{prop}
\label{Sc. unique}
The Schubert classes $\{ S_w^P \}_{w \in W/W_P}$ are unique if they exist.
\end{prop}

\begin{proof}
For $w \in W/W_P$, if there are two distinct Schubert classes $S_w^P$, $T_w^P$,
then $S_w^P - T_w^P$ vanishes out of $\{ v \in W/W_P \suchthat l^P(v) \geq l^P(w)+1 \}$.
By Proposition \ref{length}, such GKM function is of degree at least $2(l^P(w) +1)$.
However $S_w^P$ and $T_w^{P}$ are of degree $2l^P(w)$,
and this is a contradiction.
\end{proof}

\begin{lem}
\label{lem. daSw=Ssaw}
Suppose that $w \in W$ and $\alpha \in \Delta$ satisfy $\sigma_\alpha w < w$
and that the Schubert class $S_w$ exists.
Then $\delta_\alpha S_w$ exists and $\delta_\alpha S_w = S_{\sigma_\alpha w}$.
\end{lem}

\begin{proof}
By the definition of the left divided difference operator,
$\delta_\alpha S_w$ vanishes out of $\{ \sigma_\alpha w \} \cup \{ v \in W \suchthat l(v) \geq l(w) \}$.
In particular, if $w < v$, then $\sigma_\alpha w < \sigma_\alpha v$ by Lemma \ref{lem. wedge}.
Hence $\delta_\alpha S_w$ vanishes out of $\{ v \in W \suchthat \sigma_\alpha w < v \}$.
By Lemma \ref{BGG lem. 2.2}, a straightforward computation shows
\[
\delta_\alpha S_w (\sigma_\alpha w) =
\prod_{\beta \in (\Phi^+ \setminus \alpha) \cap w(\Phi^-)} \sigma_\alpha(\beta) = \prod_{\beta \in \Phi^+ \cap \sigma_\alpha w(\Phi^-)}\beta.
\]
Therefore $\delta_\alpha S_w = S_{\sigma_\alpha w}$.
\end{proof}

\begin{prop}
\label{Sc. exists}
There exist the Schubert classes $\{ S_w^P \}_{w \in W/W_P}$.
\end{prop}

\begin{proof}
First, we prove that Proposition \ref{Sc. exists} holds in the case $P=B$.
The unique element $w_0 \in W$ which sends all positive roots to negative roots
is called the \textit{longest element} of $W$.
A cell $BwB/B$ of the Bruhat decomposition $G/T = G_\C/B = \coprod_{w \in W}BwB/B$
is of dimension $2l(w)$
and $G/T$ is of dimension $2 |\Phi^+|$.
In particular, the top cell is unique since each cell is even dimensional.
Hence the element of $W(G)$ corresponding to the top cell
must be the longest element.
Obviously,
\begin{align*}
S_{w_0}(w) =
 \begin{cases}
 \prod_{\alpha \in \Phi^+} \alpha & w=w_0,\\
 0					    & w \neq w_0
 \end{cases}
\end{align*}
is a GKM function.
To show that there exists the Schubert class $S_w$ for other $w \in W$,
consider reduced decompositions of $w$ and $w_0$
and regard them as chains from the unit element $e$.
Then iteration of Lemma \ref{lem. wedge} leads us to the facts
that there is a chain $w \to \cdots \to w_0$ equipped with simple roots as arrows.
Therefore, in the case $P=B$, Proposition \ref{Sc. exists} follows from Lemma \ref{lem. daSw=Ssaw}.

Next, we prove Proposition \ref{Sc. exists} for general $P$.
Recall that, for $w \in W$, the minimal length representative $\bar{w}$ of $wW_P$
is minimum in $wW_P$,
and let us identify $W/W_P$ with the subposet of $W$ consisting of all minimal representatives.
By Lemma \ref{D lem. 3.5},
${\bar{w}_0}$ is the maximum element of $W/W_P$
and we can see that
\begin{align*}
S_{\bar{w}_0}^P(w) =
 \begin{cases}
 \prod_{\alpha \in \Phi^+ \cap \bar{w}_0(-\langle \Sigma \rangle)} \alpha & w = \bar{w}_0,\\
 0												  & w \neq \bar{w}_0
 \end{cases}
\end{align*}
is a GKM function on $\G(G_\C/P)$.
By definition, there is a natural GKM graph homomorphism $\G(G_\C/B) \to \G(G_\C/P)$
with a surjective underlying graph homomorphism
and it induces an injective ring homomorphism $H^*(\G(G_\C/P)) \to H^*(\G(G_\C/B))$.
Under this injection, $S_{\bar{w}_0}^P$ is the Schubert class $S_{\bar{w}_0}$.
To show there exists the Schubert class $S_{\bar{w}}^P$ for other $\bar{w} \in W/W_P$,
we show the equation
\[
l(\bar{w}_0\bar{w}^{-1})=l(\bar{w}_0) -l(\bar{w}).
\]
If this equation holds,
then a reduced decomposition of $\bar{w}_0\bar{w}^{-1}$
forms a chain from $\bar{w}$ to $\bar{w}_0$ such that all arrows are equipped with simple roots.
It means that the sequence of the corresponding divided difference operators sends $S_{\bar{w}_0}$ to $S_{\bar{w}}$
and, under the injection $H^*(\G(G_\C/P)) \to H^*(\G(G_\C/B))$, $S_{\bar{w}}^P$ is the Schubert class $S_{\bar{w}}$.

Let $\sigma$ be the longest element of $W_P$.
By Lemma \ref{minimal length}, we have $w_0 = \bar{w}_0 \sigma$,
and then
\[
l(\bar{w}_0\bar{w}^{-1})  = l({w}_0(\bar{w}\sigma)^{-1}).
\]
Since $w_0$ sends all positive roots to negative roots,
\[
l({w}_0(\bar{w}\sigma)^{-1}) = l({w}_0)-l(\bar{w}\sigma) = l(\bar{w}_0)-l(\bar{w}),
\]
where the latter equality holds by Lemma \ref{minimal length}.
\end{proof}

\begin{prop}
\label{Sc. basis}
Assume that $\Phi^+$ is pairwise relatively prime in $H^*(BT)$.
The Schubert classes $\{ S_w^P \}_{w \in W/W_P}$ form a basis of $H^*(\G(G_\C/P))$ over $H^*(BT)$.
\end{prop}

\begin{proof}
First, we show that $\{ S_w^P \}_{w \in W/W_P}$ is linearly independent over $H^*(BT)$.
Suppose that $\Sigma_{w\in W/W_P} a_w S_w^P = 0$ ($a_w \in H^*(BT)$)
and consider a minimal element $w'$ such that $a_{w'} \neq 0$.
If such $w'$ exists,
then $0 = \Sigma_{w\in W/W_P} a_w S_w^P(w') = a_{w'} S^P_{w'}(w') \neq 0$.
This is a contradiction.

Next we show that $\{ S_w^P \}_{w \in W/W_P}$ spans $H^*(\G(G_\C/P))$.
Fix any GKM function $f \in H^*(\G(G_\C/P))$
and let $w' \in W/W_P$ be a minimal element such that $f(w') \neq 0$.
Since $\Phi^+$ is pairwise relatively prime, the GKM condition says that
$f(w') = b_{w'}\prod_{\alpha \in \Phi^+ \cap w(\Phi^- \setminus -\langle \Sigma \rangle)} \alpha = b_{w'}S_{w'}^P(w') $ for some $b_{w'} \in H^*(BT)$.
Then $f-b_{w'}S_{w'}^P$ vanishes at $w'$ and $\{ w \in W/W_P \suchthat f(w)=0 \}$.
By induction we can write $f$ as a $H^*(BT)$-linear combination of $S_w^P$'s.
\end{proof}

\subsection{GKM fiber bundles and Leray-Hirsch theorem}
\label{Leray-Hirsch}
Let $\Sigma_1 \subset \Sigma_2$ be subsets of $\Delta$
and $P_1 \subset P_2$ the corresponding parabolic subgroups containing the Borel subgroup $B$.
The maximal torus $T$ acts on the quotient space $P_2/P_1$ by the left multiplication
and $P_2/P_1$ is isomorphic to $G'/P'$,
where $G'$ is the Levi subgroup of $P_2$ corresponding to $\Sigma_2$
and $P'$ is the parabolic subgroup of $G'$ corresponding to $\Sigma_1$.
By this fact, let $\G(P_2/P_1)$ denote the GKM graph $\G(G'/P')$.
Note that $\G(P_2/P_1)$ is identified with the induced GKM subgraph of $\G(G_\C/P_1)$ on $W_{P_2}$
and that
the natural inclusion $i \colon \G(P_2/P_1) \to \G(G_\C/P_1)$ and the natural projection $p \colon \G(G_\C/P_1) \to \G(G_\C/P_2)$
are GKM graph homomorphisms.

More generally,
let $\Phi_1 \subset \Phi_2$ be subsets of the root system $\Phi$
with $\Phi_i = \langle \Phi_i \rangle$
and $W_i$ the subgroup of $W=W(G)$
generated by the reflections associated with roots in $\Phi_i$ for $i=1,2$.
We define the GKM graph $\G(W_2/W_1)$ similarly to $\G(P_2/P_1)$.
The vertex set is $W_2/W_1$.
For $v$, $w \in W_2$ such that $vW_1 \neq wW_1$,
there is an edge between $vW_1$ and $wW_1$ labelled by $(\alpha)$
if and only if $v = \sigma_\alpha w$.
Moreover, for any $\alpha \in \Phi$, there is at most one such edge between $vW_1$ and $wW_1$.
To accord with our definition of GKM graph homomorphisms,
we will treat $\G(W_2/W_1)$ as if it has a loop labelled by the ideal $\{ 0 \} \subset H^*(BT)$ at each vertex.
The natural inclusion $i \colon \G(W_2/W_1) \to \G(W/W_1)$ and the natural projection $p \colon \G(W/W_1) \to \G(W/W_2)$
are also GKM graph homomorphisms.
In this situation, let $\F$, $\G$, and $\B$ denote the GKM graphs $\G(W_2/W_1)$, $\G(W/W_1)$, and $\G(W/W_2)$, respectively.
This sequence $\F \to \G \to \B$ of GKM graph homomorphisms is called a \textit{GKM fiber bundle}.
The term ``GKM fiber bundle'' was introduced by Guillemen, Sabatini, and Zara in \cite{GSZ}
and our definition is a special case of theirs
except that our definition allows multiple edges.

From now suppose that there is a sequence of vertices $v_1 = e$, $v_2$, \ldots, $v_N$ of $\B$
and there is a sequence of GKM functions $\varphi_1=1$, $\varphi_2$, \ldots, $\varphi_N$ on $\B$
which satisfy the following condition;
all $v_i$'s are distinct,
$\{ v_i \suchthat i \in [N] \} = V(\B)$, 
\[
\varphi_i(v_j)=0 \ (j < i), \text{ and } \quad \varphi_i(v_i) = \prod \alpha,
\]
where the product is taken over all roots $\alpha$ appearing as the labels of edges between $v_i$ and $v_j$ $(j < i)$.
Let $A_i$ denote the product $\varphi_i(v_i) = \prod \alpha$.
Since all labels are of degree $2$ and $\G(W/W_1)$ is connected,
for $i>1$, $\varphi_i(v_i)-\varphi_i(v_1)$ must be of positive degree.
Hence, for $i>1$, $\varphi_i$ is of degree at least $2$.
These sequences are needed for the induction
in the proof of Theorem \ref{generators} and \ref{relations}.

\begin{rem}
We can take such a sequence of vertices
and a sequence of GKM functions
for the GKM fiber bundle $\G(P_2/P_1) \to \G(G_\C/P_1) \to \G(G_\C/P_2)$.
Indeed, any sequence of vertices of $\G(G_\C/P_2)$ such that its image by $l^{P_2}$ is non-decreasing
and the sequence of the corresponding equivariant Schubert classes $S_{v_i}^{P_2}$'s
satisfy the conditions.
\end{rem}

\begin{ex}[{cf. \cite[Section $3$ and $4$]{S}}]
Let $G=F_4$, $\Phi_1=\emptyset$, and
\[
\Phi_2=\Phi(\mathrm{Spin}(9)) \subset
\Phi(F_4)=\{\pm(t_i+t_j), \pm(t_i-t_j), \pm t_k, \frac{1}{2}(\pm t_1 \pm t_2 \pm t_3 \pm t_4) \suchthat 1\leq i < j \leq 4,\ 1 \leq k \leq 4\}
\]
There is an element $\rho \in W(F_4)$ such that $\rho$, $\rho^2$, $\rho^3 = e$
form a complete system of representatives
for the left cosets of $W(\mathrm{Spin}(9))$ in $W(F_4)$ and
$\{ \rho^\varepsilon t_i \suchthat \varepsilon =1,2, \ 1\leq i\leq 4\}
=\{ \frac{1}{2}(\pm t_1 \pm t_2 \pm t_3 \pm t_4) \}$.
In particular, we have $\prod_{i=1}^4 t_i + \prod_{i=1}^4 \rho t_i +\prod_{i=1}^4 \rho^2 t_i = 0$.
Then the base graph $\B$ has only three vertices.
There are $4$ edges between $e$ and $\rho$,
and the set of their labels are $\{ \rho^2 t_i \suchthat 1\leq i\leq 4\}$.
The labels of other edges are easily calculated by multiplying $\rho$.
The following GKM function $\omega$ on $\B$ generates $H^*(\B)$.
\begin{equation}
\label{ataiomega}
\omega(\rho^\varepsilon)=
\begin{cases}
0					& \varepsilon = 0,\\
-\prod_{i=1}^4 \rho^2 t_i	& \varepsilon = 1,\\
\prod_{i=1}^4 \rho t_i		& \varepsilon = 2.
\end{cases}
\end{equation}
Put $v_1 = e$, $v_2 = \rho$, $v_3 = \rho^2$,
$\varphi_1=1$, $\varphi_2 = \omega$, and $\varphi_3 = \omega(\omega +\prod_{i=1}^4 \rho^2 t_i)$.
They satisfy the above condition.
\end{ex}

Theorem \ref{generators} and \ref{relations} show that
generators of $H^*(\G)$ and their relations
come from those of $H^*(\F)$ and $H^*(\B)$.

\begin{thm}
\label{generators}
Assume that $\Phi^+$ is pairwise relatively prime in $H^*(BT)$
and that $\{g_\lambda\}_{\lambda \in \Lambda}$ generates $H^*(\F)$ as an $H^*(BT)$-algebra.
Then the cohomology $H^*(\G)$ is generated by lifts $\{\tilde{g}_\lambda\}_{\lambda \in \Lambda}$ as an $H^*(\B)$-algebra.
In particular, $H^*(\G)$ is generated by $\{ \varphi_i\}_{i \in [N]}$ and $\{\tilde{g}_\lambda\}_{\lambda \in \Lambda}$ as an $H^*(BT)$-algebra.
\end{thm}

We will prove this theorem by considering GKM functions on each fiber.
For a vertex $v \in W$, we define an $H^*(BT)$-algebra homomorphism $\mathfrak{r}_v \colon H^*(\G) \to H^*(v\F)$,
which restricts GKM functions on $\G$ to ones on $v\F$.

\begin{lem}
\label{vg_i}
For any $v \in W$,
the cohomology $H^*(v\F)$ is generated by
$\{ \mathfrak{r}_v (v \cdot \tilde{g}_\lambda) \}_{\lambda \in \Lambda}$ as an $H^*(BT)$-algebra.
\end{lem}

\begin{proof}
Let $v \cdot g_\lambda$ denote $\mathfrak{r}_v (v \cdot \tilde{g}_\lambda)$.
Fix any GKM function $f$ on $v\F$.
Since $v^{-1}\cdot f$ is a GKM function on $\F$,
it can be expressed as a polynomial in $\{g_\lambda\}_{\lambda \in \Lambda}$.
Hence, acting $v\cdot$ on them, $f$ is expressed as a polynomial in  $\{ v \cdot g_\lambda \}_{\lambda \in \Lambda}$.
\end{proof}

\begin{proof}[Proof of Theorem \ref{generators}]
We prove Theorem \ref{generators} by induction on the degree $2n$ of a GKM function $f$ on $\G$.
When $n=0$, $f$ is a constant function,
hence Theorem \ref{generators} holds for GKM functions of degree $0$.
Assume that Theorem \ref{generators} holds for GKM functions of degree at most $2(n-1)$.
Fix any homogeneous GKM function $f$ of degree $2n$
and put $f_1=f$.
If the following claim holds for any $j \leq N \ (= \# V(\B))$,
then we will obtain $f_{N+1} = 0$.
\begin{claim}[$j$]
If a GKM function $f_j$ on $\G$ vanishes on $v_i\F$ for $i<j$,
then there is a GKM function $\tilde{F}_j$ generated by $\{\tilde{g}_\lambda\}_{\lambda \in \Lambda}$ over $H^*(\B)$
such that $f_{j+1} = f_j - \varphi_j \tilde{F}_j$ vanishes on $v_j \F$.
\end{claim}
First let us show that this claim holds in the case where $j=1$.
Since $\mathfrak{r}_e(f_1)$ is a GKM function on $\F$,
it can be expressed as a polynomial in  $\{ g_\lambda\}_{\lambda \in \Lambda}$.
Replacing $g_\lambda$ with $\tilde{g}_\lambda$,
we obtain a GKM function $\tilde{F}_1$ on $\G$ such that $f_1 - \tilde{F}_1$ vanishes on $\F$.
Then $f_2 = f_1 - \varphi_1 \tilde{F}_1$ vanishes on $v_1 \F = \F$.

Next let us show that the claim holds in the case where $j>1$.
Since the set of the positive roots of $G$ is pairwise relatively prime in $H^*(BT)$,
by the GKM condition, $f_j(v) \in (A_j)$ for any $v \in v_j\F$.
Let $V(v_j \F)$ denote the vertex set of $v_j \F$
and $F_j$ denote the set function $V(v_j \F) \to H^*(BT)$ defined by $f_j(v) = F_j(v) A_j$.
Since $\Phi^+$ is pairwise relatively prime
and any factor $\alpha \in \Phi$ of $A_j$ does not appear as the label of edge of $v_j \F$,
the GKM condition of $f_j$ on $v_j \F$
requires that $F_j$ is a GKM function on $v_j \F$.
Hence $F_j$ can be expressed as a polynomial in
$\{ v_j \cdot g_\lambda \suchthat |g_\lambda| \leq 2(n-1)\}$
by Lemma \ref{vg_i}
since $A_j$ is at least of degree 2.
Replacing $v_j \cdot g_\lambda$ with $v_j \cdot \tilde{g}_\lambda$,
we obtain a GKM function $\tilde{F}_j$ on $\G$,
which is a polynomial in $\{g_\lambda\}_{\lambda \in \Lambda}$ over $H^*(\B)$
by the induction on $n$.
By construction we have $\mathfrak{r}_{v_j}(\tilde{F}_j) = F_j$.
Hence $f_{j+1} = f_j - \varphi_j \tilde{F}_j$ vanishes on $v_j \F$.

We can verify that the last assertion holds as follows.
Let us regard a GKM function $f$ on $\B$ as a GKM function on $\G$
which is constant on each fiber.
For such GKM function $f$ on $G$,
the above $F_j$ is a constant GKM function on $v_j \F$.
Therefore $H^*(\B)$ is generated by $\{ \varphi_i\}_{i \in [N]}$.
\end{proof}

Let $\{ x_\lambda \}_{\lambda \in \Lambda}$ and $\{ y_i\}_{i \in [N]}$
be the sets of indeterminates with $|x_\lambda | = |g_\lambda |$ and $|y_i|=|\varphi_i|$.
We define graded $H^*(BT)$-algebra homomorphisms $\Xi_\F$, $\Xi_\G$, $\Xi_\B$, and $\iota_v$
for a vertex $v \in \B$ as follows.
\begin{align*}
\Xi_\F & \colon H^*(BT)[x_\lambda \suchthat \lambda \in \Lambda] \to H^*(\F), && x_\lambda \mapsto g_\lambda, \\
\Xi_\G & \colon H^*(BT)[x_\lambda, y_i \suchthat \lambda \in \Lambda, i \in [N]] \to H^*(\G),
				&& x_\lambda \mapsto \tilde{g}_\lambda, \ y_i \mapsto p^*(\varphi_i), \\
\Xi_\B & \colon H^*(BT)[y_i \suchthat i \in [N]] \to H^*(\B), && y_i \mapsto \varphi_i, \\
\iota_v  & \colon H^*(BT)[x_\lambda, y_i \suchthat \lambda \in \Lambda, i \in [N]] \to H^*(BT)[x_\lambda \suchthat \lambda \in \Lambda],
				&& x_\lambda \mapsto x_\lambda, \ y_i \mapsto \varphi_i(v)
\end{align*}
Let $I_\F$, $I_\G$, and $I_\B$ denote the kernels of $\Xi_\F$, $\Xi_\G$, and $\Xi_\B$, respectively.
Then we have
\begin{align*}
H^*(\F) & = H^*(BT)[x_\lambda \suchthat \lambda \in \Lambda]/I_\F, \\
H^*(\G) & = H^*(BT)[x_\lambda, y_i \suchthat \lambda \in \Lambda, i \in [N]]/I_\G,\\
H^*(\B) & = H^*(BT)[y_i \suchthat i \in [N]]/I_\B.
\end{align*}

\begin{thm}
\label{relations}
Assume that $\Phi^+$ is pairwise relatively prime in $H^*(BT)$,
that $\{ r_{\theta} \}_{\theta \in \Theta}$ generates the ideal $I_\F$,
and that there are lifts $\{\tilde{r}_{\theta}\}_{\theta \in \Theta}$
through $\iota_e$ contained in $I_\G$.
Then the ideal $I_\G$ is generated by $\{\tilde{r}_{\theta}\}_{\theta \in \Theta}$ and $I_\B$.
\end{thm}

For brevity, we abbreviate the range of variables if it is obvious.
For example, we write $H^*(BT)[x_\lambda]$ for $H^*(BT)[x_\lambda \suchthat \lambda \in \Lambda]$.
Let us define some operators on $H^*(BT)[x_\lambda,y_i]$
corresponding to the action of $W$ on $H^*(\G)$ as follows.
For any $\alpha \in \Phi$,
by the definition of $\delta_\alpha$, we have
$\sigma_\alpha \cdot \tilde{g}_\lambda = \tilde{g}_\lambda - \alpha \delta_\alpha \tilde{g}_\lambda$.
Hence, for any $v \in W$,
\begin{equation}
\label{vg=g}
v \cdot \tilde{g}_\lambda \in \tilde{g}_\lambda + H^2(BT)H^{2(n-1)}(\G),
\end{equation}
since the Weyl group $W$ is generated by the reflections $\sigma_\alpha$'s.
For any $v \in W$ and $\lambda \in \Lambda$,
by \eqref{vg=g},
we have $v \cdot \tilde{g}_\lambda = \tilde{g}_\lambda + g_{\lambda,v}$,
where $g_{\lambda,v}$ is a polynomial in $\{ \varphi_j \suchthat |\varphi_j| < |\tilde{g}_\lambda|\}$ and
$\{ \tilde{g}_{\lambda'} \suchthat |\tilde{g}_{\lambda'}| < |\tilde{g}_\lambda|\}$ over $H^*(BT)$.
For any $i \in [N]$, in the same way we can see that
$v \cdot \varphi_i = \varphi_i + \varphi_{i,v}$,
where $\varphi_{i,v}$ is a polynomial in $\{ \varphi_j \suchthat |\varphi_j| < |\varphi_i|\}$ over $H^*(BT)$.
Put $g_{\lambda,e} = 0$ and $\varphi_{i,e} = 0$.
Fix $g_{\lambda,v}$'s and $\varphi_{i,v}$'s and replace $\tilde{g}_\lambda$ by $x_\lambda$ and $\varphi_i$ by $y_i$.
Then, for any $v \in W$,
we can define $v \cdot x_\lambda \in (H^*(BT)[y_i])[x_\lambda]$ and $v \cdot y_i \in H^*(BT)[y_i]$
and then a ring homomorphism
\[
v \cdot \colon H^*(BT)[x_\lambda,y_i] \to H^*(BT)[x_\lambda,y_i],
\quad x_\lambda \mapsto v \cdot x_\lambda, \ y_i \mapsto v \cdot y_i
\]
($v$ acts on $H^*(BT)$ ordinary).
Moreover we can define $\bar{v} \cdot$ to be the inverse.

\begin{lem}
\label{gr_i}
For any $f \in H^*(BT)[x_\lambda,y_i]$ such that $\Xi_\G(f)$ vanishes on $v \F$,
there exists an element $\tilde{f}$ of $H^*(BT)[x_\lambda,y_i]$
contained in the ideal $(\{ v \cdot \tilde{r}_\theta\}_{\theta \in \Theta})$
and satisfies that $\iota_v (f-\tilde{f}) =0$.
\end{lem}

\begin{proof}
Fix any $f \in H^*(BT)[x_\lambda,y_i]$ such that $\Xi_\G(f)$ vanishes on $v \F$.
Since $\Xi_\G(\bar{v} \cdot f)$ vanishes on $\F$,
$\iota_e (\bar{v} \cdot f) \in H^*(BT)[x_\lambda]$
is contained in the ideal $(\{ r_\theta\}_{\theta \in \Theta})$.
Replacing $r_\theta$ with $\tilde{r}_\theta$,
we obtain $f'$ which satisfies $\iota_e (\bar{v} \cdot f - f') = 0$.
This equation means that $\Xi_\B$ sends each coefficient of the polynomial
$\bar{v} \cdot f - f' \in (H^*(BT)[y_i])[x_\lambda]$ to a GKM function vanishing at $e \in \B$.
Therefore $\tilde{f} = v \cdot f'$ satisfies $\iota_v (f-\tilde{f}) =0$.
\end{proof}

\begin{proof}[Proof of Theorem \ref{relations}]
We prove Theorem \ref{relations} by induction on the degree $2n$ of $f \in I_\G$.
When $n=0$, we have $f \in H^0(BT)$,
and then $f=0$.
Assume that Theorem \ref{relations} holds for GKM functions of degree at most $2(n-1)$.
Fix any $f \in I_\G$ of degree $2n$ and put $f_1=f$.
If the following claim holds for any $j \leq N$,
then we will obtain $f_{N+1} (\equiv f \bmod (\{\tilde{r}_\theta\}_{\theta \in \Theta})+I_\B )$ which is contained in $I_\B H^*(BT)[x_\lambda]$.
Hence $f$ is contained in the ideal generated by $\{\tilde{r}_\theta\}_{\theta \in \Theta}$ and $I_\B$,
and the induction on $n$ proceeds.
\begin{claim}[$j$]
Let $f_j = \sum_m k_m^{(j)} x_m^{(j)} \in H^*(BT)[x_\lambda,y_i]$,
where $x_m^{(j)}$ is a monic monomial in $\{x_\lambda\}_{\lambda \in \Lambda}$ and $k_m^{(j)} \in H^*(BT)[y_i]$.
If $\Xi_\G(k_m^{(j)})$ vanishes on $v_i \F$ for any $m$ and $i<j$,
then there is an element $\tilde{F}_j$ contained in the ideal
generated by $\{\tilde{r}_\theta\}_{\theta \in \Theta}$ and $I_\B$
such that $\Xi_\G$ sends each coefficient of $f_{j+1} = f_j - y_j \tilde{F}_j = \sum_m k_m^{(j+1)} x_m^{(j+1)}$
to a GKM function which vanishes on $v_j \F$.
\end{claim}

First let us show that this claim holds in the case where $j=1$.
Since $\Xi_\G(f_1)=0$,
by Lemma \ref{gr_i}, there exists an element $\tilde{F}_1$ of $H^*(BT)[x_\lambda,y_i]$,
which is contained in the ideal $(\{ \tilde{r}_\theta\}_{\theta \in \Theta})$
and satisfies that $\iota_e (f_1-\tilde{F}_1) =0$.
Put $f_2 = f_1 - y_1 \tilde{F}_1$, then $\iota_e(f_2)$ vanishes on $v_1 \F = \F$.

Next let us show that the claim holds in the case where $j>1$.
Since the set of the positive roots of $G$ is pairwise relatively prime in $H^*(BT)$
and $\Xi_\B(k_m^{(j)})$ is a GKM function on $\B$,
by the GKM condition, there exists $K_m^{(j)} \in H^*(BT)$ such that
$\Xi_\B(k_m^{(j)})(v_j) = K_m^{(j)} A_j$.
Put $F_j = \sum_m K_m^{(j)} x_m^{(j)}$.
Since $\Xi_\G(F_j)$ vanishes on $v_j \F$,
by Lemma \ref{gr_i},
there exists an element $\tilde{F}_j$ of $H^*(BT)[x_\lambda,y_i]$,
which is contained in the ideal
$(\{ v_j \cdot \tilde{r}_\theta\}_{\theta \in \Theta})$
and satisfies that $\iota_{v_j} (F_j-\tilde{F}_j) =0$.
In particular, $\tilde{F}_j$ is contained in the ideal
$(\{ v_j \cdot \tilde{r}_\theta \suchthat |\tilde{r}_\theta| \leq 2(n-1)\})$
since $A_j$ is at least of degree 2.
By the induction on $n$, $\{ v_j \cdot \tilde{r}_\theta \in H^*(BT)[x_\lambda,y_i] \suchthat |\tilde{r}_\theta| \leq 2(n-1)\}$
is contained in the ideal generated by $\{ \tilde{r}_\theta\}_{\theta \in \Theta}$ and $I_\B$.

Finally we verify that $f_{N+1}$ is contained in $I_\B H^*(BT)[x_\lambda ]$.
Since $\iota_{v_j}(f_{N+1}) = 0$, for any $j$,
the homomorphism $\Xi_\B$ sends all coefficients $k_m^{(N+1)}$ to $0 \in H^*(\B)$.
Hence $k_m^{(N+1)}$ is contained in $I_\B$.
\end{proof}

\section{\textsc{The GKM graph of $(E_6)_\C/B$ and its cohomology}}
\label{sec. GKM graph}
\subsection{The root system of $E_6$}
Let $T$ be a maximal torus of $E_6$.
First of all, we choose generators of $H^2(BT)$.
According to \cite[Planche V]{B}, the Dynkin diagram of $E_6$ is
\[
    \begin{xy}
    ( -16, 0)="1"*+[Fo:<2mm>]{},
    (  16,-16)="2"*+[Fo:<2mm>]{},
    (   0, 0)="3"*+[Fo:<2mm>]{},
    (  16, 0)="4"*+[Fo:<2mm>]{},
    (  32, 0)="5"*+[Fo:<2mm>]{},
    (  48, 0)="6"*+[Fo:<2mm>]{},
    ( -16,6)=""*{\alpha_1},
    (   0,6)=""*{\alpha_3},
    (  16,-22)=""*{\alpha_2},
    (  16,6)=""*{\alpha_4},
    (  32,6)=""*{\alpha_5},
    (  48,6)=""*{\alpha_6},
    \ar@{-} "1"+/r2mm/;"3"+/l2mm/,
    \ar@{-} "2"+/u2mm/;"4"+/d2mm/,
    \ar@{-} "3"+/r2mm/;"4"+/l2mm/,
    \ar@{-} "4"+/r2mm/;"5"+/l2mm/,
    \ar@{-} "5"+/r2mm/;"6"+/l2mm/
    \end{xy}
\]
where $\alpha_i$'s are the simple roots of $E_6$.
We can identify the dual $\mathfrak{t}^*$ of the Lie algebra of $T$ with the subspace of $\R^8$
consisting of points $\sum_{i=1}^8 x_i e_i$ such that $x_6=x_7=-x_8$,
where $\{ e_i \}_{i=1}^8$ denotes the standard orthonormal basis of $\R^8$ and $x_i \in \R$.
Then we can describe $\alpha_i$'s as follows;
\begin{align*}
& \alpha_1 = \frac{1}{2}(e_1 -e_2 -e_3 -e_4 -e_5 -e_6 -e_7 +e_8),\\
& \alpha_2 = e_1+e_2, \quad \alpha_i = e_{i-1} -e_{i-2} \ (3 \leq i \leq 6).
\end{align*}
The root system $\Phi(E_6)$ is given as
\begin{align}
\label{Phi(E_6)}
 \Phi(E_6)=
	\left\{
	 \begin{aligned}
	  &\pm(e_i+e_j), \ \pm(e_i-e_j),\\
	  &\pm \frac{1}{2}\biggl(\sum_{i=1}^5 \mu_i e_i + e_8 - e_7 - e_6 \biggr)
	 \end{aligned}
	 \ \Biggl| \ 1\leq i < j \leq 4, \ \mu_k=\pm 1, \ \prod_{k=1}^5 \mu_k = 1 \Biggr.
	\right\}.
\end{align}

According to \cite[Section 4 (B)]{TW}, the following $t_i$'s $(0 \leq i \leq 5)$ and $x$ are generators of $H^2(BT)$
and $H^*(BT) = \Z[x,t_0,t_1,\ldots,t_5]/(3x-c_1)$ for $c_1 = t_0+t_1+\cdots+t_5$.
\begin{align*}
t_0 & = \frac{1}{2}(e_1 +e_2 +e_3 +e_4 +e_5) -\frac{1}{6}(e_8 -e_7 -e_6),\\
t_i & = e_i + \frac{1}{3}(e_8 -e_7 -e_6) \ (1 \leq i \leq 5),\\
x   & = \frac{1}{2}(e_1 +e_2 +e_3 +e_4 +e_5 -e_6 -e_7 +e_8)
\end{align*}
Put $\bar{t} = x - t_0 = \frac{2}{3}(e_8 -e_7 -e_6)$ for later use.

According to \cite[Chapter 8]{A}, there is a maximal rank subgroup $U$ of $E_6$
of local type $D_5\times T^1$ with $D_5 \cap T^1 = \Z_4$,
which is the centralizer of the one dimensional torus $T^1$
defined by 
\[
\alpha_i(t) = 0 \quad (2 \leq i \leq 6, \ t \in T).
\]
The quotient manifold  $E_6/U$ is denoted by $EIII$.
Let $T^5$ be the standard maximal torus of $\mathrm{SO}(10)$
and $t_i'$ the standard basis of $H^2(BT^5)$.
Moreover let $\tilde{T}^5$ be the inverse image of $T^5$ under the universal covering $\mu \colon \mathrm{Spin}(10) \to \mathrm{SO}(10)$
and $\tilde{t}_i$ the image of $t_i'$ by $\mu^*$.
Changing the maximal torus $T_0 = T/T^1$ of $U/T^1=\mathrm{SO}(10)/\Z_2$
by a suitable inner automorphism,
we may regard $T^5$ as the inverse image of $T_0$
under the double covering $\mathrm{SO}(10) \to \mathrm{SO}(10)/\Z_2$.
For a finite set $\mathbf{x} = \{x_1, \ldots, x_n\}$ and $0 \leq i \leq n$,
let $c_i(\mathbf{x})$ denote the $i$-th elementary symmetric polynomial in $x_1, \ldots, x_n$.
According to \cite[(4.11)]{TW}, the natural homomorphism
$\iota_0^* \colon H^2(BT) \to H^2(U/T) \cong H^2(\mathrm{SO}(10)/T^5)
\cong H^2(\mathrm{Spin}(10)/\tilde{T}^5)\cong H^2(B\tilde{T}^5)$
satisfies
\[
\iota_0^*(\bar{t})=0, \ \iota_0^*(x) = \iota_0^*(t_0) = \gamma, \ \iota_0^*(t_i) = \tilde{t}_i \quad \text{for } i=1,2,3,4,5,
\]
where $H^*(B\tilde{T}^5) = \Z[\tilde{t}_i,\gamma \suchthat 1 \leq i \leq 5]/(2\gamma - c_1(\tilde{\mathbf{t}}))$
for $\tilde{\mathbf{t}}=(\tilde{t}_1, \ldots, \tilde{t}_5)$.

\subsection{A decomposition of $W(E_6)$}
From now, let us decompose $\G((E_6)_\C/B)$ into a fiber and a base GKM graph.
We define GKM functions $z_i$'s and $\bar{z}$ on $\G((E_6)_\C/B)$ as follows.
\begin{align*}
z_i(w) = w(t_i), \quad \bar{z}(w) = w(\bar{t}) \quad \text{ for } w \in W(E_6), \ 0 \leq i \leq 5
\end{align*}

Let
$\Lambda = \{\lambda \in H^2(BT) \suchthat \lambda = \bar{z}(w), \: w \in W(E_6)\}$,
and let $D_5$ denote the Lie subalgebra of that of $E_6$
such that its Dynkin diagram is given as
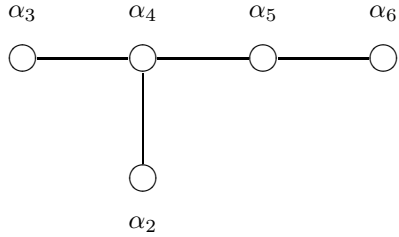
\[
    \begin{xy}
    (  16,-16)="2"*+[Fo:<2mm>]{},
    (   0, 0)="3"*+[Fo:<2mm>]{},
    (  16, 0)="4"*+[Fo:<2mm>]{},
    (  32, 0)="5"*+[Fo:<2mm>]{},
    (  48, 0)="6"*+[Fo:<2mm>]{},
    (   0,6)=""*{\alpha_3},
    (  16,-22)=""*{\alpha_2},
    (  16,6)=""*{\alpha_4},
    (  32,6)=""*{\alpha_5},
    (  48,6)=""*{\alpha_6},
    \ar@{-} "2"+/u2mm/;"4"+/d2mm/,
    \ar@{-} "3"+/r2mm/;"4"+/l2mm/,
    \ar@{-} "4"+/r2mm/;"5"+/l2mm/,
    \ar@{-} "5"+/r2mm/;"6"+/l2mm/
    \end{xy}
\]
Note that, for any $w \in W(D_5)$, we have $\bar{z}(w) = w(\bar{t}) = \bar{t}$.
By straightforward calculations we have
\begin{equation}
\label{Lambda}
\Lambda = \left\{ \left. \bar{t}, \: \frac{1}{2}\biggl(\sum_{i=1}^5 \mu_i e_i\biggr) + \frac{1}{6}(e_8 - e_7 - e_6), \: \pm e_j - \frac{1}{3}(e_8 -e_7 -e_6)
		\: \right| \: \mu_i = \pm 1, \: \prod_{i=1}^5 \mu_i = -1, \: 1 \leq j \leq 5 \right\}.
\end{equation}
For $\lambda \in \Lambda$, let $\rho_\lambda$ be an element of $W(E_6)$ such that $\rho_\lambda (\bar{t}) = \lambda$
and fix it.
For example, we can take them as in Table \ref{rho_lambda} (see Appendix \ref{appendix}).
Since $\# \Lambda = \#(W(E_6)/W(D_5)) (= 27)$,
the set of $\rho_\lambda$'s is a complete system of representatives
for the left cosets of $W(D_5)$ in $W(E_6)$.
Then we have a coset decomposition of $W(E_6)$.
\[
W(E_6) = \coprod_{\lambda \in \Lambda} \rho_\lambda W(D_5)
\]

\subsection{The GKM graph of $EIII$}
Let $P$ be the parabolic subgroup of $(E_6)_\C$ corresponding to $\{\alpha_2,\alpha_3,\alpha_4,\alpha_5,\alpha_6\}$.
Then we have $P/B \cong U/T$.
Consider the GKM graph $\G((E_6)_\C/P)$,
call it the GKM graph of $EIII$, and denote it by $\mathcal{G}(EIII)$.
Let us denote the vertex $\rho_\lambda W(D_5) \in W(E_6)/W(D_5)$ simply by the representative $\rho_\lambda$.
Since $\Phi(E_6) \setminus \Phi(D_5)$ consists of $16$ elements,
16 edges meet at the vertex $\rho_{\bar{t}}$.
For any element of $\Phi(E_6) \setminus \Phi(D_5)$,
the corresponding reflection appears as $\rho_\lambda$ in Table \ref{rho_lambda}.
Hence the other end points of these edges are distinct.
For any $\rho \in W(E_6)$, $\alpha \in \Phi(E_6)$, and two vertices $wW(D_5)$, $w'W(D_5)$ of $\mathcal{G}(EIII)$,
we have $w'W(D_5) = \sigma_\alpha wW(D_5)$ if and only if $\rho w'W(D_5) = \sigma_{\rho(\alpha)} \rho wW(D_5)$.
Hence each vertex of $\G(EIII)$ has $16$ edges
and the other endpoints of these edges are distinct.

If $\rho_\lambda$ and $\rho_{\lambda'}$ are adjacent
and the label of the edge is $(\alpha)$ for some $\alpha \in \Phi^+$,
then we have $\sigma_\alpha \rho_\lambda W(D_5) = \rho_{\lambda'}W(D_5)$.
Hence we obtain $\sigma_\alpha \lambda = \lambda'$, since $\bar{t}$ is a fixed point with respect to the action of $W(D_5)$ on $H^2(BT)$.
By the definition of $\bar{z}$, we have
\[
\bar{z}(\rho_\lambda) - \bar{z}(\sigma_\alpha \rho_\lambda)
= \rho_\lambda(\bar{t}) - \sigma_\alpha \rho_\lambda(\bar{t})
= \lambda - \sigma_\alpha \lambda = \lambda - \lambda'.
\]
On the other hand
\[
\lambda - \sigma_\alpha \lambda = 2 \frac{(\lambda, \alpha)}{(\alpha, \alpha)}\alpha,
\]
and the length of any root $\alpha \in \Phi(E_6)$ is $\sqrt{2}$.
Hence we obtain
\[
\lambda - \lambda' = (\lambda, \alpha)\alpha.
\]
Moreover one can easily see that $(\lambda, \alpha)$ is $1$, $0$, or $-1$ by \eqref{Phi(E_6)} and \eqref{Lambda}.
If $(\lambda, \alpha) = 0$, then $\lambda = \lambda'$ and
it is a contradiction.
If $(\lambda, \alpha) = 1$ (respectively $-1$), then the difference $\bar{z}(\rho_\lambda) - \bar{z}(\sigma_\alpha \rho_\lambda)$
is the root $\alpha$ (respectively $-\alpha$) and the difference is the generator of the label of the edge between $\rho_\lambda W(D_5)$ and $\rho_{\lambda'} W(D_5)$.
In particular, the number of elements of $\{\lambda' \in \Lambda \suchthat \bar{t} - \lambda' \text{ is a root}\}$ is 16
and such $\rho_{\lambda'}$ is adjacent to $\rho_{\bar{t}}$ as in Table \ref{rho_lambda}.
For $\rho$, $w$, $w' \in W(E_6)$, 
we have $w' = \sigma_\alpha w \Leftrightarrow \rho w' = \sigma_{\rho(\alpha)} \rho w$
and $\rho$ permutes the roots of $E_6$.
Then, for any $\lambda \in \Lambda$,
the number of elements of $\{\lambda' \in \Lambda \suchthat \lambda - \lambda'  \in \Phi(E_6)\}$ is 16
and such $\rho_{\lambda'}$ is adjacent to $\rho_\lambda$.
By the above argument we have obtained the following proposition.
\begin{prop}
\label{adjacent}
In the GKM graph of $EIII$,
$\rho_\lambda$ and $\rho_{\lambda'}$ are adjacent
if and only if $\lambda - \lambda' = \alpha$ for some $\alpha \in \Phi(E_6)$.
The label of the corresponding edge is $(\alpha)$.
\end{prop}

\subsection{The cohomology of $\G(EIII)$}
First, let us consider $H^*(BT)$-algebra generators of $H^*(\G(EIII))$.
According to \cite[Corollary C]{TW},
one element of degree $2$ and one element of degree $8$ generate $H^*(EIII)$.
Thus we need corresponding GKM functions on $\G(EIII)$
since $H^*(\G(EIII)) \cong H^*_T(EIII)  \to H^*(EIII)$ is surjective.
First, let us define a GKM function of degree $2$.
Since $\bar{t}$ is a fixed point with respect to the action of $W(D_5)$,
$\tau = \bar{t} - \bar{z}$ can be regarded as a GKM function on $\G(EIII)$.
Next we define GKM function $\omega \colon V(\mathcal{G}(EIII)) \to H^*(BT)$
of degree $8$ as in Table \ref{omega} (see Appendix \ref{appendix}).
By straightforward calculations,
we can see that $\omega$ is a GKM function.
In this subsection, we will prove that $\tau$ and $\omega$ generate $H^*(\G(EIII))$ as an $H^*(BT)$-algebra
and reveal their relations.

Let $\mathbf{t}$ denote $\{ t_0, \ldots, t_5\}$
and $\mathbf{t}_{\geq i}$ denote the subset $\{ t_i, t_{i+1}, \ldots, t_5 \}$.
Moreover let us define four GKM functions $\omega_6$, $\omega_7$, $\omega_8$, and $\omega_{12}$,
which are polynomials in $\tau$ and $\omega$ over $H^*(BT)$, as follows:
\begin{align*}
\omega_6 =& \: -\prod_{\lambda} (\bar{z} - \lambda) +(3\tau^2 -(c_1(\mathbf{t}_{\geq 3}) -6\bar{t})\tau\\
	    & \: +(-2\bar{t}^2+(+2t_2+2t_1-5t_0)\bar{t}+c_2(\mathbf{t}_{\geq 3}) +t_0(t_2+t_1-t_0))\omega,\nonumber\\
\omega_7 =& \: 2(\tau +t_0 -t_1 -t_2 -\bar{t})\prod_{\lambda} (\bar{z} - \lambda) -(5\tau^3 +3(c_1(\mathbf{t}_{\geq 3}) -6\bar{t})\tau^2\\
	    & \: -(t_0^2 +t_0\bar{t} -(t_1+t_2)t_0 +20\bar{t}^2 -6\bar{t}c_1(\mathbf{t}_{\geq 3}) +c_2(\mathbf{t}_{\geq 3}))\tau \nonumber\\
          & \: -(2\bar{t}-t_3)(2\bar{t}-t_4)(2\bar{t}-t_5))\omega,\nonumber\\
\end{align*}
\begin{align*}
\omega_8 =& \: \omega^2 + (2\tau^4-(5\bar{t} +5t_1 -3t_0)\tau^3+(4\bar{t}^2+(9t_1-7t_0)\bar{t}+3t_1^2-6t_0t_1+t_0^2)\tau^2 \\
	    & \: +(-56\bar{t}^3+(32c_1(\mathbf{t}_{\geq 2})-20t_0)\bar{t}^2+(-4c_1(\mathbf{t}_{\geq 2}^2)-10c_2(\mathbf{t}_{\geq 2})10t_0^2)\bar{t} \nonumber \\
	    & \: +2t_0c_1(\mathbf{t}_{\geq 2}^2)+c_3(\mathbf{t}_{\geq 2})+3t_0c_2(\mathbf{t}_{\geq 2})-5t_0^2c_1(\mathbf{t}_{\geq 2})+3t_0^3)\tau \nonumber \\
          & \: +(t_1-t_0)(8\bar{t}^3+4(t_0-c_1(\mathbf{t}_{\geq 2}))\bar{t}^2+2(c_2(\mathbf{t}_{\geq 2})\bar{t}-t_0c_1(\mathbf{t}_{\geq 2})+t_0^2)\bar{t} \nonumber \\
          & \: -c_3(\mathbf{t}_{\geq 2})+t_0c_2(\mathbf{t}_{\geq 2})-t_0^2c_1(\mathbf{t}_{\geq 2})+t_0^3))\omega \nonumber \\
	    & \: +(-\tau^2+(\bar{t}+2t_1-t_0)\tau+(t_0-t_1)(\bar{t}+t_1))\prod_{\lambda} (\bar{z} - \lambda), \nonumber \\
\omega_{12}=& \: \omega_8 \biggl({\displaystyle r \Bigl(\frac{1}{2}\sum_{i=1}^5 e_i + \frac{1}{2}(e_8 - e_7 - e_6)\Bigr)} \cdot \omega \biggr),
\end{align*}
where the products are taken over $\{\lambda = \bar{t},\ -e_n +\frac{1}{2}\sum_{i=1}^5 e_i + \frac{1}{6}(e_8-e_7-e_6) \suchthat 1\leq n \leq 5\}$.
By evaluation at each vertex of $\G(EIII)$,
we can see that the following equation holds:
\begin{align*}
  & \: r \Bigl(\frac{1}{2}\sum_{i=1}^5 e_i + \frac{1}{2}(e_8 - e_7 - e_6)\Bigr) \cdot \omega \\
= & \: \omega +(\bar{t}+t_0)(2\tau^3 +(+c_2(\mathbf{t})-3(\bar{t}+t_0)^2)\tau +(c_3(\mathbf{t})-2c_2(\mathbf{t})(\bar{t}+t_0)+5(\bar{t}+t_0)^3)).
\end{align*}
In fact, it is sufficient to evaluate them only at $\rho_\lambda$
for $\lambda = {\displaystyle -\frac{1}{2}\sum_{i=1}^5 e_i + \frac{1}{4}\bar{t}}$, ${\displaystyle -e_j - \frac{1}{2}\bar{t}}$ ($1\leq j \leq 5$)
and $ \frac{1}{2}(-e_1 -e_2 -e_3 +e_4 +e_5)+ \frac{1}{4}\bar{t}$
because of the uniqueness of the Schubert class $S^P_{\rho_\lambda}$ for $\lambda = \frac{1}{2}(-e_1 -e_2 -e_3 +e_4 +e_5)+ \frac{1}{4}\bar{t}$.

By definition, we can easily see that
$\omega_6$, $\omega_7$, $\omega_8$, and $\omega_{12}$ vanish on
\[
V_4 = \biggl\{\rho_\lambda \ \Bigl| \Bigr. \ \bar{t}, \ -e_n +\frac{1}{2}\sum_{i=1}^5 e_i + \frac{1}{4}\bar{t}, \ 1 \leq n \leq 5 \biggr\}.
\]
By straightforward calculations,
we can see that\\
$\omega_6$ vanishes on $V_6 = V_4 \cup \biggl\{ \rho_\lambda \ \Bigl| \Bigr. \ \lambda = {\displaystyle -e_1 -e_2 -e_n +\frac{1}{2}\sum_{i=1}^5 e_i + \frac{1}{4}\bar{t}},\ 3\leq n \leq 5 \biggl\}$,\\
$\omega_7$ vanishes on $V_7 = V_6 \cup \biggl\{ \rho_\lambda \ \Bigl| \Bigr. \ \lambda = {\displaystyle e_j - \frac{1}{2}\bar{t}}, \ 3\leq j \leq 5 \biggl\}$,\\
$\omega_8$ vanishes on $V_8 = V_7 \cup \biggl\{ \rho_\lambda \ \Bigl| \Bigr. \ \lambda = {\displaystyle e_2 +e_n -\frac{1}{2}\sum_{i=1}^5 e_i + \frac{1}{4}\bar{t}, \ -e_1 - \frac{1}{2}\bar{t}, \ e_2 - \frac{1}{2}\bar{t}},\ 3\leq n \leq 5 \biggl\}$,\\
and $\omega_{12}$ vanishes on $V_{12} = V_8 \cup \biggl\{ \rho_\lambda \ \Bigl| \Bigr. \ \lambda = {\displaystyle -\frac{1}{2}\sum_{i=1}^5 e_i + \frac{1}{4}\bar{t}, \ -e_j - \frac{1}{2}\bar{t}},\ 2\leq j \leq 5 \biggl\}$.\\
Moreover
\begin{align*}
\omega_6(\rho_\lambda) &= \prod_{ \substack{ \lambda' \in V_6 \\ \lambda-\lambda' \in \Phi }}(\lambda-\lambda') \quad \text{ for }\rho_\lambda \in V_7 \setminus V_6,\\
\omega_7(\rho_\lambda) &= \prod_{ \substack{ \lambda' \in V_7 \\ \lambda-\lambda' \in \Phi }}(\lambda-\lambda') \quad \text{ for }\rho_\lambda \in V_8 \setminus V_7,\\
\omega_8(\rho_\lambda) &= \prod_{ \substack{ \lambda' \in V_8 \\ \lambda-\lambda' \in \Phi }}(\lambda-\lambda') \quad \text{ for }\rho_\lambda \in V_{12} \setminus V_8,\\
\omega_{12}(\rho_\lambda) &= \prod_{ \substack{ \lambda' \in V_{12} \\ \lambda-\lambda' \in \Phi }}(\lambda-\lambda') \quad \text{ for } \rho_\lambda \in V(\G(EIII)) \setminus V_{12}.\\
\end{align*}
By Proposition \ref{adjacent}, we can see that
the subgraphs of $\G(EIII)$ induced by $V_4$, $V_6 \setminus V_4$, $V_7 \setminus V_6$, $V_8 \setminus V_7$, $V_{12} \setminus V_8$, and $V(\G(EIII)) \setminus V_{12}$
are complete graphs.

\begin{thm}
\label{generators of EIII}
The cohomology ring $H^*(\mathcal{G}(EIII))$ is generated by $\tau$ and $\omega$ as an $H^*(BT)$-algebra.
\end{thm}

\begin{proof}
It is sufficient to show that the cohomology ring $H^*(\mathcal{G}(EIII))$ is generated
by $\tau$, $\omega$, $\omega_6$, $\omega_7$, $\omega_8$, and $\omega_{12}$ as an $H^*(BT)$-algebra.
The proof is similar to the proof of Proposition \ref{Sc. basis}.
Assume that we have a sequence $(v_1, \ldots, v_{27})$ of vertices of $\G(EIII)$
and a sequence $(\varphi_1=1, \ldots, \varphi_{27})$ of GKM functions on $\G(EIII)$ such that
\[
\{v_1, \ldots, v_{27}\}= V(\G(EIII)), \quad \varphi_j(v_i) = 0 \ (i < j), \quad \varphi_j(v_j)= \prod \alpha(e),
\]
where $\alpha$ is the axial function $E(\G(EIII)) \to H^*(BT)$ and
the product is taken over all edges $e$ between $v_i$ and $v_j$ for some $i < j$.
We show that any GKM function $f$ on $\G(EIII)$ is a linear combination of $\varphi_j$'s over $H^*(BT)$ by induction.
The GKM function $f-f(v_1)\varphi_1$ vanishes on $v_1$.
Assume that $f$ vanishes on $\{ v_i \suchthat i < j \}$.
Since the labels of the edges meeting at a vertex are distinct
and $\Phi^+$ is pairwise relatively prime,
the GKM condition says that $f(v_j)$ is a multiple of $\varphi_j(v_j)$,
that is, $f(v_j) = F_j\varphi_j(v_j)$ for some $F_j \in H^*(BT)$.
Hence $f-F_j\varphi_j$ vanishes on $\{ v_i \suchthat i \leq j \}$.
Therefore the induction proceeds, and
Table \ref{EIII seq.} shows the existence of such sequences of vertices and GKM functions
(see Appendix \ref{appendix}).
\end{proof}
By Theorem \ref{generators of EIII},
we have a surjective $H^*(BT)$-algebra homomorphism
\[
\Xi \colon H^*(BT)[\tau, \omega] \to H^*(\G(EIII)).
\]

Next let us consider relations of $\tau$, and $\omega$.
To find them
we need the following results of Toda and Watanabe.
\begin{prop}[{\cite[Corollary C]{TW}}]
\label{H^*(EIII)}
$H^*(EIII) = \Z[\tau', \omega']/(\tau'^9-3\omega'^2\tau', \omega'^3+15\omega'^2\tau'^4-9\omega'\tau'^8)$.
\end{prop}
Recall that, by Theorem \ref{HHH}, we have 
\[
H^*_T(EIII) \cong H^*(\G(EIII)).
\]
By Proposition \ref{H^*(EIII)},
we expect $H^*(\G(EIII))$ has similar relations.
Indeed, we can obtain them as follows.
Consider two linear combinations
of all monomials in $t_i$'s, $x$, $\tau$, and $\omega$ of degree $18$ or $24$ separately,
evaluate them at all vertex of $EIII$,
assume that all the values are $0$,
and solve the simultaneous linear equations.
\begin{prop}
\label{relations of EIII}
The kernel of $\Xi \colon H^*(BT)[\tau, \omega] \to H^*(\G(EIII))$ is generated by $r_9$ and $r_{12}$
(see Appendix \ref{appendix}).
\end{prop}

By Proposition \ref{relations of EIII},
the natural surjection
\[
\Xi \colon H^*(BT)[\tau, \omega] \to H^*(\G(EIII))
\]
factors through
\[
H^*(BT)[\tau, \omega]/(r_9, r_{12}) \to H^*(\G(EIII)).
\]

Finally, we will show that the above $H^*(BT)$-algebra homomorphism
\[
H^*(BT)[\tau, \omega]/(r_9, r_{12}) \to H^*(\G(EIII))
\]
is an isomorphism.
Consider the natural surjection
\[
H^*(BT)[\tau, \omega]/(r_9, r_{12}) \to \Z[\tau, \omega]/(\bar{r}_9, \bar{r}_{12}),
\]
where $\bar{r}_9 = 2\tau^9 - 6\tau^5\omega + 3\tau\omega^2$ and
$\bar{r}_{12} = \omega^3 -18\tau^4\omega^2 +24\tau^8\omega -7\tau^{12}$.
\begin{prop}
\label{isom. of EIII}
We have $\Z[\tau, \omega]/(\bar{r}_9, \bar{r}_{12}) \cong H^*(EIII)$.
\end{prop}

\begin{proof}
By Proposition \ref{H^*(EIII)}, we have
\[
H^*(EIII) = \Z[\tau', \omega']/(\tau'^9-3\omega'^2\tau', \omega'^3+15\omega'^2\tau'^4-9\omega'\tau'^8).
\]
We define a ring homomorphism $f \colon \Z[\tau, \omega] \to \Z[\tau',\omega']$
by $\tau \mapsto -\tau'$, $\omega \mapsto -\omega'+\tau'^4$.
In fact, $f$ sends $\bar{r}_9$ to $\tau'^9-3\omega'^2\tau'$ and $-\bar{r}_{12}$ to $\omega'^3+15\omega'^2\tau'^4-9\omega'\tau'^8$.
Hence $f$ induces an isomorphism between $\Z[\tau, \omega]/(\bar{r}_9, \bar{r}_{12})$ and $H^*(EIII)$.
\end{proof}

Let us show that $\Z[\tau, \omega]/(\bar{r}_9, \bar{r}_{12})$ is a free $H^*(BT)$-module
in the following setting.
Let $R_1$ be a non-negatively graded ring,
$R_2 \cong R_1/R_1^+$,
$\psi \colon R_1 \to R_2$ a natural surjective ring homomorphism
with a section $s$,
$\{ x_\lambda \}_{\lambda \in \Lambda}$ are indeterminates of positive degree,
$\{ r_\theta \}_{\theta \in \Theta} \subset R_1[x_\lambda \suchthat{\lambda \in \Lambda}]$
a subset of homogeneous elements,
$\psi^*$ the homomorphism $R_1[x_\lambda \suchthat{\lambda \in \Lambda}] \to R_2[x_\lambda \suchthat{\lambda \in \Lambda}]$
induced by $\psi \colon R_1 \to R_2$.
\begin{prop}
\label{free}
If $R_2[x_\lambda \suchthat{\lambda \in \Lambda}]/(\psi^*(r_\theta) \suchthat {\theta \in \Theta})$
is a free $R_2$-module,
then $R_1[x_\lambda \suchthat{\lambda \in \Lambda}]/(r_\theta \suchthat {\theta \in \Theta})$
is a free $R_1$-module.
\end{prop}

\begin{proof}
Fix $R_2$-basis $\{ b_i \}_{i \in I}$ of $R_2[x_\lambda \suchthat{\lambda \in \Lambda}]/(r_\theta \suchthat {\theta \in \Theta})$.
Let $f \in R_2[x_\lambda \suchthat{\lambda \in \Lambda}]$
be a representative of an element which is not expressed as an $R_1$-linear combination of $s^*(b_i)$'s.
Let us consider the degree of the highest terms of $f$ with respect to $x_\lambda$'s.
Assume that $f$ is a minimal element with respect to this degree.
However, if $f$ is not $0$,
the highest monomials in $x_\lambda$'s are expressed
as $R_2$-linear combinations of $b_i$'s.
Then, applying $s^*$, the highest monomials are expressed as $R_1$-linear combinations of $s^*(b_i)$'s.
It contradicts the minimality of $f$.
Therefore $s^*(b_i)$'s span $R_1[x_\lambda \suchthat{\lambda \in \Lambda}]/(r_\theta \suchthat {\theta \in \Theta})$.

Let us show that $s^*(b_i)$'s are $R_1$-linearly independent.
Consider the following exact sequence
\[
K \to \bigoplus_i R_1 s^*(b_i) \to R_1[x_\lambda \suchthat{\lambda \in \Lambda}]/(r_\theta \suchthat {\theta \in \Theta}) \to 0,
\]
where $K$ is the kernel.
Applying the right exact functor $R_2 \otimes_{R_1} \cdot$ to this exact sequence,
we obtain the exact sequence
\[
R_2 \otimes_{R_1} K \to \bigoplus_i R_2 b_i \to R_2[x_\lambda \suchthat{\lambda \in \Lambda}]/(\psi^*(r_\theta) \suchthat {\theta \in \Theta}) \to 0.
\]
Since $b_i$'s are $R_2$-basis of $R_2[x_\lambda \suchthat{\lambda \in \Lambda}]/(\psi^*(r_\theta) \suchthat {\theta \in \Theta})$,
we have $K/KR_1^+ \cong R_2 \otimes_{R_1} K = 0$.
By the graded Nakayama's lemma, $K=0$.
\end{proof}

By Proposition \ref{Sc. basis},
$H^*(\G(EIII))$ is a free $H^*(BT)$-module,
and then $H^*(EIII)$ is a free abelian group.
Thus, by Proposition \ref{isom. of EIII},
$\Z[\tau, \omega]/(\bar{r}_9, \bar{r}_{12})$ is also a free abelian group.
By Proposition \ref{free},
$H^*(BT)[\tau, \omega]/(r_9, r_{12})$ is a free $H^*(BT)$-module.
Hence the surjection
\[
H^*(BT)[\tau, \omega]/(r_9, r_{12}) \to H^*(\G(EIII))
\]
is an $H^*(BT)$-module isomorphism,
and then an $H^*(BT)$-algebra isomorphism.

\subsection{The fiber GKM graph}
\label{subsec. fiber}
The fiber GKM graph is the GKM subgraph induced by $W(D_5) \subset W(E_6)$.
Let $T^5$ be the standard maximal torus of $\mathrm{SO}(10)$
and $\{ t_i' \suchthat 1 \leq i \leq 5\}$ the standard basis of $H^2(BT^5)$.
We define some GKM functions as follows:
\begin{align*}
z'_i(w) & = w(t'_i) \quad \text{for } w \in W(D_5),\\
\gamma'_j & = \frac{1}{2}(c_j(\mathbf{z}')-c_j(\mathbf{t}')),
\end{align*}
where $\mathbf{z}' = \{ z_1',\ldots, z_5'\}$ and $\mathbf{t}' = \{ t_1',\ldots, t_5'\}$.
According to \cite{FIM},
the cohomology $H^*(\G(\mathrm{SO}(10)/T^5))$ are given as follows.
\begin{thm}[{\cite[Theorem 6.1]{FIM}}]
Let $t'_i$'s, $z'_i$'s, and $\gamma'_j$'s be indeterminates corresponding to the above GKM functions.
In particular, $|t'_i| = |z'_i| = 2$ and $|\gamma'_j| = 2j$.
Then
\[
H^*(\G(\mathrm{SO}(10)/T^5)) = \Z[t'_i, z'_i, \gamma'_j \suchthat 1 \leq i \leq 5, \ 1 \leq j \leq 4]/I',
\]
where $I'$ is the ideal generated by
\begin{align*}
& 2\gamma'_j - c_j(\mathbf{z}') +c_j(\mathbf{t}') \quad (1 \leq j \leq 4),\\
& \sum_{j=1}^{2k}(-1)^j \gamma'_j (\gamma'_{2k-j} + c_{2k-j}(\mathbf{t}')) \quad (1 \leq k \leq 4), \\
& c_5(\mathbf{z}') - c_5(\mathbf{t}'),
\end{align*}
where $\gamma'_j = 0$ for $j \geq 5$ and $c_i(\mathbf{t}') = 0$ for $i > 5$.
\end{thm}
For $1 \leq k \leq 4$, we have
\begin{equation}
\label{c_k}
c_k(\mathbf{z}'^2) - c_k(\mathbf{t}'^2) = 0
\end{equation}
as a GKM function
since $W(\mathrm{SO}(10))$ permutes $t'_i$'s and changes (even number of) their signs.
Fukukawa, Masuda, and Ishida showed that
the left-hand side of \eqref{c_k} is divisible by $4$ and the quotient is exactly the ideal generator
$\sum_{j=1}^{2k}(-1)^j \gamma'_j (\gamma'_{2k-j} + c_{2k-j}(\mathbf{t}'))$
(cf. \cite[Proof of Lemma 5.5]{FIM}).

We may assume that $T^5$ is the inverse image of the maximal torus $T_0=T/T^1$
of $U/T^1 = \mathrm{SO}(10)/\Z_2$ under $\mathrm{SO}(10) \to \mathrm{SO}(10)/\Z_2$.
Let $\tilde{T}^5$ be the inverse image of $T^5$
under the universal covering $\mu \colon \mathrm{Spin}(10) \to \mathrm{SO}(10)$
and $\tilde{t}_i$ the image of $t_i'$ by $\mu^*$.
Let $\tilde{z}_i(w) = w(\tilde{t}_i)$ for $w \in W(D_5)$,
$\tilde{\mathbf{t}} = \{\tilde{\mathbf{t}}_1, \ldots, \tilde{\mathbf{t}}_5\}$,
$\tilde{\mathbf{z}} = \{\tilde{\mathbf{z}}_1, \ldots, \tilde{\mathbf{z}}_5\}$,
and $\tilde{\gamma}_j = \frac{1}{2}(c_j(\tilde{\mathbf{z}}) - c_j(\tilde{\mathbf{t}}))$.
The maps between tori
$T^5 \leftarrow \tilde{T}^5 \to T$ induce a diagram of ring homomorphisms
\[
\Z[t'_i \suchthat 1 \leq i \leq 5]
\to \Z[\tilde{t}_i, \tilde{\gamma} \suchthat 1 \leq i \leq 5]/(2\tilde{\gamma}-c_1(\tilde{\mathbf{t}}))
\leftarrow \Z[t_i,x \suchthat 0 \leq i \leq 5]/(3x-c_1(\mathbf{t})).
\]
Let $G$ denote the underlying graph of the fiber GKM graph.
Let $\alpha_1$, $\alpha_2$, and $\alpha_1'$ denote the canonical axial function
corresponding to $T^5$, $\tilde{T}^5$, and $T$, respectively.
Recall that $H^*(G,\alpha'_1)$ has the Schubert classes and they form a $H^*(BT^5)$-basis.
The Schubert classes of $H^*(G,\alpha'_1)$ naturally can be regarded as elements of $H^*(G,\alpha_1)$ (respectively $H^*(G,\alpha_2)$)
and they form an $H^*(BT)$-basis of $H^*(G,\alpha_1)$ (respectively an $H^*(B\tilde{T}^5)$-basis of $H^*(G,\alpha_2)$).
Hence the following conditions hold for $(\beta,S) = (\alpha_1,T)$, $(\alpha_1',T^5)$.
\begin{enumerate}
\item
$H^*(G,\beta)$ is a free $H^*(BS)$-module.
\item
$H^*(G,\alpha_2)$ is isomorphic with $H^*(G,\alpha)\otimes_{H^*(BS)} H^*(B\tilde{T}^5)$.
\end{enumerate}
Therefore, by Proposition \ref{images of relations},
\[
H^*(G,\alpha_2) \cong H^*(B\tilde{T}^5)[\tilde{z}_i, \tilde{\gamma}_j \suchthat 1 \leq i \leq 5, \ 1 \leq j \leq 4]/\tilde{I},
\]
where $\tilde{I}$ is the ideal generated by 
the images of the generators of $I'$ under $\mu^*$.


If we have lifts of $\tilde{z}_i$'s and $\tilde{\gamma}_j$'s through
$H^*(B\tilde{T}^5) \leftarrow H^*(BT)$,
then Proposition \ref{lifts of generators} guarantees that
the lifts are $H^*(BT)$-algebra generators of $H^*(G,\alpha_1)$.
Obviously $z_i \in H^*(G,\alpha_1)$ is a lift of $\tilde{z}_i \in H^*(G,\alpha_2)$.
Let $\mathbf{t} = \{\mathbf{t}_0, \ldots, \mathbf{t}_5\}$ and
$\mathbf{z} = \{\mathbf{z}_0, \ldots, \mathbf{z}_5\}$.
Let us consider lifts of $\tilde{\gamma}_j$'s.
We define four functions $\gamma_1, \ldots, \gamma_4 \colon W(D_5) \to H^*(BT)\otimes \Q$ to be
\begin{align*}
\gamma_1 &= \frac{1}{3}(c_1(\mathbf{z}) - c_1(\mathbf{t})),\\
\gamma_2 &= \frac{1}{4}(c_2(\mathbf{z}) - c_2(\mathbf{t})),\\
\gamma_3   &= \frac{1}{2}(c_3(\mathbf{z}) - c_3(\mathbf{t}) -\gamma_1 c_2(\mathbf{t})) -\gamma_2 z_0 +\gamma_1 t_0^2,\\
\gamma_4 &= \frac{1}{3}(c_4(\mathbf{z}) - c_4(\mathbf{t}) -\gamma_2^2 -\gamma_2 (c_2(\mathbf{t}) -2t_0^2 -3t_0 \bar{t}) +\gamma_1(2\gamma_1 \bar{t}^2 +\gamma_1 t_0 \bar{t} +\bar{t}^3 +2t_0^2 \bar{t}))).
\end{align*}
We can verify that they are GKM functions $W(D_5) \to H^*(BT)$
by evaluating them at each vertex.
In fact, by symmetry,
it is sufficient to evaluate them at the unit element $e$,
$\sigma_{(e_1 +e_2)}$, and $\sigma_{(e_1 +e_2)} \sigma_{(e_3+e_4)}$.
The images of $\gamma_j$'s under $\iota_0^* \colon H^*(BT) \to H^*(B\tilde{T}^5)$
are as follows.
\begin{align*}
\iota_0^* \gamma_1 = & \: \frac{1}{3}(c_1(\tilde{\mathbf{z}}) + \tilde{\gamma}_1 +\tilde{\gamma} -c_1(\tilde{\mathbf{t}}) - \tilde{\gamma}) = \tilde{\gamma}_1,\\
\iota_0^* \gamma_2 = & \: \frac{1}{4}(c_2(\tilde{\mathbf{z}}) + \frac{1}{2}c_1(\tilde{\mathbf{z}})^2 - c_2(\tilde{\mathbf{t}}) - \frac{1}{2}c_1(\tilde{\mathbf{t}})^2)\\
			 = & \: \frac{1}{4}(c_2(\tilde{\mathbf{z}}) + \frac{1}{2}c_1(\tilde{\mathbf{z}}^2) +c_2(\tilde{\mathbf{z}}) - c_2(\tilde{\mathbf{t}}) -\frac{1}{2}c_1(\tilde{\mathbf{t}}^2) -c_2(\tilde{\mathbf{t}})) = \tilde{\gamma}_2,\\
\iota_0^* \gamma_3   = & \: \frac{1}{2}(c_3(\tilde{\mathbf{z}}) + (\tilde{\gamma}_1+\tilde{\gamma})c_2(\tilde{\mathbf{z}}) - c_3(\tilde{\mathbf{t}}) - \tilde{\gamma}c_2(\tilde{\mathbf{t}}) -\tilde{\gamma}_1 (c_2(\tilde{\mathbf{t}}) +2\tilde{\gamma}^2)) -\tilde{\gamma}_2 (\tilde{\gamma}_1 +\tilde{\gamma}) +\tilde{\gamma}_1\tilde{\gamma}^2 = \tilde{\gamma}_3,\\
\iota_0^* \gamma_4 = & \: \frac{1}{3}(c_4(\tilde{\mathbf{z}}) + \frac{1}{2}c_1(\tilde{\mathbf{z}})c_3(\tilde{\mathbf{z}}) - c_4(\tilde{\mathbf{t}}) - \frac{1}{2}c_1(\tilde{\mathbf{t}})c_3(\tilde{\mathbf{t}}) - \tilde{\gamma}_2^2 -\tilde{\gamma}_2 c_2(\tilde{\mathbf{t}})) \\
			 = & \: \frac{1}{3}(c_4(\tilde{\mathbf{z}}) - c_4(\tilde{\mathbf{t}}) + \frac{1}{2}(c_4(\tilde{\mathbf{z}}) - c_4(\tilde{\mathbf{t}})) + \frac{1}{4}(c_2(\tilde{\mathbf{z}})^2 - c_2(\tilde{\mathbf{t}})^2) - \tilde{\gamma}_2(\tilde{\gamma}_2 + c_2(\tilde{\mathbf{t}}))) = \tilde{\gamma}_4. \\
\end{align*}
Hence $z_i$'s and $\gamma_j$'s generate $H^*(G,\alpha_1)$ as an $H^*(BT)$-algebra.

Next we have to consider lifts of the generators of the ideal $\tilde{I}$.
However, since we have to lift them from the fiber $\F$ to the whole GKM graph $\G$ further,
we consider them later and will obtain the generators of $\tilde{I}=I_\F$
when we obtain the ideal generators of $I_\G$.

\subsection{The equivariant cohomology ring of $E_6/T$}
We have defined the GKM functions $\gamma_j$'s on the fiber GKM graph in the previous subsection
and they are also GKM functions on $\G((E_6)_\C/B)$.
We can check it by evaluating them at each vertex of $\G((E_6)_\C/B)$.
In fact, by symmetry,
it is sufficient to evaluate them at a much smaller number of vertices as above.

Let us consider lifts of the generators of $\tilde{I}$.
Let $R$ be a ring.
For $\mathbf{x}= \{ x_1, \ldots, x_n \},\mathbf{y}= \{ y_1, \ldots, y_n \} \subset R$ and $z \in R$,
let $\mathbf{x}-z$ denote $\{ x_1-z, \ldots, x_n-z \}$
and $\mathbf{x}\mathbf{y}$ denote $\{ x_1 y_1, \ldots, x_n y_n\}$.
Let $\mathbf{t}_+ = \{t_1, \ldots, t_5\}$ and
$\mathbf{z}_+ = \{z_1, \ldots, z_5\}$.
Since, for $w \in W(D_5)$ and $1 \leq i, j \leq 5$, $w(t_i) = -(t_j -\bar{t})$ if and only if $w(t_i -\bar{t}) = -t_j$,
$W(D_5)$ permutes $\{t_i(t_i -\bar{t}) \suchthat 1 \leq i \leq 5\}$.
Hence, for $1\leq j \leq 4$,
$c_j(\mathbf{z}_+(\mathbf{z}_+-\bar{z}))$ is constant on each fiber $\rho_\lambda W(D_5)$.
Thus the GKM function $c_j(\mathbf{z}_+(\mathbf{z}_+-\bar{z})) - c_j(\mathbf{t}_+(\mathbf{t}_+-\bar{t}))$ is regarded as a GKM function on $\G(EIII)$.
Moreover we have the GKM function $c_5(2\mathbf{z}_+-\bar{z}) - c_5(2\mathbf{t}_+-\bar{t})$ on $\G(EIII)$.
Let $\gamma_i = 0$ for $i > 4$ and $c_i(\mathbf{x}) = 0$ for $i > n$.
By Theorem \ref{generators of EIII}, we can express these GKM functions as polynomials in $\tau$ and $\omega$ over $H^*(BT)$.
Moreover we can divide these polynomials by some integers (see Appendix \ref{appendix}).

Let $r_2$, $r_4$, $r_5$, $r_6$, and $r_8$ denote
the left-hand side polynomials over $H^*(BT)$ of \eqref{c1/8}, \eqref{c2/6}, \eqref{c5/32}, \eqref{c3/4}, and \eqref{c4/4}, respectively.
The images of all terms of the above equations by $\iota_0^*$ are zero
except for $\sum_{j=1}^{2k}(-1)^j \gamma_j (\gamma_{2k-j} + c_{2k-j}(\mathbf{t}_+))$ in $r_k$
and $c_5(\mathbf{z}_+) - c_5(\mathbf{t}_+)$ in $r_5$
since the vanishing terms are multiples of $\tau$, $\omega$, or $\bar{t}$.
Hence, for $k=2,4,6,8$,
$r_k$ is a lift of $\sum_{j=1}^{2k}(-1)^j \gamma_j' (\gamma_{2k-j}' + c_{2k-j}(\mathbf{t}'))$
and $r_5$ is a lift of $c_5(\mathbf{z}') - c_5(\mathbf{t}')$.

By Theorem \ref{generators},
$z_i$'s, $\gamma_j$'s, $\tau$, and $\omega$ generates $H^*(\G((E_6)_\C/B))$ as an $H^*(BT)$-algebra.
Put
\begin{align*}
q_1 &= -3\gamma_1 +(c_1(\mathbf{z}) - c_1(\mathbf{t})),\\
q_2 &= -4\gamma_2 +(c_2(\mathbf{z}) - c_2(\mathbf{t})),\\
q_3 &= -2\gamma_3  +(c_3(\mathbf{z}) - c_3(\mathbf{t}) -\gamma_1 c_2(\mathbf{t})) -2\gamma_2 z_0 +2\gamma_1 t_0^2,\\
q_4 &= -3\gamma_4 +(c_4(\mathbf{z}) - c_4(\mathbf{t}) -3\gamma_2^2 -3\gamma_2 (c_2(\mathbf{t}) -2t_0^2 -3t_0 \bar{t}) +3\gamma_1(2\gamma_1 \bar{t}^2 +\gamma_1 t_0 \bar{t} +\bar{t}^3 +2t_0^2 \bar{t}))).
\end{align*}
We have verified that $q_i$ is a lift of $2\tilde{\gamma}_i - (c_i(\tilde{\mathbf{z}}) - c_i(\tilde{\mathbf{t}}))$
through the ring homomorphism
\[
\Z[\tilde{t}_i, \tilde{\gamma} \suchthat 1 \leq i \leq 5]/(2\tilde{\gamma}-c_1(\tilde{\mathbf{t}}))
\overset{\iota^*_0}{\leftarrow} \Z[t_i,x \suchthat 0 \leq i \leq 5]/(3x-c_1(\mathbf{t}))
\]
by computing $\iota^*_0 \gamma_i$ for $1 \leq i \leq 4$ in subsection \ref{subsec. fiber}.
By Theorem \ref{relations}, the generators of the ideal 
$\mathrm{Ker} (\Xi \colon H^*(BT)[z_i,\gamma_j,\tau,\omega \suchthat 1 \leq i \leq 5, 1 \leq j \leq 4] \to H^*(\G((E_6)_\C/B)))$ are
$q_1$, $q_2$, $q_3$, $q_4$, $r_2$, $r_4$, $r_5$, $r_6$, $r_8$, $r_9$, and $r_{12}$.
Now we obtain a concrete description of the integral equivariant cohomology of $E_6/T$.

\begin{thm}
\label{thm. main}
Let $z_i$'s, $\gamma_j$'s, $\tau$, and $\omega$ be indeterminates corresponding to the above GKM functions.
In particular, $|z_i| = |\tau| = 2$, $|\gamma_j| = 2j$, and $|\omega| = 8$.
Then
\[
H^*_T(E_6/T) = H^*(BT)[z_i,\gamma_j,\tau,\omega \suchthat 1 \leq i \leq 5, 1 \leq j \leq 4]/I,
\]
where
$I$ is the ideal generated by $q_1$, $q_2$, $q_3$, $q_4$, $r_2$, $r_4$, $r_5$, $r_6$, $r_8$, $r_9$, and $r_{12}$.
\end{thm}

As a corollary, we obtain the ordinary integral cohomology ring of $E_6/T$.

\begin{cor}[{\cite[Theorem B]{TW}}]
Let $|y_i| = 2$ and $|\delta_j| = 2j$ for $0 \leq i \leq 5$ and $j=1$, $3$, $4$.
Then
$H^*(E_6/T) \cong \Z[y_i,\delta_1,\delta_3,\delta_4 \suchthat 0 \leq i \leq 5]/J$,
where
$|y_i|=2$, $|\delta_i|=2i$ and
$J$ is the ideal generated by $Q_1$, $Q_2$, $Q_3$, $Q_4$, $R_5$, $R_6$, $R_8$, $R_9$, and $R_{12}$:
\begin{align*}
& Q_1 = c_1(\mathbf{y}) -3\delta_1,
 \qquad Q_2 = c_2(\mathbf{y}) -4\delta_1^2,
 \qquad Q_3 = c_3(\mathbf{y}) -2\delta_3,\\
& Q_4 = c_4(\mathbf{y}) +2\delta_1^4 -3\delta_4,
 \qquad R_5 = c_5(\mathbf{y}) -c_4(\mathbf{y}) \delta_1 +c_3(\mathbf{y})\delta_1^2 -2\delta^5, \\
& R_6 = 2c_6(\mathbf{y}) -c_4(\mathbf{y})\delta_1^2 -\delta_1^6 +\delta_3^2,
 \qquad R_8 = -9c_6(\mathbf{y})\delta_1^2 +3c_5(\mathbf{y})\delta_1^3 -\delta_1^8 +3\delta_4(\delta_4-c_3(\mathbf{y})\delta_1 +2\delta_1^4),\\
& R_9 = -3\omega'\tau' +\tau'^9, \qquad R_{12} = \omega'^3 +15\omega'^2\tau'^4 -9\omega'\tau'^8,
\end{align*}
for $\mathbf{y}=\{y_0, y_1, \ldots, y_5\}$, $\tau' = \delta_1-y_0$,
and $\omega' = \delta_4 -c_3(\mathbf{y})\delta_1 +2\delta_1^4 +(\delta_3 -2\delta_1^3 +\delta_1^2\tau' -\delta_1\tau'^2 +\tau'^3)\tau'$.
\end{cor}

\begin{proof}
Let us consider the fiber sequence.
\[
E_6/T \to ET \times_T E_6/T \to BT
\]
Since the cohomology rings of $E_6/T$ and $BT$ have vanishing odd parts,
by Theorem \ref{thm. main}
\[
H^*(E_6/T) = (H^*(BT)[z_i,\gamma_j,\tau,\omega \suchthat 1 \leq i \leq 5, 1 \leq j \leq 4]/I)/H^{>0}(BT).
\]
Under the homomorphism $H^*(ET \times_T E_6/T) \to H^*(E_6/T)$,
the relations $q_i$'s and $r_i$'s turn into
\begin{align*}
\bar{q}_1 &= -3\gamma_1 +c_1(\mathbf{z}), &&\bar{q}_2 = -4\gamma_2 +c_2(\mathbf{z}),\\
\bar{q}_3 &= -2\gamma_3 +c_3(\mathbf{z}) -2\gamma_2 z_0, && \bar{q}_4 = -3\gamma_4 +c_4(\mathbf{z}) -3\gamma_2^2,
\end{align*}
\begin{align*}
 \bar{r}_2 = \ & \gamma_2 - \gamma_1^2,\\
 \bar{r}_4 = \ & \sum_{j=1}^{4}(-1)^j \gamma_j (\gamma_{4-j} + c_{4-j}(\mathbf{t}_+)) +\omega - \gamma_1^3\tau + \gamma_1\tau^3 - \gamma_3\tau, \\
 \bar{r}_5 = \ & c_5(\mathbf{z}_+) + \gamma_4\tau + \tau^5 - 2\tau\omega, \\
 \bar{r}_6 = \ & \displaystyle{\sum_{j=1}^{6}(-1)^j} \gamma_j (\gamma_{6-j} + c_{6-j}(\mathbf{t}_+))
 +2\tau^6 +2\gamma_1\tau^5 +\gamma_1^2 \tau^4 -\gamma_1^3 \tau^3\\
 & +(-\gamma_1\gamma_3+2\gamma_4-2\gamma_1^4)\tau^2
 +(-3\gamma_1\omega -3\gamma_1^2\gamma_3+3\gamma_1\gamma_4)\tau -\gamma_1^2 \omega -4\tau^2\omega, \\
 \bar{r}_8 = \ & \displaystyle{\sum_{j=1}^{8}(-1)^j} \gamma_j (\gamma_{8-j} + c_{8-j}(\mathbf{t}_+))
 +\omega^2 - 5\tau^4\omega - 4\gamma_1^2\tau^2\omega - 2\gamma_1\tau^3\omega - 2\gamma_3\tau\omega - \gamma_4\omega \\
 & + 2\tau^8 + \gamma_1\tau^7 + 2\gamma_1^2\tau^6 + \gamma_3\tau^5  + \gamma_4\tau^4 + \gamma_1\gamma_4\tau^3
 + 2\gamma_1^2\gamma_4\tau^2 + \gamma_3\gamma_4\tau, \\
 \bar{r}_9 = \ & 2\tau^9 - 6\tau^5\omega + 3\tau\omega^2, \\
 \bar{r}_{12} = \ & \omega^3 -18\tau^4 \omega^2 +24\tau^8 \omega -7\tau^{12}.
\end{align*}
Hence
\[
H^*(E_6/T) = \Z[z_i,\gamma_j, \tau, \omega \suchthat 1 \leq i \leq 5,1 \leq j \leq 4]/(\bar{q}_1, \bar{q}_2, \bar{q}_3, \bar{q}_4, \bar{r}_2, \bar{r}_4, \bar{r}_5, \bar{r}_6, \bar{r}_8, \bar{r}_9, \bar{r}_{12}).
\]
We can erase $\gamma_2$ and $\omega$ by $\bar{r}_2=0$ and $\bar{r}_4 = 0$,
and we can define a ring homomorphism 
\[
f \colon \Z[z_i,\gamma_1, \gamma_3, \gamma_4, \tau \suchthat 1 \leq i \leq 5] \to \Z[y_i, \delta_1, \delta_3, \delta_4 \suchthat 0 \leq i \leq 5]
\]
by
$z_i \mapsto y_i$ for $1 \leq i \leq 5$, $\tau \mapsto -\tau' = y_0-\delta_1$,
$\gamma_1 \mapsto \delta_1$, $\gamma_3 \mapsto \delta_3 -\delta_1^2 y_0$, $\gamma_4 \mapsto \delta_4 -\delta_1^4$.
Obviously $f$ is isomorphic.
Straightforward calculations show that
\begin{align*}
& f(\bar{q}_i) = Q_i \text{ for } 1 \leq i \leq 4, \\
& f(\bar{r}_5) \equiv R_5 \pmod{Q_i \suchthat 1 \leq i \leq 4}, \\
& f(\bar{r}_6) \equiv -R_6 \pmod{Q_i, R_5 \suchthat 1 \leq i \leq 4},\\
& f(\bar{r}_8) \equiv R_8 \pmod{Q_i, R_5, R_6 \suchthat 1 \leq i \leq 4}.
\end{align*}
Notice that $f(\omega) \equiv -\omega' +\tau'^4 \pmod{Q_i, R_5, R_6, R_8 \suchthat 1 \leq i \leq 4}$
and then $f$ is compatible with the homomorphism in the proof of Proposition \ref{isom. of EIII}.
Hence
\begin{align*}
& f(\bar{r}_9) \equiv R_9 \pmod{Q_i, R_5, R_6, R_8 \suchthat 1 \leq i \leq 4},\\
& f(\bar{r}_{12}) \equiv -R_{12} \pmod{Q_i, R_5, R_6, R_8 \suchthat 1 \leq i \leq 4}.
\end{align*}
Therefore $f$ induces a ring isomorphism
\[
H^*(E_6/T) \cong \Z[y_i, \delta_1, \delta_3, \delta_4 \suchthat 0 \leq i \leq 5]/J.
\]
\end{proof}

\appendix

\section{Long equations and tables}
\label{appendix}
The kernel of $\Xi \colon H^*(BT)[\tau, \omega] \to H^*(\G(EIII))$ is generated by the following elements.
\begin{align*}
\nonumber
r_9 = & \: 2\tau^9 - 6\tau^5\omega + 3\tau\omega^2 + 6\tau^8t_0 - 12\tau^8\bar{t} - 3\tau^7t_0^2 - 54\tau^7t_0\bar{t} + 21\tau^7\bar{t}^2 + 3\tau^7c_2(\mathbf{t}) - 34\tau^6t_0^3 - 81\tau^6t_0^2\bar{t} \\ \nonumber
	& \: + 108\tau^6t_0\bar{t}^2 + 13\tau^6t_0c_2(\mathbf{t}) - 13\tau^6\bar{t}^3 - 8\tau^6\bar{t}c_2(\mathbf{t}) - 3\tau^6c_3(\mathbf{t}) - 45\tau^5t_0^4 + 24\tau^5t_0^3\bar{t} + 279\tau^5t_0^2\bar{t}^2 \\ \nonumber
	& \: + 15\tau^5t_0^2c_2(\mathbf{t}) - 30\tau^5t_0\bar{t}^3 - 48\tau^5t_0\bar{t}c_2(\mathbf{t}) - 9\tau^5t_0c_3(\mathbf{t}) + 12\tau^5\bar{t}^4 + 9\tau^5\bar{t}c_3(\mathbf{t}) + \tau^5c_2(\mathbf{t})^2 + 2\tau^5c_4(\mathbf{t}) \\ \nonumber
      & \: - 6\tau^4\omega t_0 + 24\tau^4\omega \bar{t} + 12\tau^4t_0^5 + 285\tau^4t_0^4\bar{t} + 510\tau^4t_0^3\bar{t}^2 - 19\tau^4t_0^3c_2(\mathbf{t}) + 45\tau^4t_0^2\bar{t}^3 - 132\tau^4t_0^2\bar{t}c_2(\mathbf{t}) \\ \nonumber
	& \: - 3\tau^4t_0^2c_3(\mathbf{t}) + 60\tau^4t_0\bar{t}^4 - 12\tau^4t_0\bar{t}^2c_2(\mathbf{t}) + 39\tau^4t_0\bar{t}c_3(\mathbf{t}) + 6\tau^4t_0c_2(\mathbf{t})^2 + 4\tau^4t_0c_4(\mathbf{t}) - 4\tau^4\bar{t}^3c_2(\mathbf{t}) \\ \nonumber
      & \: - 3\tau^4\bar{t}^2c_3(\mathbf{t}) + \tau^4\bar{t}c_2(\mathbf{t})^2 - 6\tau^4\bar{t}c_4(\mathbf{t}) - 2\tau^4c_2(\mathbf{t})c_3(\mathbf{t}) - 2\tau^4c_5(\mathbf{t}) + 10\tau^3\omega t_0^2 + 44\tau^3\omega t_0\bar{t} - 26\tau^3\omega \bar{t}^2 \\ \nonumber
      & \: - 4\tau^3\omega c_2(\mathbf{t}) + 95\tau^3t_0^6 + 504\tau^3t_0^5\bar{t} + 675\tau^3t_0^4\bar{t}^2 - 73\tau^3t_0^4c_2(\mathbf{t}) + 220\tau^3t_0^3\bar{t}^3 - 208\tau^3t_0^3\bar{t}c_2(\mathbf{t}) \\ \nonumber
      & \: + 24\tau^3t_0^3c_3(\mathbf{t}) + 120\tau^3t_0^2\bar{t}^4 - 48\tau^3t_0^2\bar{t}^2c_2(\mathbf{t}) + 78\tau^3t_0^2\bar{t}c_3(\mathbf{t}) + 14\tau^3t_0^2c_2(\mathbf{t})^2 - 3\tau^3t_0^2c_4(\mathbf{t}) \\ \nonumber
      & \: - 16\tau^3t_0\bar{t}^3c_2(\mathbf{t}) + 4\tau^3t_0\bar{t}c_2(\mathbf{t})^2 - 18\tau^3t_0\bar{t}c_4(\mathbf{t}) - 9\tau^3t_0c_2(\mathbf{t})c_3(\mathbf{t}) + 4\tau^3\bar{t}^3c_3(\mathbf{t}) + 3\tau^3\bar{t}^2c_4(\mathbf{t}) \\ \nonumber
      & \: - \tau^3\bar{t}c_2(\mathbf{t})c_3(\mathbf{t}) + 6\tau^3\bar{t}c_5(\mathbf{t}) + \tau^3c_2(\mathbf{t})c_4(\mathbf{t}) + \tau^3c_3(\mathbf{t})^2 + 24\tau^2\omega t_0^3 + 42\tau^2\omega t_0^2\bar{t} - 24\tau^2\omega t_0\bar{t}^2 \\ \nonumber
      & \: - 9\tau^2\omega t_0c_2(\mathbf{t}) + 18\tau^2\omega \bar{t}^3 + 3\tau^2\omega \bar{t}c_2(\mathbf{t}) + 3\tau^2\omega c_3(\mathbf{t}) + 102\tau^2t_0^7 + 429\tau^2t_0^6\bar{t} + 558\tau^2t_0^5\bar{t}^2 \\ \nonumber
      & \: - 81\tau^2t_0^5c_2(\mathbf{t}) + 285\tau^2t_0^4\bar{t}^3 - 186\tau^2t_0^4\bar{t}c_2(\mathbf{t}) + 39\tau^2t_0^4c_3(\mathbf{t}) + 120\tau^2t_0^3\bar{t}^4 - 72\tau^2t_0^3\bar{t}^2c_2(\mathbf{t}) \\ \nonumber
      & \: + 84\tau^2t_0^3\bar{t}c_3(\mathbf{t}) + 16\tau^2t_0^3c_2(\mathbf{t})^2 - 11\tau^2t_0^3c_4(\mathbf{t}) - 24\tau^2t_0^2\bar{t}^3c_2(\mathbf{t}) + 18\tau^2t_0^2\bar{t}^2c_3(\mathbf{t}) + 6\tau^2t_0^2\bar{t}c_2(\mathbf{t})^2 \\ \nonumber
      & \: - 24\tau^2t_0^2\bar{t}c_4(\mathbf{t}) - 15\tau^2t_0^2c_2(\mathbf{t})c_3(\mathbf{t}) + 3\tau^2t_0^2c_5(\mathbf{t}) + 12\tau^2t_0\bar{t}^3c_3(\mathbf{t}) - 3\tau^2t_0\bar{t}^2c_4(\mathbf{t}) - 3\tau^2t_0\bar{t}c_2(\mathbf{t})c_3(\mathbf{t}) \\ \nonumber
      & \: + 6\tau^2t_0\bar{t}c_5(\mathbf{t}) + 4\tau^2t_0c_2(\mathbf{t})c_4(\mathbf{t}) + 3\tau^2t_0c_3(\mathbf{t})^2 - 4\tau^2\bar{t}^3c_4(\mathbf{t}) - 3\tau^2\bar{t}^2c_5(\mathbf{t}) + \tau^2\bar{t}c_2(\mathbf{t})c_4(\mathbf{t}) \\ \nonumber
      & \: - \tau^2c_2(\mathbf{t})c_5(\mathbf{t}) - \tau^2c_3(\mathbf{t})c_4(\mathbf{t}) + 24\tau\omega t_0^4 + 48\tau\omega t_0^3\bar{t} + 30\tau\omega t_0^2\bar{t}^2 - 10\tau\omega t_0^2c_2(\mathbf{t}) + 36\tau\omega t_0\bar{t}^3 \\ \nonumber
      & \: - 2\tau\omega t_0\bar{t}c_2(\mathbf{t}) + 6\tau\omega t_0c_3(\mathbf{t}) - 4\tau\omega \bar{t}^2c_2(\mathbf{t}) - 2\tau\omega c_4(\mathbf{t}) + 45\tau t_0^8 + 183\tau t_0^7\bar{t} + 261\tau t_0^6\bar{t}^2 - 38\tau t_0^6c_2(\mathbf{t}) \\ \nonumber
      & \: + 171\tau t_0^5\bar{t}^3 - 87\tau t_0^5\bar{t}c_2(\mathbf{t}) + 24\tau t_0^5c_3(\mathbf{t}) + 60\tau t_0^4\bar{t}^4 - 51\tau t_0^4\bar{t}^2c_2(\mathbf{t}) + 51\tau t_0^4\bar{t}c_3(\mathbf{t}) + 8\tau t_0^4c_2(\mathbf{t})^2 \\ \nonumber
      & \: - 10\tau t_0^4c_4(\mathbf{t}) - 16\tau t_0^3\bar{t}^3c_2(\mathbf{t}) + 27\tau t_0^3\bar{t}^2c_3(\mathbf{t}) + 4\tau t_0^3\bar{t}c_2(\mathbf{t})^2 - 24\tau t_0^3\bar{t}c_4(\mathbf{t}) - 10\tau t_0^3c_2(\mathbf{t})c_3(\mathbf{t}) \\ \nonumber
      & \: + 5\tau t_0^3c_5(\mathbf{t}) + 12\tau t_0^2\bar{t}^3c_3(\mathbf{t}) - 18\tau t_0^2\bar{t}^2c_4(\mathbf{t}) - 3\tau t_0^2\bar{t}c_2(\mathbf{t})c_3(\mathbf{t}) + 12\tau t_0^2\bar{t}c_5(\mathbf{t}) + 4\tau t_0^2c_2(\mathbf{t})c_4(\mathbf{t}) \\ \nonumber
      & \: + 3\tau t_0^2c_3(\mathbf{t})^2 - 8\tau t_0\bar{t}^3c_4(\mathbf{t}) + 9\tau t_0\bar{t}^2c_5(\mathbf{t}) + 2\tau t_0\bar{t}c_2(\mathbf{t})c_4(\mathbf{t}) - 2\tau t_0c_2(\mathbf{t})c_5(\mathbf{t}) - 2\tau t_0c_3(\mathbf{t})c_4(\mathbf{t}) \\ \nonumber
      & \: + 4\tau\bar{t}^3c_5(\mathbf{t}) - \tau\bar{t}c_2(\mathbf{t})c_5(\mathbf{t}) + \tau c_3(\mathbf{t})c_5(\mathbf{t}) - 3\omega^2\bar{t} + 2\omega t_0^5 - 14\omega t_0^4\bar{t} - 52\omega t_0^3\bar{t}^2 - \omega t_0^3c_2(\mathbf{t}) - 62\omega t_0^2\bar{t}^3 \\ \nonumber
      & \: + 7\omega t_0^2\bar{t}c_2(\mathbf{t}) + \omega t_0^2c_3(\mathbf{t}) - 40\omega t_0\bar{t}^4 + 8\omega t_0\bar{t}^2c_2(\mathbf{t}) - 4\omega t_0\bar{t}c_3(\mathbf{t}) - \omega t_0c_4(\mathbf{t}) - 8\omega \bar{t}^5 + 4\omega \bar{t}^3c_2(\mathbf{t}) \\ \nonumber
      & \: - 2\omega \bar{t}^2c_3(\mathbf{t}) + \omega \bar{t}c_4(\mathbf{t}) + \omega c_5(\mathbf{t}),
 \\
r_{12} = & \: \omega^3
+\omega^2 M 
+\omega L 
+\prod_{i=0}^5 (\tau +t_0 -t_i) K,
\end{align*}
where
\begin{align*}
M = & \: -18\tau^4 + 24\tau^3t_0 - 12\tau^3\bar{t} - 69\tau^2t_0^2 + 150\tau^2t_0\bar{t} - 447\tau^2\bar{t}^2 + 12\tau^2c_2(\mathbf{t}) + 156\tau t_0^3 - 1026\tau t_0^2\bar{t} \\
    & \: + 1116\tau t_0\bar{t}^2 - 30\tau t_0c_2(\mathbf{t}) - 3450\tau \bar{t}^3 + 210\tau \bar{t}c_2(\mathbf{t}) - 12\tau c_3(\mathbf{t}) + 12t_0^4 - 108t_0^3\bar{t} + 1167t_0^2\bar{t}^2 \\
    & \: - 5t_0^2c_2(\mathbf{t}) - 1242t_0\bar{t}^3 + 20t_0\bar{t}c_2(\mathbf{t}) + 3t_0c_3(\mathbf{t}) + 3939\bar{t}^4 - 227\bar{t}^2c_2(\mathbf{t}) + 15\bar{t}c_3(\mathbf{t}) - c_4(\mathbf{t}),
\end{align*}
\begin{align*}
L = & \: +24\tau^8 - 18\tau^7t_0 + 6\tau^7\bar{t} + 65\tau^6t_0^2 - 248\tau^6t_0\bar{t} + 611\tau^6\bar{t}^2 - 11\tau^6c_2(\mathbf{t}) - 204\tau^5t_0^3 + 1182\tau^5t_0^2\bar{t} \\ \nonumber
    & \: - 1182\tau^5t_0\bar{t}^2 + 45\tau^5t_0c_2(\mathbf{t}) + 4560\tau^5\bar{t}^3 - 273\tau^5\bar{t}c_2(\mathbf{t}) + 9\tau^5c_3(\mathbf{t}) - 505\tau^4t_0^4 + 2600\tau^4t_0^3\bar{t} \\ \nonumber
    & \: - 10623\tau^4t_0^2\bar{t}^2 + 191\tau^4t_0^2c_2(\mathbf{t}) + 10544\tau^4t_0\bar{t}^3 - 527\tau^4t_0\bar{t}c_2(\mathbf{t}) + 15\tau^4t_0c_3(\mathbf{t}) - 27748\tau^4\bar{t}^4 \\ \nonumber
    & \: + 2213\tau^4\bar{t}^2c_2(\mathbf{t}) - 138\tau^4\bar{t}c_3(\mathbf{t}) - 19\tau^4c_2(\mathbf{t})^2 + 12\tau^4c_4(\mathbf{t}) + 102\tau^3t_0^5 - 894\tau^3t_0^4\bar{t} - 13812\tau^3t_0^3\bar{t}^2 \\ \nonumber
    & \: - 45\tau^3t_0^3c_2(\mathbf{t}) + 7320\tau^3t_0^2\bar{t}^3 + 1437\tau^3t_0^2\bar{t}c_2(\mathbf{t}) - 75\tau^3t_0^2c_3(\mathbf{t}) - 64680\tau^3t_0\bar{t}^4 + 3129\tau^3t_0\bar{t}^2c_2(\mathbf{t}) \\ \nonumber
    & \: - 66\tau^3t_0\bar{t}c_3(\mathbf{t}) - 2\tau^3t_0c_2(\mathbf{t})^2 - 7\tau^3t_0c_4(\mathbf{t}) + 31320\tau^3\bar{t}^5 + 3003\tau^3\bar{t}^3c_2(\mathbf{t}) - 351\tau^3\bar{t}^2c_3(\mathbf{t}) \\ \nonumber
    & \: - 278\tau^3\bar{t}c_2(\mathbf{t})^2 + 17\tau^3\bar{t}c_4(\mathbf{t}) + 30\tau^3c_2(\mathbf{t})c_3(\mathbf{t}) - 9\tau^3c_5(\mathbf{t}) + 693\tau^2t_0^6 - 6450\tau^2t_0^5\bar{t} - 8097\tau^2t_0^4\bar{t}^2 \\ \nonumber
    & \: - 371\tau^2t_0^4c_2(\mathbf{t}) - 10200\tau^2t_0^3\bar{t}^3 + 4519\tau^2t_0^3\bar{t}c_2(\mathbf{t}) - 72\tau^2t_0^3c_3(\mathbf{t}) - 74226\tau^2t_0^2\bar{t}^4 - 27\tau^2t_0^2\bar{t}^2c_2(\mathbf{t}) \\ \nonumber
    & \: - 957\tau^2t_0^2\bar{t}c_3(\mathbf{t}) + 45\tau^2t_0^2c_2(\mathbf{t})^2 + 35\tau^2t_0^2c_4(\mathbf{t}) + 34164\tau^2t_0\bar{t}^5 + 11353\tau^2t_0\bar{t}^3c_2(\mathbf{t}) + 498\tau^2t_0\bar{t}^2c_3(\mathbf{t}) \\ \nonumber
    & \: - 696\tau^2t_0\bar{t}c_2(\mathbf{t})^2 - 101\tau^2t_0\bar{t}c_4(\mathbf{t}) + 33\tau^2t_0c_2(\mathbf{t})c_3(\mathbf{t}) + 12\tau^2t_0c_5(\mathbf{t}) - 23466\tau^2\bar{t}^6 - 2045\tau^2\bar{t}^4c_2(\mathbf{t}) \\ \nonumber
    & \: - 3975\tau^2\bar{t}^3c_3(\mathbf{t}) + 201\tau^2\bar{t}^2c_2(\mathbf{t})^2 + 302\tau^2\bar{t}^2c_4(\mathbf{t}) + 207\tau^2\bar{t}c_2(\mathbf{t})c_3(\mathbf{t}) - 3\tau^2\bar{t}c_5(\mathbf{t}) - 9\tau^2c_2(\mathbf{t})c_4(\mathbf{t}) \\ \nonumber
    & \: - 12\tau^2c_3(\mathbf{t})^2 + 1204\tau t_0^7 - 5774\tau t_0^6\bar{t} - 5334\tau t_0^5\bar{t}^2 - 730\tau t_0^5c_2(\mathbf{t}) - 18864\tau t_0^4\bar{t}^3 + 4524\tau t_0^4\bar{t}c_2(\mathbf{t}) \\ \nonumber
    & \: + 194\tau t_0^4c_3(\mathbf{t}) - 57204\tau t_0^3\bar{t}^4 - 461\tau t_0^3\bar{t}^2c_2(\mathbf{t}) - 2224\tau t_0^3\bar{t}c_3(\mathbf{t}) + 96\tau t_0^3c_2(\mathbf{t})^2 - 82\tau t_0^3c_4(\mathbf{t}) \\ \nonumber
    & \: - 15912\tau t_0^2\bar{t}^5 + 14513\tau t_0^2\bar{t}^3c_2(\mathbf{t}) + 2145\tau t_0^2\bar{t}^2c_3(\mathbf{t}) - 682\tau t_0^2\bar{t}c_2(\mathbf{t})^2 + 652\tau t_0^2\bar{t}c_4(\mathbf{t}) \\ \nonumber
    & \: - 16\tau t_0^2c_2(\mathbf{t})c_3(\mathbf{t}) - 22\tau t_0^2c_5(\mathbf{t}) - 41308\tau t_0\bar{t}^6 + 2310\tau t_0\bar{t}^4c_2(\mathbf{t}) - 7486\tau t_0\bar{t}^3c_3(\mathbf{t}) - 74\tau t_0\bar{t}^2c_2(\mathbf{t})^2 \\ \nonumber
    & \: - 649\tau t_0\bar{t}^2c_4(\mathbf{t}) + 430\tau t_0\bar{t}c_2(\mathbf{t})c_3(\mathbf{t}) + 52\tau t_0\bar{t}c_5(\mathbf{t}) + 16\tau t_0c_2(\mathbf{t})c_4(\mathbf{t}) - 24\tau t_0c_3(\mathbf{t})^2 + 1256\tau \bar{t}^7 \\ \nonumber
    & \: + 4576\tau \bar{t}^5c_2(\mathbf{t}) + 314\tau \bar{t}^4c_3(\mathbf{t}) - 276\tau \bar{t}^3c_2(\mathbf{t})^2 + 2461\tau \bar{t}^3c_4(\mathbf{t}) + 8\tau \bar{t}^2c_2(\mathbf{t})c_3(\mathbf{t}) - 163\tau \bar{t}^2c_5(\mathbf{t}) \\ \nonumber
    & \: - 142\tau \bar{t}c_2(\mathbf{t})c_4(\mathbf{t}) + 4\tau c_2(\mathbf{t})c_5(\mathbf{t}) + 8\tau c_3(\mathbf{t})c_4(\mathbf{t}) + 133t_0^8 - 1201t_0^7\bar{t} + 4189t_0^6\bar{t}^2 - 96t_0^6c_2(\mathbf{t}) \\ \nonumber
    & \: + 8905t_0^5\bar{t}^3 + 853t_0^5\bar{t}c_2(\mathbf{t}) + 58t_0^5c_3(\mathbf{t}) + 18650t_0^4\bar{t}^4 - 3493t_0^4\bar{t}^2c_2(\mathbf{t}) - 439t_0^4\bar{t}c_3(\mathbf{t}) + 15t_0^4c_2(\mathbf{t})^2 \\ \nonumber
    & \: - 54t_0^4c_4(\mathbf{t}) + 57244t_0^3\bar{t}^5 - 2125t_0^3\bar{t}^3c_2(\mathbf{t}) + 2155t_0^3\bar{t}^2c_3(\mathbf{t}) - 120t_0^3\bar{t}c_2(\mathbf{t})^2 + 404t_0^3\bar{t}c_4(\mathbf{t}) \\ \nonumber
    & \: - 12t_0^3c_2(\mathbf{t})c_3(\mathbf{t}) + 49t_0^3c_5(\mathbf{t}) + 67678t_0^2\bar{t}^6 - 11737t_0^2\bar{t}^4c_2(\mathbf{t}) - 1903t_0^2\bar{t}^3c_3(\mathbf{t}) + 469t_0^2\bar{t}^2c_2(\mathbf{t})^2 \\ \nonumber
    & \: - 822t_0^2\bar{t}^2c_4(\mathbf{t}) + 65t_0^2\bar{t}c_2(\mathbf{t})c_3(\mathbf{t}) - 372t_0^2\bar{t}c_5(\mathbf{t}) + 11t_0^2c_2(\mathbf{t})c_4(\mathbf{t}) - 2t_0^2c_3(\mathbf{t})^2 + 49160t_0\bar{t}^7 \\ \nonumber
    & \: - 11768t_0\bar{t}^5c_2(\mathbf{t}) + 4612t_0\bar{t}^4c_3(\mathbf{t}) + 572t_0\bar{t}^3c_2(\mathbf{t})^2 + 1725t_0\bar{t}^3c_4(\mathbf{t}) - 326t_0\bar{t}^2c_2(\mathbf{t})c_3(\mathbf{t}) + 429t_0\bar{t}^2c_5(\mathbf{t}) \\ \nonumber
    & \: - 80t_0\bar{t}c_2(\mathbf{t})c_4(\mathbf{t}) + 20t_0\bar{t}c_3(\mathbf{t})^2 - 9t_0c_2(\mathbf{t})c_5(\mathbf{t}) + 3t_0c_3(\mathbf{t})c_4(\mathbf{t})  + 10504\bar{t}^8 - 5852\bar{t}^6c_2(\mathbf{t}) \\ \nonumber
    & \: + 2666\bar{t}^5c_3(\mathbf{t}) + 300\bar{t}^4c_2(\mathbf{t})^2 - 1313\bar{t}^4c_4(\mathbf{t}) - 170\bar{t}^3c_2(\mathbf{t})c_3(\mathbf{t}) - 1309\bar{t}^3c_5(\mathbf{t}) + 75\bar{t}^2c_2(\mathbf{t})c_4(\mathbf{t}) \\ \nonumber
    & \: + 10\bar{t}^2c_3(\mathbf{t})^2 + 74\bar{t}c_2(\mathbf{t})c_5(\mathbf{t}) - 5\bar{t}c_3(\mathbf{t})c_4(\mathbf{t}) - 4c_3(\mathbf{t})c_5(\mathbf{t}),
\end{align*}
\begin{align*}
K = & \: -7\tau^6 + 15\tau^5t_0 - 15\tau^5\bar{t} - 41\tau^4t_0^2 + 83\tau^4t_0\bar{t} - 209\tau^4\bar{t}^2 + 6\tau^4c_2(\mathbf{t}) + 101\tau^3t_0^3 - 513\tau^3t_0^2\bar{t} \\ \nonumber
    & \: + 717\tau^3t_0\bar{t}^2 - 20\tau^3t_0c_2(\mathbf{t}) - 1737\tau^3\bar{t}^3 + 96\tau^3\bar{t}c_2(\mathbf{t}) - 4\tau^3c_3(\mathbf{t}) + 45\tau^2t_0^4 - 303\tau^2t_0^3\bar{t} \\ \nonumber
    & \: + 2397\tau^2t_0^2\bar{t}^2 - 28\tau^2t_0^2c_2(\mathbf{t}) - 1395\tau^2t_0\bar{t}^3 + 52\tau^2t_0\bar{t}c_2(\mathbf{t}) - 6\tau^2t_0c_3(\mathbf{t}) + 7467\tau^2\bar{t}^4 - 598\tau^2\bar{t}^2c_2(\mathbf{t}) \\ \nonumber
    & \: + 42\tau^2\bar{t}c_3(\mathbf{t}) + 5\tau^2c_2(\mathbf{t})^2 - \tau^2c_4(\mathbf{t}) - 31\tau t_0^5 + 787\tau t_0^4\bar{t} + 1307\tau t_0^3\bar{t}^2 + 7\tau t_0^3c_2(\mathbf{t}) + 1609\tau t_0^2\bar{t}^3  \\ \nonumber
    & \: - 491\tau t_0^2\bar{t}c_2(\mathbf{t}) + 37\tau t_0^2c_3(\mathbf{t}) + 10283\tau t_0\bar{t}^4 - 319\tau t_0\bar{t}^2c_2(\mathbf{t}) - 34\tau t_0\bar{t}c_3(\mathbf{t}) - 3323\tau \bar{t}^5 - 1257\tau \bar{t}^3c_2(\mathbf{t}) \\ \nonumber
    & \: + 139\tau \bar{t}^2c_3(\mathbf{t}) + 78\tau \bar{t}c_2(\mathbf{t})^2 + 2\tau \bar{t}c_4(\mathbf{t}) - 9\tau c_2(\mathbf{t})c_3(\mathbf{t}) + \tau c_5(\mathbf{t}) - 262t_0^6 + 1188t_0^5\bar{t} + 1728t_0^4\bar{t}^2 \\ \nonumber
    & \: + 155t_0^4c_2(\mathbf{t}) + 4435t_0^3\bar{t}^3 - 919t_0^3\bar{t}c_2(\mathbf{t}) - 33t_0^3c_3(\mathbf{t}) + 13308t_0^2\bar{t}^4 - 285t_0^2\bar{t}^2c_2(\mathbf{t}) + 411t_0^2\bar{t}c_3(\mathbf{t}) \\ \nonumber
    & \: - 20t_0^2c_2(\mathbf{t})^2 + t_0^2c_4(\mathbf{t}) + 10077t_0\bar{t}^5 - 2809t_0\bar{t}^3c_2(\mathbf{t}) - 396t_0\bar{t}^2c_3(\mathbf{t}) + 136t_0\bar{t}c_2(\mathbf{t})^2 + 2t_0c_2(\mathbf{t})c_3(\mathbf{t}) \\ \nonumber
    & \: - t_0c_5(\mathbf{t}) + 5244\bar{t}^6 - 1605\bar{t}^4c_2(\mathbf{t}) + 1327\bar{t}^3c_3(\mathbf{t}) + 73\bar{t}^2c_2(\mathbf{t})^2 - \bar{t}^2c_4(\mathbf{t}) - 77\bar{t}c_2(\mathbf{t})c_3(\mathbf{t}) - \bar{t}c_5(\mathbf{t}) + 4c_3(\mathbf{t})^2.
\end{align*}

We can express the GKM functions
$c_i(\mathbf{z}_+(\mathbf{z}_+-\bar{z})) - c_i(\mathbf{t}_+(\mathbf{t}_+-\bar{t}))$
and $c_5(2\mathbf{z}_+-\bar{z}) - c_5(2\mathbf{t}_+-\bar{t})$
as polynomials in $\tau$ and $\omega$ over $H^*(BT)$ as follows:
\begin{align}
\label{c1}
  & c_1(\mathbf{z}_+(\mathbf{z}_+-\bar{z})) - c_1(\mathbf{t}_+(\mathbf{t}_+-\bar{t})) = -2\tau^2 + 4\bar{t}\tau,\\
\label{c2}
  & c_2(\mathbf{z}_+(\mathbf{z}_+-\bar{z})) - c_2(\mathbf{t}_+(\mathbf{t}_+-\bar{t})) \\ \nonumber
= & \: 4\tau^4+(6t_0-10\bar{t})\tau^3+(-8\bar{t}^2-31t_0\bar{t}-4t_0^2+5c_2(\mathbf{t}_+))\tau^2\\ \nonumber
  & \: +(6\bar{t}^3-t_0\bar{t}^2+(-19t_0^2-c_2(\mathbf{t}_+))\bar{t}-6t_0^3+6c_2(\mathbf{t}_+)t_0-3c_3(\mathbf{t}_+))\tau-6\omega,\\
\label{c3}
  & c_3(\mathbf{z}_+(\mathbf{z}_+-\bar{z})) - c_3(\mathbf{t}_+(\mathbf{t}_+-\bar{t})) \\ \nonumber
= & \: 2\tau^6+(6t_0-6\bar{t})\tau^5+(6\bar{t}^2-18t_0\bar{t}+6t_0^2)\tau^4+(-20\bar{t}^3-15t_0\bar{t}^2+(7c_2(\mathbf{t}_+)-37t_0^2)\bar{t}-2t_0^3+4c_2(\mathbf{t}_+)t_0 \\ \nonumber
  & \: -3c_3(\mathbf{t}_+))\tau^3 +(6\bar{t}^4-27t_0\bar{t}^3+(3c_2(\mathbf{t}_+)-39t_0^2)\bar{t}^2+(-42t_0^3+18c_2(\mathbf{t}_+)t_0+3c_3(\mathbf{t}_+))\bar{t} \\ \nonumber
  & \: -8t_0^4+10c_2(\mathbf{t}_+)t_0^2-4c_3(\mathbf{t}_+)t_0+2c_4(\mathbf{t}_+)-2c_2(\mathbf{t}_+)^2)\tau^2+(6t_0\bar{t}^4+(-7t_0^2-2c_2(\mathbf{t}_+))\bar{t}^3 +(-18t_0^3-6c_3(\mathbf{t}_+))\bar{t}^2 \\ \nonumber
  & \: +(-17t_0^4+10c_2(\mathbf{t}_+)t_0^2-2c_3(\mathbf{t}_+)t_0+c_4(\mathbf{t}_+)+c_2(\mathbf{t}_+)^2)\bar{t}-4t_0^5+6c_2(\mathbf{t}_+)t_0^3-2c_3(\mathbf{t}_+)t_0^2\\ \nonumber
  & \: -2c_2(\mathbf{t}_+)^2t_0-5c_5(\mathbf{t}_+)+c_2(\mathbf{t}_+)c_3(\mathbf{t}_+))\tau -6\omega \tau^2 +12\omega\bar{t}\tau -14\bar{t}^2\omega-10t_0\bar{t}\omega+(2c_2(\mathbf{t}_+)-4t_0^2)\omega,\\
\label{c4}
  & c_4(\mathbf{z}_+(\mathbf{z}_+-\bar{z})) - c_4(\mathbf{t}_+(\mathbf{t}_+-\bar{t})) \\ \nonumber
= & \: -3\omega^2+(18\tau^4+(6t_0-66\bar{t})\tau^3+(64\bar{t}^2-46t_0\bar{t}-10t_0^2+8c_2(\mathbf{t}_+))\tau^2+(-26\bar{t}^3+29t_0\bar{t}^2+(-7t_0^2-7c_2(\mathbf{t}_+))\bar{t}\\ \nonumber
  & \: -6t_0^3+6c_2(\mathbf{t}_+)t_0-3c_3(\mathbf{t}_+))\tau-14\bar{t}^4-55t_0\bar{t}^3+(9c_2(\mathbf{t}_+)-47t_0^2)\bar{t}^2+(-20t_0^3+8c_2(\mathbf{t}_+)t_0-3c_3(\mathbf{t}_+))\bar{t} \\ \nonumber
  & \: -4t_0^4+4c_2(\mathbf{t}_+)t_0^2-4c_3(\mathbf{t}_+)t_0+2c_4(\mathbf{t}_+))\omega -7\tau^8+(38\bar{t}-18t_0)\tau^7+(-55\bar{t}^2+129t_0\bar{t}-3t_0^2-9c_2(\mathbf{t}_+))\tau^6 \\ \nonumber
  & \: +(22\bar{t}^3-207t_0\bar{t}^2+(137t_0^2+19c_2(\mathbf{t}_+))\bar{t}+28t_0^3-26c_2(\mathbf{t}_+)t_0+9c_3(\mathbf{t}_+))\tau^5 +(-32\bar{t}^4+9t_0\bar{t}^3 +(9c_2(\mathbf{t}_+)\\ \nonumber
  & \: -348t_0^2)\bar{t}^2+(-9t_0^3+77c_2(\mathbf{t}_+)t_0-24c_3(\mathbf{t}_+))\bar{t}+23t_0^4-18c_2(\mathbf{t}_+)t_0^2+15c_3(\mathbf{t}_+)t_0-6c_4(\mathbf{t}_+)-2c_2(\mathbf{t}_+)^2)\tau^4 \\ \nonumber
  & \: +(6\bar{t}^5-77t_0\bar{t}^4+(7c_2(\mathbf{t}_+)-83t_0^2)\bar{t}^3+(-341t_0^3+40c_2(\mathbf{t}_+)t_0+7c_3(\mathbf{t}_+))\bar{t}^2+(-127t_0^4+125c_2(\mathbf{t}_+)t_0^2 \\ \nonumber
  & \: -56c_3(\mathbf{t}_+)t_0+14c_4(\mathbf{t}_+)-4c_2(\mathbf{t}_+)^2)\bar{t}-6t_0^5+14c_2(\mathbf{t}_+)t_0^3-c_3(\mathbf{t}_+)t_0^2+(-2c_4(\mathbf{t}_+)-8c_2(\mathbf{t}_+)^2)t_0 \\ \nonumber
  & \: +8c_5(\mathbf{t}_+)+4c_2(\mathbf{t}_+)c_3(\mathbf{t}_+))\tau^3+(12t_0\bar{t}^5+(-58t_0^2-2c_2(\mathbf{t}_+))\bar{t}^4+(-105t_0^3+9c_2(\mathbf{t}_+)t_0-10c_3(\mathbf{t}_+))\bar{t}^3 \\ \nonumber
  & \: +(-191t_0^4+49c_2(\mathbf{t}_+)t_0^2-4c_3(\mathbf{t}_+)t_0-4c_4(\mathbf{t}_+)+c_2(\mathbf{t}_+)^2)\bar{t}^2+(-96t_0^5+94c_2(\mathbf{t}_+)t_0^3-50c_3(\mathbf{t}_+)t_0^2 \\ \nonumber
  & \: +(12c_4(\mathbf{t}_+)-7c_2(\mathbf{t}_+)^2)t_0-19c_5(\mathbf{t}_+)+3c_2(\mathbf{t}_+)c_3(\mathbf{t}_+))\bar{t}-13t_0^6+23c_2(\mathbf{t}_+)t_0^4-13c_3(\mathbf{t}_+)t_0^3 \\ \nonumber
  & \: +(5c_4(\mathbf{t}_+)-10c_2(\mathbf{t}_+)^2)t_0^2+(10c_2(\mathbf{t}_+)c_3(\mathbf{t}_+)-5c_5(\mathbf{t}_+))t_0-c_2(\mathbf{t}_+)c_4(\mathbf{t}_+)-2c_3(\mathbf{t}_+)^2)\tau^2 \\ \nonumber
  & \: +(6t_0^2\bar{t}^5+(-13t_0^3-2c_2(\mathbf{t}_+)t_0+2c_3(\mathbf{t}_+))\bar{t}^4+(-35t_0^4+2c_2(\mathbf{t}_+)t_0^2-5c_3(\mathbf{t}_+)t_0+10c_4(\mathbf{t}_+))\bar{t}^3 \\ \nonumber
  & \: +(-46t_0^5+18c_2(\mathbf{t}_+)t_0^3-7c_3(\mathbf{t}_+)t_0^2+(12c_4(\mathbf{t}_+)+c_2(\mathbf{t}_+)^2)t_0+2c_5(\mathbf{t}_+)-c_2(\mathbf{t}_+)c_3(\mathbf{t}_+))\bar{t}^2 \\ \nonumber
  & \: +(-24t_0^6+27c_2(\mathbf{t}_+)t_0^4-17c_3(\mathbf{t}_+)t_0^3+(5c_4(\mathbf{t}_+)-3c_2(\mathbf{t}_+)^2)t_0^2+(2c_2(\mathbf{t}_+)c_3(\mathbf{t}_+)-5c_5(\mathbf{t}_+))t_0 \\ \nonumber
  & \: -4c_2(\mathbf{t}_+)c_4(\mathbf{t}_+)+c_3(\mathbf{t}_+)^2)\bar{t}-4t_0^7+8c_2(\mathbf{t}_+)t_0^5-6c_3(\mathbf{t}_+)t_0^4+(2c_4(\mathbf{t}_+)-4c_2(\mathbf{t}_+)^2)t_0^3 \\ \nonumber
  & \: +(6c_2(\mathbf{t}_+)c_3(\mathbf{t}_+)-6c_5(\mathbf{t}_+))t_0^2+(-2c_2(\mathbf{t}_+)c_4(\mathbf{t}_+)-2c_3(\mathbf{t}_+)^2)t_0+3c_2(\mathbf{t}_+)c_5(\mathbf{t}_+)+c_3(\mathbf{t}_+)c_4(\mathbf{t}_+))\tau, \\
\label{c5}
  & c_5(2\mathbf{z}_+-\bar{z}) - c_5(2\mathbf{t}_+-\bar{t})\\ \nonumber
= & \: 48\omega t-48\omega\bar{t}
	-21t^5+(57\bar{t}-48t_0)t^4+(14\bar{t}^2+196t_0\bar{t}-8t_0^2-20c_2(\mathbf{t}_+))t^3 \\ \nonumber
  &   +(18\bar{t}^3+60t_0\bar{t}^2+(240t_0^2-12c_2(\mathbf{t}_+))\bar{t}+48t_0^3-48c_2(\mathbf{t}_+)t_0+24c_3(\mathbf{t}_+))t^2\\ \nonumber
  &	+(-9\bar{t}^4+12t_0\bar{t}^3+(56t_0^2+4c_2(\mathbf{t}_+))\bar{t}^2+(112t_0^3-16c_2(\mathbf{t}_+)t_0)\bar{t} \\ \nonumber
  &   +32t_0^4-32c_2(\mathbf{t}_+)t_0^2+32c_3(\mathbf{t}_+)t_0-16c_4(\mathbf{t}_+))t.
\end{align}

Note that we have $c_1(\mathbf{t}_+) = 2t_0 +3\bar{t}$
and that, for $\mathbf{x}= \{ x_1, \ldots, x_5 \} \subset R$, $y \in R$,
and $1 \leq k \leq 5$, a straight forward calculation shows
\begin{align}
  & c_k(\mathbf{x}(\mathbf{x}-y))\\ \nonumber
= & \: \sum_{i=1}^k (-1)^i y^i (c_k(\mathbf{x})c_{k-i}(\mathbf{x}) -\binom{i+2}{1}c_{k+1}(\mathbf{x})c_{k-i-1}(\mathbf{x}) +(-\binom{i+4}{2} +\binom{i+2}{1}\binom{i+4}{1})c_{k+2}(\mathbf{x})c_{k-i-2}(\mathbf{x})).
\end{align}

By erasing $c_k(\mathbf{z}_+)$ from \eqref{c1} and dividing it by $8$, we obtain
\begin{align}
\label{c1/8}
& \gamma_2 - \gamma_1(\gamma_1 +c_1(\mathbf{t}_+)) -\gamma_1\bar{t}= 0.
\end{align}
By erasing $c_1(\mathbf{z}_+)$ and $c_2(\mathbf{z}_+)$ from \eqref{c2}, using \eqref{c1/8},
and dividing \eqref{c2} by $6$, we obtain
\begin{align}
\label{c2/6}
& \sum_{j=1}^{4}(-1)^j \gamma_j (\gamma_{4-j} + c_{4-j}(\mathbf{t}_+)) +\omega - \gamma_1^3\tau + \gamma_1\tau^3 - \gamma_3\tau \\ \nonumber
& - 4t_0^2\bar{t}\gamma_1 - 3t_0^2\gamma_1\tau - 4t_0\bar{t}^2\gamma_1 - 2t_0\bar{t}\gamma_1^2 - 9t_0\bar{t}\gamma_1\tau - 3t_0\gamma_1^2\tau - 2\bar{t}^3\gamma_1 \\ \nonumber
& - \bar{t}^2\gamma_1^2 - 3\bar{t}^2\gamma_1\tau + \bar{t}\gamma_1^3 - 3\bar{t}\gamma_1^2\tau - 3\bar{t}\gamma_1\tau^2 + 2\bar{t}\gamma_3 + c_2(\mathbf{t}_+)\gamma_1\tau  = 0.
\end{align}
By erasing $c_1(\mathbf{z}_+), \ldots, c_4(\mathbf{z}_+)$ from \eqref{c5}, using \eqref{c1/8},
and dividing \eqref{c5} by $32$, we obtain
\begin{align}
\label{c5/32}
& c_5(\mathbf{z}_+) - c_5(\mathbf{t}_+)  + \gamma_4\tau + \tau^5 - 2\tau\omega -t_0^4\tau - 3t_0^3\bar{t}\tau - 2t_0^3\tau^2 - 9t_0^2\bar{t}\tau^2 \\ \nonumber
& + t_0^2c_2(\mathbf{t}_+)\tau - 2t_0\bar{t}^3\gamma_1 + 2t_0\bar{t}^2\gamma_1\tau - 9t_0\bar{t}\tau^3 + 2t_0c_2(\mathbf{t}_+)\tau^2 - t_0c_3(\mathbf{t}_+)\tau + 2t_0\tau^4 \\ \nonumber
& - \bar{t}^4\gamma_1 + \bar{t}^3\gamma_1\tau - \bar{t}\gamma_4 - 3\bar{t}\tau^4 + 2\bar{t}\omega + c_2(\mathbf{t}_+)\tau^3 - c_3(\mathbf{t}_+)\tau^2 + c_4(\mathbf{t}_+)\tau = 0.
\end{align}
Now we can erase $c_5(\mathbf{z}_+)$ by \eqref{c5/32}.
By erasing $c_1(\mathbf{z}_+), \ldots, c_5(\mathbf{z}_+)$ from \eqref{c3}, using \eqref{c1/8},
and dividing \eqref{c3} by $4$, we obtain
\begin{align}
\label{c3/4}
& \sum_{j=1}^{6}(-1)^j \gamma_j (\gamma_{6-j} + c_{6-j}(\mathbf{t}_+))
 +2\tau^6 +2\gamma_1\tau^5 +\gamma_1^2 \tau^4 -\gamma_1^3 \tau^3 +(-\gamma_1\gamma_3+2\gamma_4-2\gamma_1^4)\tau^2 \\ \nonumber
& +(-3\gamma_1\omega -3\gamma_1^2\gamma_3+3\gamma_1\gamma_4)\tau -\gamma_1^2 \omega -4\tau^2\omega
 -\gamma_1\bar{t}^5 +\bigl(-4\gamma_1\tau-8\gamma_1t_0+\gamma_1^2 \bigr)\bar{t}^4 \\ \nonumber
& +\bigl(8\gamma_1\tau^2+(3\gamma_1^2-3\gamma_1t_0)\tau-8\gamma_1t_0^2+6\gamma_1^2t_0-2\gamma_3-2\gamma_1^3 \bigr)\bar{t}^3 +\bigl(-2\gamma_1\tau^3+(10\gamma_1t_0+2\gamma_1^2)\tau^2 \\ \nonumber
& +(4\gamma_1t_0^2+10\gamma_1^2t_0+2\gamma_3+5\gamma_1^3)\tau -2\omega+8\gamma_1^2t_0^2+(-2\gamma_3+2\gamma_1^3+2c_2(\mathbf{t}_+)\gamma_1)t_0+3\gamma_1\gamma_3 \\ \nonumber
& -3\gamma_4-2\gamma_1^4-c_2(\mathbf{t}_+)\gamma_1^2+c_3(\mathbf{t}_+)\gamma_1 \bigr)\bar{t}^2+ \bigl(-6\tau^5+(-24t_0-6\gamma_1)\tau^4+(-36t_0^2-24\gamma_1t_0-6\gamma_1^2)\tau^3 \\ \nonumber
& +(-24t_0^3-24\gamma_1t_0^2-14\gamma_1^2t_0-4\gamma_1^3)\tau^2+(4\omega-6t_0^4-10\gamma_1t_0^3-2\gamma_1^2t_0^2+(2\gamma_3+6\gamma_1^3+c_2(\mathbf{t}_+)\gamma_1)t_0 \\ \nonumber
& -2\gamma_1\gamma_3-2\gamma_4+3\gamma_1^4-2c_3(\mathbf{t}_+)\gamma_1)\tau + (t_0+\gamma_1)\omega+(2\gamma_1\gamma_3-4\gamma_4)t_0+(3\gamma_1^2+c_2(\mathbf{t}_+))\gamma_3  \\ \nonumber
& -3\gamma_1\gamma_4+c_3(\mathbf{t}_+)\gamma_1^2-c_4(\mathbf{t}_+)\gamma_1 \bigr)\bar{t} +6t_0\tau^5+(4t_0^2+5\gamma_1t_0+2c_2(\mathbf{t}_+))\tau^4  \\ \nonumber
& +\bigl(-4t_0^3-\gamma_1t_0^2+(6c_2(\mathbf{t}_+)-2\gamma_1^2)t_0-\gamma_1^3+2c_2(\mathbf{t}_+)\gamma_1-2c_3(\mathbf{t}_+) \bigr)\tau^3
+ \bigl(-6t_0^4-7\gamma_1t_0^3  \\ \nonumber
& +(6c_2(\mathbf{t}_+)-10\gamma_1^2)t_0^2+(-\gamma_3-8\gamma_1^3+5c_2(\mathbf{t}_+)\gamma_1-4c_3(\mathbf{t}_+))t_0 \\ \nonumber
& +c_2(\mathbf{t}_+)\gamma_1^2-2c_3(\mathbf{t}_+)\gamma_1+2c_4(\mathbf{t}_+) \bigr)\tau^2+ \bigl(-3t_0\omega-2t_0^5-3\gamma_1t_0^4+(2c_2(\mathbf{t}_+)-\gamma_1^2)t_0^3  \\ \nonumber
& +(-\gamma_3+3c_2(\mathbf{t}_+)\gamma_1-2c_3(\mathbf{t}_+))t_0^2+(-6\gamma_1\gamma_3+3\gamma_4+c_2(\mathbf{t}_+)\gamma_1^2-5c_3(\mathbf{t}_+)\gamma_1+2c_4(\mathbf{t}_+))t_0  \\ \nonumber
& -2c_3(\mathbf{t}_+)\gamma_1^2+2c_4(\mathbf{t}_+)\gamma_1-2c_5(\mathbf{t}_+)\bigr)\tau+(-t_0^2-2\gamma_1t_0)\omega = 0.
\end{align}
Now we can erase $\gamma_3^2$ by \eqref{c3/4}.
By erasing $c_1(\mathbf{z}_+), \ldots, c_5(\mathbf{z}_+)$, and $\gamma_3^2$ from \eqref{c4},
using \eqref{c1/8}, and dividing \eqref{c4} by $4$, we obtain
\begin{align}
\label{c4/4}
& \sum_{j=1}^{8}(-1)^j \gamma_j (\gamma_{8-j} + c_{8-j}(\mathbf{t}_+))
 +\omega^2 - 5\tau^4\omega - 4\gamma_1^2\tau^2\omega - 2\gamma_1\tau^3\omega - 2\gamma_3\tau\omega - \gamma_4\omega \\ \nonumber
& + 2\tau^8 + \gamma_1\tau^7 + 2\gamma_1^2\tau^6 + \gamma_3\tau^5  + \gamma_4\tau^4 + \gamma_1\gamma_4\tau^3
 + 2\gamma_1^2\gamma_4\tau^2 + \gamma_3\gamma_4\tau \\ \nonumber
& -2 \gamma_1 \bar{t}^7+(4 \gamma_1 \tau-6 \gamma_1 t_0+3 \gamma_1^2) \bar{t}^6+(4 \gamma_1 \tau^2+(16 \gamma_1 t_0-5 \gamma_1^2) -4 \gamma_1 t_0^2+10 \gamma_1^2 t_0+2 \gamma_1^3+c_2(\mathbf{t}_+) \gamma_1) \bar{t}^5 \\ \nonumber
& +(-7 \gamma_1 \tau^3+(-4 \gamma_1 t_0-5 \gamma_1^2) \tau^2+(8 \gamma_1 t_0^2-20 \gamma_1^2 t_0+2 \gamma_3-2 \gamma_1^3-c_2(\mathbf{t}_+) \gamma_1) \tau+8 \omega+8 \gamma_1^2 t_0^2 \\ \nonumber
& +(4 \gamma_1^3+2 c_2(\mathbf{t}_+) \gamma_1) t_0 -2 \gamma_4-\gamma_1 \gamma_3-c_3(\mathbf{t}_+) \gamma_1) \bar{t}^4+(3 \tau^5+(12 t_0+12 \gamma_1) \tau^4+(18 t_0^2+15 \gamma_1 t_0+2 \gamma_1^2) \tau^3 \\ \nonumber
& +(12 t_0^3+4 \gamma_1 t_0^2-10 \gamma_1^2 t_0-\gamma_3-8 \gamma_1^3-c_2(\mathbf{t}_+) \gamma_1) \tau^2 +(3 t_0^4-3 \gamma_1 t_0^3-20 \gamma_1^2 t_0^2+(2 \gamma_3-10 \gamma_1^3-3 c_2(\mathbf{t}_+) \gamma_1) t_0 \\ \nonumber
& +6 \gamma_4-2 \gamma_1 \gamma_3+2 \gamma_1^4+c_2(\mathbf{t}_+) \gamma_1^2-c_3(\mathbf{t}_+) \gamma_1) \tau+(18 t_0-5 \gamma_1) \omega+(-2 \gamma_4-2 \gamma_1 \gamma_3-2 c_3(\mathbf{t}_+) \gamma_1) t_0 \\ \nonumber
& +4 \gamma_1 \gamma_4+c_4(\mathbf{t}_+) \gamma_1) \bar{t}^3+(17 \tau^6+(78 t_0-3 \gamma_1) \tau^5+(142 t_0^2+10 \gamma_1 t_0+12 \gamma_1^2-c_2(\mathbf{t}_+)) \tau^4 \\ \nonumber
& +(128 t_0^3+31 \gamma_1 t_0^2+(37 \gamma_1^2-3 c_2(\mathbf{t}_+)) t_0+\gamma_3+5 \gamma_1^3-3 c_2(\mathbf{t}_+) \gamma_1+c_3(\mathbf{t}_+)) \tau^3+(-15 \omega+57 t_0^4 +26 \gamma_1 t_0^3 \\ \nonumber
& +(24 \gamma_1^2-3 c_2(\mathbf{t}_+)) t_0^2+(-6 \gamma_1^3-5 c_2(\mathbf{t}_+) \gamma_1+2 c_3(\mathbf{t}_+)) t_0-\gamma_4+4 \gamma_1 \gamma_3-4 \gamma_1^4-c_2(\mathbf{t}_+) \gamma_1^2+5 c_3(\mathbf{t}_+) \gamma_1 \\ \nonumber
& -c_4(\mathbf{t}_+)) \tau^2 +((7 \gamma_1-15 t_0) \omega+10 t_0^5+6 \gamma_1 t_0^4+(6 \gamma_1^2-c_2(\mathbf{t}_+)) t_0^3+c_3(\mathbf{t}_+) t_0^2+(6 \gamma_4+2 c_3(\mathbf{t}_+) \gamma_1-c_4(\mathbf{t}_+)) t_0 \\ \nonumber
& -5 \gamma_1 \gamma_4+(-3 \gamma_1^2-c_2(\mathbf{t}_+)) \gamma_3-c_3(\mathbf{t}_+) \gamma_1^2-3 c_4(\mathbf{t}_+) \gamma_1+c_5(\mathbf{t}_+)) \tau+(12 t_0^2-6 \gamma_1 t_0-4 \gamma_1^2-4 c_2(\mathbf{t}_+)) \omega \\ \nonumber
& +(6 \gamma_1 \gamma_4+2 c_4(\mathbf{t}_+) \gamma_1) t_0+(2 \gamma_1^2+c_2(\mathbf{t}_+)) \gamma_4+2 c_5(\mathbf{t}_+) \gamma_1) \bar{t}^2+(-12 \tau^7+(-45 t_0-3 \gamma_1) \tau^6 \\ \nonumber
& +(-54 t_0^2-24 \gamma_1 t_0-9 \gamma_1^2-9 c_2(\mathbf{t}_+))\tau^5+(-6 t_0^3-48 \gamma_1 t_0^2+(-20 \gamma_1^2-33 c_2(\mathbf{t}_+)) t_0-3 \gamma_3+\gamma_1^3+6 c_3(\mathbf{t}_+)) \tau^4 \\ \nonumber
& +(20 \omega+36 t_0^4-36 \gamma_1 t_0^3+(-8 \gamma_1^2-45 c_2(\mathbf{t}_+)) t_0^2+(-8 \gamma_3+8 \gamma_1^3-3 c_2(\mathbf{t}_+) \gamma_1+15 c_3(\mathbf{t}_+)) t_0-4 \gamma_4 \\ \nonumber
& +\gamma_1 \gamma_3+2 \gamma_1^4-3 c_2(\mathbf{t}_+) \gamma_1^2-6 c_4(\mathbf{t}_+)) \tau^3+((22 t_0+\gamma_1) \omega+27 t_0^5-7 \gamma_1 t_0^4+(-\gamma_1^2-27 c_2(\mathbf{t}_+)) t_0^3 \\ \nonumber
& +(-8 \gamma_3-5 c_2(\mathbf{t}_+) \gamma_1+12 c_3(\mathbf{t}_+)) t_0^2+(-10 \gamma_4+6 \gamma_1 \gamma_3-5 c_2(\mathbf{t}_+) \gamma_1^2+5 c_3(\mathbf{t}_+) \gamma_1-9 c_4(\mathbf{t}_+)) t_0 \\ \nonumber
& -2 \gamma_1 \gamma_4+3 \gamma_1^2 \gamma_3+4 c_3(\mathbf{t}_+) \gamma_1^2+6 c_5(\mathbf{t}_+)) \tau^2+((5 t_0^2+14 \gamma_1 t_0+9 \gamma_1^2+4 c_2(\mathbf{t}_+)) \omega+6 t_0^6+2 \gamma_1 t_0^5 \\ \nonumber
& +(2 \gamma_1^2-6 c_2(\mathbf{t}_+)) t_0^4+(-3 \gamma_3-2 c_2(\mathbf{t}_+) \gamma_1+3 c_3(\mathbf{t}_+)) t_0^3+(-4 \gamma_4-2 c_2(\mathbf{t}_+) \gamma_1^2+2 c_3(\mathbf{t}_+) \gamma_1-3 c_4(\mathbf{t}_+)) t_0^2 \\ \nonumber
& +(-10 \gamma_1 \gamma_4+2 c_3(\mathbf{t}_+) \gamma_1^2-4 c_4(\mathbf{t}_+) \gamma_1+3 c_5(\mathbf{t}_+)) t_0+(-6 \gamma_1^2-2 c_2(\mathbf{t}_+)) \gamma_4+\gamma_3^2+c_3(\mathbf{t}_+) \gamma_3-3 c_4(\mathbf{t}_+) \gamma_1^2 \\ \nonumber
& -c_5(\mathbf{t}_+) \gamma_1) \tau+(5 t_0^3-2 c_2(\mathbf{t}_+) t_0+2 \gamma_3+2 c_3(\mathbf{t}_+)) \omega+2 c_5(\mathbf{t}_+) \gamma_1 t_0+(-\gamma_3-c_3(\mathbf{t}_+)) \gamma_4+2 c_5(\mathbf{t}_+) \gamma_1^2) \bar{t} \\ \nonumber
& +6 t_0 \tau^7 +(3 t_0^2+6 \gamma_1 t_0+3 c_2(\mathbf{t}_+)) \tau^6+(-7 t_0^3+8 \gamma_1 t_0^2+(4 \gamma_1^2+9 c_2(\mathbf{t}_+)) t_0+c_2(\mathbf{t}_+) \gamma_1-2 c_3(\mathbf{t}_+)) \tau^5 \\ \nonumber
& +(-8 t_0^4-2 \gamma_1 t_0^3+7 c_2(\mathbf{t}_+) t_0^2+(2 \gamma_3+6 c_2(\mathbf{t}_+) \gamma_1-4 c_3(\mathbf{t}_+)) t_0+2 c_2(\mathbf{t}_+) \gamma_1^2-c_3(\mathbf{t}_+) \gamma_1+2 c_4(\mathbf{t}_+)+c_2(\mathbf{t}_+)^2) \tau^4 \\ \nonumber
& +(-4 t_0 \omega-9 \gamma_1 t_0^4+(-4 \gamma_1^2-3 c_2(\mathbf{t}_+)) t_0^3+(9 c_2(\mathbf{t}_+) \gamma_1-c_3(\mathbf{t}_+)) t_0^2+(2 \gamma_4+4 c_2(\mathbf{t}_+) \gamma_1^2-5 c_3(\mathbf{t}_+) \gamma_1 \\ \nonumber
& +2 c_4(\mathbf{t}_+)+3 c_2(\mathbf{t}_+)^2) t_0+c_2(\mathbf{t}_+) \gamma_3-2 c_3(\mathbf{t}_+) \gamma_1^2+c_4(\mathbf{t}_+) \gamma_1-2 c_5(\mathbf{t}_+)-c_2(\mathbf{t}_+) c_3(\mathbf{t}_+)) \tau^3 \\ \nonumber
& +((2 t_0^2-8 \gamma_1 t_0-3 c_2(\mathbf{t}_+)) \omega+3 t_0^6-4 \gamma_1 t_0^5+(-2 \gamma_1^2-6 c_2(\mathbf{t}_+)) t_0^4+(-2 \gamma_3+4 c_2(\mathbf{t}_+) \gamma_1+2 c_3(\mathbf{t}_+)) t_0^3 \\ \nonumber
& +(2 c_2(\mathbf{t}_+) \gamma_1^2-4 c_3(\mathbf{t}_+) \gamma_1-c_4(\mathbf{t}_+)+3 c_2(\mathbf{t}_+)^2) t_0^2+(4 \gamma_1 \gamma_4+2 c_2(\mathbf{t}_+) \gamma_3-2 c_3(\mathbf{t}_+) \gamma_1^2 \\ \nonumber
& +4 c_4(\mathbf{t}_+) \gamma_1-2 c_2(\mathbf{t}_+) c_3(\mathbf{t}_+)) t_0+c_2(\mathbf{t}_+) \gamma_4-c_3(\mathbf{t}_+) \gamma_3+2 c_4(\mathbf{t}_+) \gamma_1^2-c_5(\mathbf{t}_+) \gamma_1+c_2(\mathbf{t}_+) c_4(\mathbf{t}_+)) \tau^2 \\ \nonumber
& +((2 t_0^3-2 c_2(\mathbf{t}_+) t_0) \omega+t_0^7-2 c_2(\mathbf{t}_+) t_0^5+(c_3(\mathbf{t}_+)-\gamma_3) t_0^4+(-\gamma_4-c_4(\mathbf{t}_+)+c_2(\mathbf{t}_+)^2) t_0^3 \\ \nonumber
& +(c_2(\mathbf{t}_+) \gamma_3+c_5(\mathbf{t}_+)-c_2(\mathbf{t}_+) c_3(\mathbf{t}_+)) t_0^2+(c_2(\mathbf{t}_+) \gamma_4-c_3(\mathbf{t}_+) \gamma_3-4 c_5(\mathbf{t}_+) \gamma_1+c_2(\mathbf{t}_+) c_4(\mathbf{t}_+)) t_0 \\ \nonumber
& +c_4(\mathbf{t}_+) \gamma_3-2 c_5(\mathbf{t}_+) \gamma_1^2-c_2(\mathbf{t}_+) c_5(\mathbf{t}_+)) \tau+(t_0^4-c_2(\mathbf{t}_+) t_0^2+c_3(\mathbf{t}_+) t_0-c_4(\mathbf{t}_+)) \omega = 0.
\end{align}

\begin{landscape}
\begin{table}[p]
\tablinesep =15pt
  \caption{}
  \centering
  \begin{tabular}{c|c|c}
  \label{rho_lambda}
	$\lambda$															& $\bar{t} - \lambda$												& $\rho_\lambda$ \\ \hline
   	$\bar{t}$																& $0$ 													& $e$					 \\
	${\displaystyle-e_n +\frac{1}{2}\sum_{i=1}^5 e_i + \frac{1}{6}(e_8 - e_7 - e_6)}$ 			& ${\displaystyle e_n + e_8 - \frac{1}{2}\sum_{i=1}^8 e_i}$ 				& ${\displaystyle  r \biggl(e_n + e_8 - \frac{1}{2}\sum_{i=1}^8 e_i\biggr)}$	\\ 
	${\displaystyle-e_l -e_m -e_n +\frac{1}{2}\sum_{i=1}^5 e_i + \frac{1}{6}(e_8 - e_7 - e_6)}$ 	& ${\displaystyle e_l +e_m +e_n + e_8 - \frac{1}{2}\sum_{i=1}^8 e_i}$			& ${\displaystyle  r \biggl(e_l +e_m +e_n + e_8 - \frac{1}{2}\sum_{i=1}^8 e_i\biggr)}$	\\
	${\displaystyle-\frac{1}{2}\sum_{i=1}^5 e_i + \frac{1}{6}(e_8 - e_7 - e_6)}$ 				& ${\displaystyle  \frac{1}{2}\sum_{i=1}^5 e_i + \frac{1}{2}(e_8 - e_7 - e_6)}$	& ${\displaystyle  r \biggl(\frac{1}{2}\sum_{i=1}^5 e_i + \frac{1}{2}(e_8 - e_7 - e_6)\biggr)}$	\\
	${\displaystyle e_j - \frac{1}{3}(e_8 -e_7 -e_6)}$ 								& $-e_j + e_8 -e_7 -e_6$											& 
	$\begin{aligned}
	 & r \biggl(e_k - \frac{1}{2}\sum_{i=1}^5 e_i + \frac{1}{2}(e_8 - e_7 - e_6)\biggr)\\
	 &\cdot  r \biggl(-e_j -e_k + \frac{1}{2}\sum_{i=1}^5 e_i + \frac{1}{2}(e_8 - e_7 - e_6)\biggr)
	\end{aligned}$\\
	${\displaystyle -e_j - \frac{1}{3}(e_8 -e_7 -e_6)}$ 								& $e_j + e_8 -e_7 -e_6$											& 
      $\begin{aligned}
	 & r (e_j + e_k) \cdot  r \biggl(e_j - \frac{1}{2}\sum_{i=1}^5 e_i + \frac{1}{2}(e_8 - e_7 - e_6)\biggr)\\
	 &\cdot  r \biggl(-e_j -e_k + \frac{1}{2}\sum_{i=1}^5 e_i + \frac{1}{2}(e_8 - e_7 - e_6)\biggr)
	\end{aligned}$\\
  \end{tabular}\\
where $1 \leq l$, $m$, $n \leq 5$ are distinct,
$1 \leq j$, $k \leq 5$ are distinct,
and, for $\alpha \in \Phi$, $r(\alpha)$ denotes the reflection associated with $\alpha$.
\end{table}
\end{landscape}

\begin{table}[h]
\tablinesep =15pt
  \caption{}
  \begin{tabular}{c|c}
  \label{omega}
    $\lambda$														& $\omega(\rho_\lambda)$ \\ \hline
    $\bar{t}$   														& $0$ \\
    ${\displaystyle -e_n +\frac{1}{2}\sum_{i=1}^5 e_i + \frac{1}{6}(e_8-e_7-e_6)}$  		& $0$ \\
    ${\displaystyle -e_l -e_m -e_n +\frac{1}{2}\sum_{i=1}^5 e_i + \frac{1}{6}(e_8-e_7-e_6)}$	& $(e_l+e_m)(e_m+e_n)(e_n+e_l)(\bar{t}-\lambda)$ \\
    ${\displaystyle -\frac{1}{2}\sum_{i=1}^5 e_i + \frac{1}{6}(e_8-e_7-e_6)}$ 			& $(\bar{t}-\lambda)(c_1(\mathbf{e})c_2(\mathbf{e})-c_3(\mathbf{e}))$ \\
    ${\displaystyle e_j - \frac{1}{3}(e_8 -e_7 -e_6)}$ 							& ${\displaystyle \prod_{n\neq j}\biggl( -e_n +\frac{1}{2}\sum_{i=1}^5 e_i + \frac{1}{6}(e_8-e_7-e_6) -\lambda \biggr)}$ \\
    ${\displaystyle -e_j - \frac{1}{3}(e_8 -e_7 -e_6)}$ 							& 
	$\begin{aligned}
	 & \biggl(\frac{1}{2}\sum_{i=1}^5 e_i + \frac{1}{2}(e_8 - e_7 - e_6)\biggr)\\
	 &\cdot  \biggl\{\prod_{n \neq j,k} \Bigl(-e_n -e_k +\frac{1}{2}\sum_{i=1}^5 e_i + \frac{1}{2}(e_8 - e_7 - e_6) \Bigr)\\
       & \quad +(e_j+e_k)(\bar{t}-\lambda)\Bigl(-e_k +\frac{1}{2}\sum_{i=1}^5 e_i + \frac{1}{2}(e_8 - e_7 - e_6)\Bigr)\biggr\}
	\end{aligned}$ \\
  \end{tabular}\\
where $1 \leq l$, $m$, $n \leq 5$ are distinct, $1 \leq j$, $k \leq 5$ are distinct, and $\mathbf{e}$ denotes $\{ e_1, \ldots, e_5 \}$.
\end{table}

\begin{table}[h]
\tablinesep =15pt
  \caption{}
  \centering
  \scalebox{0.75}{ 
  \begin{tabular}{c|c}
  \label{EIII seq.}
	vertices $\rho_\lambda$ for $\lambda$  & GKM functions \\ \hline \hline
	$\bar{t}$ & $1$ \\
	$-e_1 +\frac{1}{2}\sum_{i=1}^5 e_i +\frac{1}{4}\bar{t}$ & $\bar{z} - \bar{t}$ \\
	$-e_2 +\frac{1}{2}\sum_{i=1}^5 e_i +\frac{1}{4}\bar{t}$ & $(\bar{z} - \bar{t})(\bar{z} -(-e_1 +\frac{1}{2}\sum_{i=1}^5 e_i +\frac{1}{4}\bar{t}))$ \\
	$-e_3 +\frac{1}{2}\sum_{i=1}^5 e_i +\frac{1}{4}\bar{t}$ & $(\bar{z} - \bar{t}) \prod_{j=1}^2(\bar{z} -(-e_j +\frac{1}{2}\sum_{i=1}^5 e_i +\frac{1}{4}\bar{t}))$ \\ 
	$-e_4 +\frac{1}{2}\sum_{i=1}^5 e_i +\frac{1}{4}\bar{t}$ & $(\bar{z} - \bar{t}) \prod_{j=1}^3(\bar{z} -(-e_j +\frac{1}{2}\sum_{i=1}^5 e_i +\frac{1}{4}\bar{t}))$ \\
	$-e_5 +\frac{1}{2}\sum_{i=1}^5 e_i +\frac{1}{4}\bar{t}$ & $(\bar{z} - \bar{t}) \prod_{j=1}^4(\bar{z} -(-e_j +\frac{1}{2}\sum_{i=1}^5 e_i +\frac{1}{4}\bar{t}))$ \\ \hline
	$e_4+e_5-\frac{1}{2}\sum_{i=1}^5 e_i +\frac{1}{4}\bar{t}$ & $\omega$ \\
	$e_3+e_5-\frac{1}{2}\sum_{i=1}^5 e_i +\frac{1}{4}\bar{t}$ & $\omega(\bar{z}-(e_4+e_5-\frac{1}{2}\sum_{i=1}^5 e_i +\frac{1}{4}\bar{t}))$ \\
	$e_3+e_4-\frac{1}{2}\sum_{i=1}^5 e_i +\frac{1}{4}\bar{t}$ & $\omega \prod_{j=3}^4(\bar{z}-(e_j+e_5-\frac{1}{2}\sum_{i=1}^5 e_i +\frac{1}{4}\bar{t}))$ \\ \hline
	$e_5 -\frac{1}{2}\bar{t}$ & $\omega_6$ \\
	$e_4 -\frac{1}{2}\bar{t}$ & $\omega_6 (\bar{z} - (e_5 -\frac{1}{2}\bar{t}))$ \\
	$e_3 -\frac{1}{2}\bar{t}$ & $\omega_6 \prod_{j=4}^5(\bar{z} - (e_j -\frac{1}{2}\bar{t}))$ \\ \hline
	$e_2+e_5 -\frac{1}{2}\sum_{i=1}^5 e_i +\frac{1}{4}\bar{t}$ & $\omega_7$ \\
	$e_2+e_4 -\frac{1}{2}\sum_{i=1}^5 e_i +\frac{1}{4}\bar{t}$ & $\omega_7(\bar{z}-(e_2+e_5-\frac{1}{2}\sum_{i=1}^5 e_i +\frac{1}{4}\bar{t}))$ \\
	$e_2+e_3 -\frac{1}{2}\sum_{i=1}^5 e_i +\frac{1}{4}\bar{t}$ & $\omega_7 \prod_{j=4}^5(\bar{z}-(e_2+e_j-\frac{1}{2}\sum_{i=1}^5 e_i +\frac{1}{4}\bar{t}))$ \\
	$e_2 -\frac{1}{2}\bar{t}$ & $\omega_7 \prod_{j=3}^5(\bar{z}-(e_2+e_j-\frac{1}{2}\sum_{i=1}^5 e_i +\frac{1}{4}\bar{t}))$ \\
	$-e_1 -\frac{1}{2}\bar{t}$ & $\omega_7 (\bar{z} -(e_2 -\frac{1}{2}\bar{t}))\prod_{j=3}^5(\bar{z}-(e_2+e_j-\frac{1}{2}\sum_{i=1}^5 e_i +\frac{1}{4}\bar{t}))$ \\ \hline
	$-e_2 -\frac{1}{2}\bar{t}$ & $\omega_8$ \\
	$-e_3 -\frac{1}{2}\bar{t}$ & $\omega_8 (\bar{z} - (-e_2 -\frac{1}{2}\bar{t}))$ \\
	$-e_4 -\frac{1}{2}\bar{t}$ & $\omega_8 \prod_{j=2}^3(\bar{z} - (-e_j -\frac{1}{2}\bar{t}))$ \\
	$-e_5 -\frac{1}{2}\bar{t}$ & $\omega_8 \prod_{j=2}^4(\bar{z} - (-e_j -\frac{1}{2}\bar{t}))$ \\
	$-\frac{1}{2}\sum_{i=1}^5 e_i +\frac{1}{4}\bar{t}$ & $\omega_8 \prod_{j=2}^5(\bar{z} - (-e_j -\frac{1}{2}\bar{t}))$ \\ \hline
	$e_1+e_5-\frac{1}{2}\sum_{i=1}^5 e_i +\frac{1}{4}\bar{t}$ & $\omega_{12}$ \\
	$e_1+e_4-\frac{1}{2}\sum_{i=1}^5 e_i +\frac{1}{4}\bar{t}$ & $\omega_{12}(\bar{z}-(e_1+e_5-\frac{1}{2}\sum_{i=1}^5 e_i +\frac{1}{4}\bar{t}))$ \\
	$e_1+e_3-\frac{1}{2}\sum_{i=1}^5 e_i +\frac{1}{4}\bar{t}$ & $\omega_{12} \prod_{j=4}^5(\bar{z}-(e_1+e_j-\frac{1}{2}\sum_{i=1}^5 e_i +\frac{1}{4}\bar{t}))$ \\
	$e_1+e_2-\frac{1}{2}\sum_{i=1}^5 e_i +\frac{1}{4}\bar{t}$ & $\omega_{12} \prod_{j=3}^5(\bar{z}-(e_1+e_j-\frac{1}{2}\sum_{i=1}^5 e_i +\frac{1}{4}\bar{t}))$ \\
	$e_1 -\frac{1}{2}\bar{t}$ & $\omega_{12} \prod_{j=2}^5(\bar{z}-(e_1+e_j-\frac{1}{2}\sum_{i=1}^5 e_i +\frac{1}{4}\bar{t}))$ \\
  \end{tabular}
  }
\end{table}

\end{document}